\documentclass[reqno, 12pt]{amsart}

\usepackage{amsfonts, amsthm, amsmath, amssymb}
\usepackage{hyperref}
\hypersetup{colorlinks=false}

\usepackage{xcolor}

\usepackage{bbm}

\usepackage[margin=1in]{geometry}

\usepackage{helvet}

\RequirePackage{mathrsfs} \let\mathcal\mathscr

\usepackage{hyperref}
\usepackage{dsfont}
\usepackage{upgreek}
\usepackage{mathabx,yfonts}
\usepackage{enumerate}

\numberwithin{equation}{section}

\newtheorem{theorem}{Theorem}[section] 
\newtheorem{lemma}[theorem]{Lemma}
\newtheorem{proposition}[theorem]{Proposition}
\newtheorem{corollary}[theorem]{Corollary}

\newcommand{\lra}{{\longrightarrow}}

\theoremstyle{definition}
\newtheorem{example}[theorem]{Example}
 
\newtheorem{remark}[theorem]{Remark}
\newtheorem{definition}[theorem]{Definition}

\newcommand{\m}{\mathbf{m}}

\renewcommand{\phi}{\varphi}

\newcommand{\PP}{\mathbb{P}}
\renewcommand{\AA}{\mathbb{A}}

\renewcommand{\leq}{\leqslant}

\renewcommand{\geq}{\geqslant}
\renewcommand{\ge}{\geqslant}

\newcommand{\x}{\mathbf{x}}

\renewcommand{\c}{\mathbf{c}}

\newcommand{\z}{\mathbf{z}}

\renewcommand{\b}{\mathbf{b}}

\renewcommand{\r}{\mathbf{r}}

\DeclareMathOperator{\moo}{mod} 
\renewcommand{\bmod}[1]{\,(\moo{#1})}

\let\emptyset\varnothing

\DeclareSymbolFont{bbold}{U}{bbold}{m}{n}
\DeclareSymbolFontAlphabet{\mathbbold}{bbold}

\newcommand{\md}[1]{  \left(\textnormal{mod}\ #1\right)}
 
\renewcommand{\P}{\mathbb{P}}
\newcommand{\A}{\mathbb{A}}
\newcommand{\Q}{\mathbb{Q}}
\newcommand{\F}{\mathbb{F}}
\newcommand{\N}{\mathbb{N}}
\newcommand{\R}{\mathbb{R}}

\newcommand{\Z}{\mathbb{Z}}
\renewcommand{\l}{\left}
\renewcommand{\r}{\right}
\renewcommand{\b}{\mathbf}
\renewcommand{\c}{\mathcal}
\renewcommand{\epsilon}{\varepsilon}

\renewcommand{\leq}{\leqslant}
\renewcommand{\geq}{\geqslant}
\renewcommand{\#}{\sharp}

\DeclareMathOperator*{\Osum}{\sum{}^*}

\newcommand{\beq}[2]
{
\begin{equation}
\label{#1}
{#2}
\end{equation}
}

\title{Schinzel Hypothesis on average  
and rational points}

\author{Alexei N. Skorobogatov} 
\address{Department of Mathematics\\
South Kensington Campus\\
Imperial College London\\
SW7~2AZ United Kingdom \ --- \ and \ --- \ 
Institute for the Information Transmission Problems\\
Russian Academy of Sciences\\
Moscow, 127994 Russia}
\email{a.skorobogatov@imperial.ac.uk}

 \author{Efthymios Sofos} 
\address{
Department of Mathematics\\
University of Glasgow,
University Place,
 Glasgow,  G12~8QQ United Kingdom}
\email{efthymios.sofos@glasgow.ac.uk}

 \subjclass[2010]
{
11N32, 
14G05. 
}

\date{\today}
 
\begin{document}

\begin{abstract}
We resolve  
Schinzel's Hypothesis (H) for
$100\%$ of  
polynomials   
of arbitrary degrees.
We deduce that
a positive proportion of diagonal conic bundles over $\Q$ with any given number of degenerate
fibres have a rational point, and obtain similar results for generalised Ch\^atelet equations.

\end{abstract}

\maketitle

\setcounter{tocdepth}{1}
\tableofcontents

 \vspace{-1cm}

\section{Introduction}   
\label{s:intro} 

Schinzel's Hypothesis (H) \cite{SS} has very strong implications for the 
local-to-global principles for rational points on conic bundles,
as demonstrated by Colliot-Th\'el\`ene and Sansuc in \cite{CS82}. 
There have been many subsequent developments and applications to more general varieties
by Serre, Colliot-Th\'el\`ene, Swinnerton-Dyer and others. 
We call $P(t)\in\Z[t]$ a {\em Bouniakowsky polynomial}    
if the leading coefficient of $P(t)$ is positive and for every
prime $\ell$ the reduction of $P(t)$ modulo $\ell$ is not a multiple of $t^\ell-t$. 
 It is not hard to prove that an explicit positive proportion of polynomials of 
given degree are Bouniakowsky polynomials (Corollary \ref{densitySch} below). A conjecture stated by Bouniakowsky in 1854
\cite[p.~328]{Bou}, now a particular case of
Schinzel's Hypothesis (H),  
says that if $P(t)$ is an irreducible Bouniakowsky polynomial,
then there are infinitely many natural numbers $n$ such that $P(n)$ is prime.
Bouniakowsky added this remark: ``{\em Il est \`a pr\'esumer que la d\'emonstration rigoureuse du th\'eor\`eme \'enonc\'e sur les {\em progressions
arithm\'etiques des ordres sup\'erieurs} conduirait, dans l'\'etat actuel de la th\'eorie des nombres, \`a des difficult\'es insurmontables ; n\'eanmoins, sa r\'ealit\'e ne peut pas \^etre r\'evoqu\'ee en doute}".

The inaccessibility of Schinzel's hypothesis and its quantitative
version, the Bateman--Horn conjecture \cite{BatH}, in degrees greater than 1 or for
more than one polynomial motivates a search for more accessible replacements. In the case of several multivariate
polynomials of degree 1 such a replacement is provided by work of Green, Tao and Ziegler in additive combinatorics
(see \cite{GTZ} and references there, and \cite{BMS, HSW, HW} for applications to rational points).

In this paper we study rational points on varieties in families, with the aim of proving
that a positive proportion of varieties in a given family have rational points.
To apply the method of Colliot-Th\'el\`ene and Sansuc in this situation,
one does not need the full
strength of Bouniakowsky's conjecture, namely that {\em every} irreducible 
Bouniakowsky polynomial
represents {\em infinitely many} primes: it is enough to know that 
{\em most} polynomials satisfying the obvious necessary condition represent {\em at least
one} prime. We propose the following replacement for Bouniakowsky's conjecture.
The {\em height} of a polynomial $P(t)\in\Z[t]$ is defined 
as the maximum of the absolute values of its coefficients.

\begin{theorem} \label{A} Let $d$ be a positive integer. When ordered by height, for 
$100\%$ of Bouniakowsky polynomials $P(t)$ of degree $d$ there exists
a natural number $m$ such that $P(m)$ is prime.
\end{theorem}
This improves on previous work of Filaseta \cite{MR945264}  who showed that a positive proportion of Bouniakowksy polynomials represent a prime.
Note that stating Schinzel's Hypothesis for infinitely many primes is trivially equivalent
to stating it for at least one prime \cite[p.~188]{SS}, 
but this is no longer so if we are only concerned with 100\% of polynomials.

Theorem \ref{A} is a particular case of a more general result for $n$ polynomials,
where certain congruence conditions are allowed.
We denote the height of $P(t)\in\Z[t]$ by $|P|$. The height of
an $n$-tuple of polynomials $\b P=(P_1(t), \ldots, P_n(t))\in(\Z[t])^n$
is defined as $|\b P|=\max_{i=1,\ldots,n}(|P_i|)$.
We call $\b P$  a {\em Schinzel $n$-tuple} if for every prime $\ell$
the reduction modulo $\ell$ of the product $P_1(t)\ldots P_n(t)$ is not divisible by $t^\ell-t$,
and the leading coefficient of each $P_i(t)$ is positive.

\begin{theorem}
\label{cor:almostal}
Let $d_1, \ldots, d_n$ be positive integers. Fix integers $n_0$ and $M$.
Assume we are given  
$Q_1(t),\ldots,Q_n(t)$ in $\Z[t]$  
such that $\prod_{i=1}^n Q_i(n_0)$ and $M$ are coprime, and 
$\deg(Q_i(t))\leq d_i$ for $i=1,\ldots, n$.
When ordered by height, for 
$100\%$ of Schinzel $n$-tuples $(P_1(t),\ldots, P_n(t))$
such that $\deg(P_i(t))=d_i$ and $P_i(t)-Q_i(t)\in M\Z[t]$ for each $i=1,\ldots,n$,
there exists a natural number
$m \equiv  n_0  \md{M}$ such that 
 $P_1(m),\ldots,P_n(m)$ are pairwise different primes.
 \end{theorem}

The special case $M=1$ shows that, with probability $100\%$, an $n$-tuple of 
integer polynomials satisfying the necessary local conditions
simultaneously represent primes. Theorem \ref{A} is the special case for $n=1$.
The proof of Theorem \ref{cor:almostal}
occupies most of the paper; we give more details about the strategy of proof later in this introduction.

\medskip

In this paper we apply our analytic results to rational points on
varieties in families, where the parameter space is the space of coefficients of 
generic polynomials of fixed degrees.
Among many potential applications we choose to consider generalised Ch\^atelet
varieties (\ref{norm-eq}) and diagonal conic bundles (\ref{cb}).
Using Theorem \ref{cor:almostal} we obtain a weaker version of
the Hasse principle for equations
\begin{equation}
{\rm N}_{K/\Q}(\z)=P(t)\neq 0, \label{norm-eq}
\end{equation}
where $K$ is a fixed cyclic extension of $\Q$ and ${\rm N}_{K/\Q}(\z)$
is the associated norm form, for 100\% of Bouniakowsky polynomials $P(t)$ of given degree,
see Theorem \ref{thm1}.
(See also Theorem \ref{thm2} for the case when $P(t)$ is a product of generic
Bouniakowsky polynomials.) 
It implies

\begin{theorem} \label{B} Let $d$ be a positive integer. 
For a positive proportion of polynomials $P(t)\in\Z[t]$ of degree $d$ ordered by height,
the affine variety given by {\rm (\ref{norm-eq})} has a $\Q$-point.
\end{theorem}
Explicit estimates in the case $K=\Q(\sqrt{-1})$ are given in Section~\ref{s:newapp}.
If $K$ is a totally imaginary abelian extension of $\Q$ of class number 1, then
the same statement holds, with the following easy proof. By the Kronecker--Weber theorem
we have $K\subset\Q(\zeta_M)$ for some $M\geq 1$.
Hence all primes in the arithmetic progression $1\bmod M$ split in $K$. 
Theorem \ref{cor:almostal} implies that a random
Bouniakowsky polynomial of degree $d$ congruent to the constant polynomial $1$ modulo $M$
represents a prime. This prime $p$ is the norm of a principal integral ideal $(x)\subset K$.
Since $K$ is totally imaginary, we have $p={\rm N}_{K/\Q}(x)$. 
(See Theorem \ref{easy} for a more general statement.) Here, at the expense of 
the condition on the class number of $K$, we do not require $K$ to be 
cyclic over $\Q$ and we find an integral (and not just rational) solution of (\ref{norm-eq}).

A stronger version of Theorem \ref{cor:almostal}, where we require primes represented
by polynomials to satisfy additional conditions in terms of quadratic residues, allows us to incorporate
into our technique an estimate for certain character sums due to 
Heath-Brown \cite[Cor.~4]{MR1347489}. This leads to
the following result, proved in \S \ref{ppp} as a consequence of Theorem \ref{thm3}.

\begin{theorem} \label{TTT}
Let $n_1, n_2, n_3 $ be integers
such that $n_1>0$, $n_2>0$, and $n_3\geq 0$, and let $n=n_1+n_2+n_3$. Let 
$a_1, a_2, a_3$ be non-zero integers, and let $d_{ij}$ be natural numbers
for $i=1,2,3$ and $j=1,\ldots,n_i$. Then
for a positive proportion of $n$-tuples $(P_{ij})\in\Z[t]^n$ 
with $\deg(P_{ij}(t))=d_{ij}$, 
ordered by height, the following conic bundle surface has a $\Q$-point contained in a smooth fibre:
\beq
{cb}
{\hspace{-0,1cm}
a_1 \prod_{j=1}^{n_1}P_{1,j}(t)\,x^2+a_2 \prod_{k=1}^{n_2}P_{2,k}(t)\,y^2
+a_3 \prod_{l=1}^{n_3}P_{3,l}(t)\,z^2=0.}
\end{theorem}
By \cite[Thm.~1.4]{BBL} (see also \cite[Thm.~1.3]{LS}) 
in a dominant, everywhere locally solvable family of quasi-projective varieties
over an affine space such that the fibres at the points of codimension $1$ are split and enough 
real fibres have real points, a positive proportion of 
rational fibres are everywhere locally solvable. Thus, the results of Theorems \ref{B} and \ref{TTT} 
are expected consequences of
a conjecture of Colliot-Th\'el\`ene which predicts that the Hasse principle for
rational points on smooth, projective, geometrically rational varieties is controlled by the Brauer--Manin obstruction,
and generic triviality of the Brauer group in our families.
(Note that in these cases Colliot-Th\'el\`ene's conjecture follows from Schinzel's Hypothesis (H),
see \cite[Thm.~14.2.4]{CTS21}.)
A known non-trivial case of this conjecture for conic bundles
(\ref{cb}) is when the total degrees of coefficients are $(2,2,0)$; natural smooth projective models of
such surfaces are del Pezzo surfaces of degree 4 for which the result is due to
Colliot-Th\'el\`ene \cite{CT90}. 
The question is open already in
the case of total degrees $(2,2,2)$, which
corresponds to a particular kind of del Pezzo surfaces of degree $2$ (cf.~\cite[Prop.~5.2]{BMS}).
The conjecture for smooth projective varieties birationally
equivalent to (\ref{norm-eq}) is known when
$\deg(P(t))\leq 4$ (and in some cases when $\deg(P(t))=6$) and $[K:\Q]=2$ 
(Colliot-Th\'el\`ene, Sansuc and Swinnerton-Dyer \cite{CTSS87}, \cite{SD}, see
\cite[\S 7.2, \S 7.4]{Sk01}),
$\deg(P(t))\leq 3$ and $[K:\Q]=3$ (Colliot-Th\'el\`ene and Salberger \cite{CSal89}),
$\deg(P(t))\leq 2$ and $[K:\Q]$ arbitrary \cite{HBS02, CTHS03, BH12, DSW12}.
 There seem to be no known unconditional results about the Hasse principle
when the number of degenerate fibres is greater than 6. 
In contrast, for our statistical approach to the existence of rational points 
the number of degenerate fibres is immaterial.

\medskip

In the rest of the introduction we give more details about our main analytic results;
for this we need to introduce some more notation.
We write $P>0$ to denote that the leading coefficient of $P(t)$ is positive.
For a  prime $\ell$ and a polynomial $P(t) \in \F_\ell[t]$ we define
\[
Z_P(\ell )
:=\#
\l
\{ s \in \F_\ell: P(s)=0\r
\}.
 \]
In particular, $\b P$ is a Schinzel $n$-tuple if and only if 
$Z_{P_1\ldots P_n}(\ell)\neq\ell$ 
for all primes $\ell$
and $P_i>0$ for each $i=1,\ldots,n$. 
Fix integers $n_0$ and $M$, and polynomials
$Q_i(t)\in\Z[t]$ of degree at most $d_i$  for $i=1,\ldots, n$
such that $\prod_{i=1}^n Q_i(n_0)$ and $M$ are coprime.
For $H \geq 1$ define
\[
\texttt{Poly}(H):=
\l\{\b P \in (\Z[t])^n :\,
|\b P|\leq H,
\deg(P_i)=d_i,
P_i>0,
P_i \equiv Q_i \md{M}
\text{ for } i=1,\ldots,n\r\}.\]  

\subsection*{The least prime represented by a polynomial}
For   $ C>0$ define
\[
S_C(\b P):=
\{m \in \N : m\leq(\log |\b P|)^C, m\equiv n_0 \md{M}, P_i(m) \text{\ is prime for}\ i=1,\ldots, n \}
.\] 
Theorem \ref{cor:almostal} is an immediate consequence of the following more precise quantitative result.
\begin{theorem}
\label{cool}
Fix $A>0$. In the assumptions of Theorem \ref{cor:almostal}
for all $H\geq 3 $ we have  {\rm
\begin{equation} 
\hspace{-0,2cm}
 \frac{
\#\{
\b P\in \text{\texttt{Poly}}(H)
: \b P \text{\rm \ is Schinzel}, 
\, \#S_{n+A}(\b P) \geq (\log |\b P|)^{A/3}  
\}}
{
\#\{
\b P\in \text{\texttt{Poly}}(H)
: \b P \text{\rm \  is Schinzel}
\}}
\!=\!1+O\!\l(\frac{(\log \log \log H )^{d-n}}{ \sqrt{\log \log H } }
\r),   \label{eq:mahlersymphny5}  
\end{equation} }
where $d=d_1+\ldots+d_n$. The implied constant depends on $d$, $A$ and~$M$, but not on $H$.
\end{theorem}
Recall that Linnik's constant is the smallest   $L>0$ such that 
every primitive degree $1$ polynomial $P(x)=qx +a $
with $0<a<q$
represents a prime of size $\ll q^L=|P|^L$.
This subject has rich history, see~\cite[\S 18]{iwa}, for example. 
GRH implies that 
$L\leq 2+\epsilon $ for every $\epsilon>0$
and it is known that   $L\leq 5$, see~\cite{MR3086819}.
Furthermore, one cannot have $L<1$, see~\cite{MR3698292} for  accurate 
lower bounds.
 Theorem~\ref{cool}
shows that the analogue of the Linnik constant for polynomials of given degree is
at most $1+\epsilon$ for every $\epsilon>0$.

\begin{corollary}
\label{cor:linikcosnt}
Let $\epsilon>0$ and fix $d, n_0 , M \in \N$.
For 100\% of 
Bouniakowsky polynomials 
$P$ of  degree $d$ with $\gcd(P(n_0), M)=1$,
there exists a natural number
$  m \leq (\log |P|)^{1+\epsilon}$ such that 
$m\equiv n_0 \md M$
and $P(m)$ is a prime bounded by  
$ |P| (\log |P|)^{d+\epsilon}$.
\end{corollary}
Indeed,  
Theorem \ref{cool} with $n=1$ and $A=\epsilon/(2d)$
shows the existence of a natural number
$m\leq (\log |P|)^{1+\epsilon/(2d)}$
such that $P(m)$ is prime; furthermore,  we have 
 \[P(m) \leq (d+1) |P| m^d \leq (d+1) |P| (\log |P|)^{(1+\epsilon/(2d) )d}
\ll |P| (\log |P|)^{d+\epsilon/2}
\leq 
 |P| (\log |P|)^{d+\epsilon}
 .\]  
These bounds are intimately related to the efficacy of algorithms for factorisation of polynomials, see the work of 
Adleman and Odlyzko \cite{MR717715}, and for finding efficient cryptographic parameters as in the work of Freeman, Scott and Teske \cite [\S 2.1]{MR2578668}. McCurley \cite{MR854146}   has shown that for   certain polynomials the least representable prime has to be rather large. The 
case $d=2 $ of Corollary \ref{cor:linikcosnt}  is closely related to hard questions on the size of class numbers that goes all the way back to Euler; see the survey of Mollin \cite{MR1453656}. 

\subsection*{Smallest height of a rational point}
Bounding the least height of a $\Q$-point on a variety $V$ over $\Q$  
 is a hard problem whose solution implies Hilbert's 10th Problem for $\Q$.  
Amongst the Fano varieties it is only for quadrics 
that the known bound is essentially best possible, which is due to
Cassels \cite{MR69217}. Tschinkel 
gave a  conjecture for the size of the smallest  $\Q$-point \cite[Section 4.16]{MR2498064}.
In this direction we have the following result.
 
\begin{corollary}\label{cor:leastpoint} 
Let $\epsilon>0$, $a\in \Z$, $a\neq 0$, and $d\in \N$.
For a positive proportion of polynomials  $P(t)\in\Z[t]$ of degree $d$, the equation $x^2 -a y^2 =P(t)z^2 $
has a solution $(x, y, z, t) \in \N^4$ with  \[\max\{x,y, z, t\} 
 \leq |a|^{1/2} |P|^{1/2} (\log |P|)^{d/2+\epsilon }
 . \] \end{corollary}

To prove this we first note that 
the density of Bouniakowsky polynomials $P(t)$ of degree $d$
with $P(t)\equiv 1 \md{8a}$ exists and is positive; this is a special case of  Corollary~\ref{densitySch01234}.
Since these $P(t)$   satisfy $\gcd(P(0),8a)=1$, we use    
Corollary \ref {cor:linikcosnt} with 
$n_0=0 $ and $M=8a $ to see that 
for $100\%$  of Bouniakowsky polynomials $P(t)$ of degree $d$
with $P(t)\equiv 1 \md{8a}$ 
there exists a natural number $m\leq (\log |P|)^{1+\epsilon } $ such that
$P(m)$ is
a prime $p$ satisfying $p\leq |P|(\log |P|)^{d+\epsilon}$
and 
$p\equiv P(0)\equiv 1\md {8a}$. Holzer's theorem~\cite{MR35783} states that if  $f_1,f_2,f_3$ 
are square-free pairwise coprime integers, not all of the same sign
and such that $-f_i f_j $ is a quadratic residue modulo $f_k$ for all permutations $\{i,j,k\}=\{1,2,3\}$,
then there exists $(x_1,x_2,x_3) \in \Z^3\setminus\{(0,0,0) \}$ such that 
$\sum_{i=1}^3 f_i x_i^2=0$ and $|x_i|\leq \sqrt{|f_j f_k |}$. 
Writing $a=a_0 b^2$, where $a_0$ is square-free,
 we can apply Holzer's theorem for $f_1=-1, f_2=a_0, f_3=p$. Indeed, if 
$a_0=s2^\pi w$, where $s\in\{\pm1\}$, $\pi \in \{0,1\}$, and $w $ is a positive odd integer, then
the quadratic Jacobi symbols satisfy $$ \bigg(\frac{a_0}{p}\bigg)=
\bigg(\frac{w}{p}\bigg)=\bigg(\frac{p}{w}\bigg)=1,$$ due to $p\equiv 1 \md 8$ and
$p\equiv 1 \md  w $. Thus
$a_0$ is a square modulo $p$. Clearly,
$p$ is a square modulo $a_0$. By Holzer's theorem 
the equation $x^2-a_0 y^2= p z^2$ has a non-zero integer solution $( x_0,  y_0,  z_0 )$
with $\max\{|x_0|, |y_0|, |z_0|\} \leq  (|a_0 | p)^{1/2}$. Then 
 $(x_1,y_1,z_1)=(b x_0, y_0, b z_0 )$ is a non-zero solution of 
$x^2-a y^2= p z^2$ that satisfies $$\max\{| x_1|, | y_1|, | z_1|\} \leq b (|a_0 | p)^{1/2}=  (| a | p)^{1/2}
  \leq |a|^{1/2} |P|^{1/2} (\log |P|)^{d/2+\epsilon }.$$

\subsection*{The Bateman--Horn conjecture}
Theorem \ref{cool} is a corollary of Theorem \ref{thm:almostal} below.
To state it we introduce a prime counting function and a truncated singular series.

\begin{definition}  \label{def:really} 
Let $\b P \in (\Z[t])^n$, let $n_0 \in \Z$, and let $M \in \N$. For $x\geq 1 $ 
define the functions 
\beq
{def:thetP}
{
 \theta_\b P (x )=
\sum_{\substack{ m \in \N \cap [1,x]  \\ 
m\equiv n_0 \md{M}
\\
P_i(m ) \text{ prime for}\, i=1,\ldots,n }}
\prod_{i=1}^n
\log  P_i(m ),
}
\beq
{eq:bachbibaldi}
{
\mathfrak{S}_{\b P}(x)=
\frac{\mathds 1 (\gcd(M, \prod_{i=1}^n P_i(n_0) )  =1)}
{\phi(M)^{n}  }M^{n-1}
\prod_{\substack{ \ell \text{ prime}, \,  \ell \nmid M \\ \ell \leq  \log x  }}  \frac {1-\ell^{-1}Z_{P_1 \ldots P_n } (\ell )} {\l(1-\ell^{-1}\r)^n }
.}
\end{definition} 
The function $\mathfrak{S}_{\b P}(x)$ is a truncated version of the 
Hardy--Littlewood singular series associated to Schinzel's
Hypothesis for the polynomials
 $P_1(n_0+M t), \ldots, P_n(n_0+M t)$, 
see~\cite{BatH}.
The reason for considering 
 $P_i(n_0+Mt )$ instead of $P_i(t)$
is because     $\theta_{\b P}(x)$ involves the  condition
 $m\equiv n_0 \md{M}$. 
A standard argument based on the prime number theorem for number fields  shows that  for a fixed  $\b P$ the product  $\mathfrak S_\b P(x)$  converges 
as $x\to \infty$. However, the convergence is absolute only when each $P_i$ is linear.  Since we treat general polynomials, we have chosen to 
work with the truncated version  to avoid problems related to the lack of absolute convergence. 

The Bateman--Horn conjecture states that
$$\theta_\b P(x)- \mathfrak S_\b P(x) x= o(x).$$
Our next result
shows that  
the estimate 
 \[\theta_\b P(x)- \mathfrak S_\b P(x) x=
O\l(\frac{x}{\sqrt{ \log x } }\r)
 \] holds for $100\%$ of  $\b P\in (\Z[t])^n$ in a certain 
range for $x$. 
  Let  \[ \c R(x, H)= \frac{ 1}{ \#\texttt{Poly}(H)  } 
 \sum_{ \substack{  \b P \in  \texttt{Poly}(H)   } }  
\Big | \theta_\b P(x) -  \mathfrak S_\b P(x)  x\Big|
\] be   the average over all $n$-tuples $\b P$
of the error terms in
the Bateman--Horn conjecture.
 \begin{theorem}
\label
{thm:almostal}
Let $n, d_1, \ldots, d_n, M$ be positive integers. Let $n_0\in\Z$ and let $\b Q=(Q_i(t)) \in (\Z[t])^n$.
Fix arbitrary $A_1, A_2\in \R $ with $n<A_1<A_2$.
Then for all $H\geq 3  $
 and all $ x\geq 3 $ with 
\[
(\log H)^{A_1}<x\leq (\log H )^{A_2}
\]
we have
\[ \c R(x,H)\ll \frac{x}{ \sqrt{\log x } }  ,\]
where the implied constant depends only on $d_1, \ldots, d_n, M, n_0, \b Q, A_1, A_2$. 
\end{theorem}

The necessity of  $ A_1>n  $ is addressed in Remark \ref{rem:lowerbound}; one cannot expect 
  typical polynomials to represent primes when the input is not large compared to the coefficients, 
and $m\approx (\log |\b P | )^n$ seems to be a natural barrier.

From Theorem \ref{thm:almostal}
and Markov's inequality one immediately deduces a  form of the Bateman--Horn conjecture valid for almost all polynomials.
For simplicity we state this result only in the case  $n=M=n_0=1$.

\begin{corollary}
\label{cor:mark} Let $d $ be a positive integer. 
Fix any $c\in \R$ with  $0<c<1/2$ and any $A_1, A_2\in \R $ with $1<A_1<A_2$.
Then for all irreducible $P \in  \Z[t] $ with $\deg(P)=d $ and   all $ x  $ with 
$
(\log |  P| )^{A_1}<x\leq (\log |   P| )^{A_2}
$
we have
\[ \sum_{\substack{ m\in \N \cap [1,x] \\ P(m) \textrm{\rm \ prime}  } } \log P(m) = 
\l( \prod_{\substack{ \ell \text{\rm \ prime} \\ \ell \leq  \log x  }}  \frac {1-\ell^{-1}Z_{P  } (\ell )} { 1-\ell^{-1}  }
\r)   x  +O\l( \frac{x}{ (\log x )^{c} } \r) ,\] with the exception of at most   $O(H^{d+1 } (\log  \log H)^{c-1/2})$ of polynomials   $P$ such that $|P|\leq  H$. 
\end{corollary} The asymptotic is meaningful, since $\mathfrak S_  P(x) \gg (\log \log x)^{1-d  }$ if $\mathfrak S_P(x) \neq 0 $, see Lemma \ref{lem:beta0}.  
 
\subsection*{Comparison with the literature} 
Our main result, Theorem \ref{thm:almostal}, is a vast generalisation of the well-known 
Barban--Davenport--Halberstam theorem on primes in arithmetic progressions,
which gives a bound on 
\[
\sum_{\substack{  
1\leq q \leq  Q   \\ a \in (\Z/q\Z )^* 
 }  } \left( \sum_{\substack{{\rm prime} \, p\leq X \\ p\equiv  a \md{q}   } } \log p  - \frac{X}{\phi(q)} \right)^2
.\] To bring it to a form comparable to
  Theorem \ref{thm:almostal} 
we write   $H=Q$,
$x=X/Q$  and   
 $P(t) =a +q t $, from which  it   becomes  evident that 
 the left hand side is essentially equal to 
\[
\sum_{ \substack{ P \in \Z[t]: \ \deg(P)=1  \\ | P | \leq   H        }}
 \left( \sum_{\substack{m \leq x \\  P(m) \text{ prime }  } } \log P(m)   - \mathfrak S_P(x)  x \right)^2 
.\]
While the Barban--Davenport--Halberstam theorem concerns a single 
linear polynomial, our work covers an arbitrary number of polynomials, each of arbitrary degree.
Prior to our paper there has been a 
number of results  on averaged forms of Bateman--Horn  for  special polynomials.

\begin{center}
    \begin{tabular}{| l | l | l | l |}
    \hline
    $n$ &     $  P_1(t), \ldots,  P_n(t)  $  & Authors \\ \hline
$\geq 1 $   & $  t + b_1, \ldots,  t + b_n   $ & Lavrik \cite{lavrik} \\ \hline
   $2$   & $t,t+b$ &  Lavrik \cite{MR0159801},    Mikawa \cite{MR1118579}, Wolke \cite{MR990587}  \\ \hline
       $1$   & $a t+ b $ & Barban\cite{barb}, Davenport--Halberstam \cite{davenhalb}        \\ \hline
     $\geq 1 $   & $ a_1 t + b_1, \ldots, a_n t + b_n   $ & Balog \cite {MR1084173} \\ \hline
    $1$   & $t^d+a t + b$  &  Friedlander--Granville   \cite{MR1133852} \\ \hline
    $1$   & $t^2+t + b$ and $t^2+b$&  Granville--Mollin \cite{MR1814449} \\ \hline
    $1$   & $t^2+b $ & Baier--Zhao \cite{MR2317757, MR2569742}        \\ \hline
     $1$   &  $t^3+b $  &  Foo--Zhao \cite{primecube} \\ \hline
 $1$   & $t^4+b $ & Yau \cite{primequarti}   \\ \hline
    $1$   & $t^d+b $ &  Zhou \cite{MR3831402}  \\ \hline
      \end{tabular}
\end{center}    
The work of 
Friedlander--Granville \cite{MR1133852}
has special interest in connection to our work 
as it shows that there are unexpectedly large fluctuations in the error term of the 
 Bateman--Horn asymptotic; it would be interesting to understand analogous 
questions  in the  setting of  Corollary \ref{cor:mark}.
Furthermore, it would be interesting 
to investigate the case where one ranges over degree $d$
polynomials with a fixed coefficient; this corresponds to work of Friedlander--Goldston \cite{friedgold}
where this is investigated for linear polynomials with fixed leading coefficient.
 \subsection*{Method of proof} 
Theorem \ref{thm:almostal} 
is a generalisation of Montgomery's proof of 
the Barban--Davenport--Halberstam theorem, which corresponds to the case $n=1 $ and $d_1 =1 $ of Theorem \ref{thm:almostal}.
By Cauchy--Schwarz  we have 
\beq{eq:aha2}{\c R(x,H)^2\leq \c V(x, H):= \frac{ 1}{ \#\texttt{Poly}(H)}\sum_{ \substack{  \b P \in  \texttt{Poly}(H)   } }  
\l( \theta_\b P(x) -  \mathfrak S_\b P(x)  x\r)^2 , } which is the kind of second moment function studied in the BDH theorem.   
The original proof of the BDH theorem  is a direct application of the large sieve; such an approach  only  applies 
to     polynomials of very special shape, see \cite{MR2317757,primecube}.
The initial arguments in our paper are in fact  closer to     Montgomery's proof of 
the BDH theorem \cite{montg}, which does not rely on the large sieve.

First, we open up the square in $\c V(x, H)$ to get 
three terms: the second moments $\theta_\b P(x)^2$ and $x^2\mathfrak S_\b P(x)^2$, 
and the correlation $x\mathfrak S_\b P(x)\theta_\b P(x)$. 
The hardest term is $\theta_\b P(x)^2$ and here Montgomery's approach  relies exclusively on 
Lavrik's result on twin primes \cite{MR0159801, lavrik}. Lavrik's argument makes heavy use 
of      the Hardy--Littlewood  circle method and   Vinogradov's estimates of exponential sums.
In our work we need a suitable generalisation of   Lavrik's result; this is provided by our   Theorem~\ref{thm:mainresulthalb}.
It produces an asymptotic for simultaneous prime values of two linear polynomials in an arbitrary number of variables,
where the error term is uniform in the size of the coefficients.
The  difference between our work and that of  Montgomery and Lavrik
  is that to prove Theorem~\ref{thm:mainresulthalb}
we  do not use the circle method 
and we instead  employ the
\textit{M\"obius randomness law}, see Section \ref{s:mob}. 
This approach in the area of the averaged Bateman--Horn conjecture is new.

Next, we show that
the three principal terms
cancel out by constructing a probability space that models the behaviour of functions 
involving $Z$, see Section \ref{s:probabilisticmodel}. 
This task inevitably leads to new
complications of combinatorial nature, compared to the   aforementioned papers on special
polynomials  where   the  Bateman--Horn singular series has a useful expression in terms of $L$-functions (see \cite{MR2317757,primecube}, for example).
The final stages of the proof of Theorem~\ref{thm:almostal} 
can be found in \S \ref{s:asintheothersection} and that of Theorem \ref{cool} in \S\ref{s:kimamaior8ios}.  

Applications to rational points, including the proofs of Theorems \ref{B} and \ref{TTT}, 
can be found in Sections \ref{rat} and \ref{rat2}.

\subsection*{Notation}
The quantities $A_1, A_2 , \delta_1, \delta_2 , n, d_1,\ldots, d_n, \b Q, n_0, M, $
will be considered constant throughout. In particular, the dependence of 
 implied constants in the big $O$ notation on these 
quantities will not be recorded. 
Any other dependencies of the implied constants on further parameters will be explicitly specified via the use of a subscript.
Whenever we use iterated logarithm functions $\log t, \log \log t$, etc., we assume 
that $t$ is   large enough  to make  the iterated logarithm well-defined. 

\subsection*{Acknowledgement}
The work on this paper started during the research trimester 
``\`A  la red\' ecouverte des points rationnels'' at
the Institut Henri Poincar\'e in Paris, whose support is gratefully acknowledged.
The authors would like to thank Andrew Granville for his interest in this paper, his enthusiasm
and useful discussions. 
We are very grateful to the referee for careful reading of the paper and many helpful comments.

\section{Bernoulli models   of    Euler factors } \label{s:probabilisticmodel}

In this section we study the $\ell$-factor  $1-\ell^{-1}Z_{P_1 \ldots P_n } (\ell)$
of the Euler product~\eqref{eq:bachbibaldi}.
We prove that if $P_1, \ldots, P_n$ are random
polynomials of bounded degree in $\F_\ell[t]$,
this factor is modelled by the arithmetic mean of $\ell$ 
pairwise independent, identically distributed Bernoulli random variables 
defined on a product of probability spaces.   
The results of this section are 
used in Section~\ref{e:estimatingthevar} to prove cancellation of principal terms.
Proposition~\ref{prop:betagammadelta} is used to prove 
Theorem \ref{cool} in  \S\ref{s:kimamaior8ios}.

\subsection{Bernoulli model}
\label
{s:bernoulli} Let $\ell$ be a prime.
Consider the probability space $(\Omega(d),\PP)$, where
\[\Omega(d):=\{
P \in \F_\ell[t]: \deg(P)\leq d\}\]
and $\PP$ is the uniform discrete probability. For every $m \in \F_\ell$
we define the Bernoulli random variable 
$Y_m:\Omega(d)\to \{0,1\}$ by
\[Y_m = \begin{cases}  1, &\mbox{if } P(m) \neq  0 \text{ in } \F_\ell, \\ 
0, & \mbox{otherwise. }   \end{cases}  \]
It is clear that $Y_m=\chi(P(m))$, where
$\chi$ is the principal Dirichlet character on $\F_\ell$.

\begin{lemma} \label{april}
Let $\c J\subset\F_\ell $ be a subset
of cardinality $s\leq d+1$. Then the variables $Y_m$ for $m\in \c J$
are independent, and we have \[
\mathbb E_{\Omega(d)}\prod_{m\in \c J} Y_m
=\prod_{m\in \c J} 
\mathbb E_{\Omega(d)} Y_m=(1-\ell^{-1})^s.\]  
\end{lemma}
\begin{proof}
It is enough to prove that
\begin{equation} \label{8apr}
\mathbb E_{\Omega(d)}\prod_{m\in \c J} (1-Y_m)
=\frac{1}{\ell^{d+1}}\,\#\l\{P\in \F_\ell[t]: \deg(P)\leq d, P(m)=0\ 
\text{\rm if}\ m \in \c J\r\}=\frac{1}{\ell^s}.
\end{equation}
By the non-vanishing of the Vandermonde determinant this condition 
describes an $\F_\ell$-vector subspace of $\Omega(d)$ of codimension $s$,
hence the result.
\end{proof}

Let $n \in \N$ and let $d_1, \ldots, d_n \in \N$. 
Consider $\Omega=\Omega(d_1) \times \ldots \times\Omega(d_n)$ as a
Cartesian probability space 
equipped with the product measure 
\beq
{eq:mumumu}
{
\PP(A_1 \times \ldots \times  A_n ):=
\PP_1(A_1) \ldots  \PP_n(A_n),
\
\text{ for all } 
\
A_i \subseteq \Omega(d_i),
}
where each $\PP_i$ is the uniform discrete probability on $\Omega(d_i)$.
For $m \in \F_\ell$ define the Bernoulli random variable 
$X_m:\Omega\to \{0,1\}$
 by
\[X_m = \begin{cases}  1, &\mbox{if } \prod_{i=1}^n P_i(m) \neq  0 \text{ in } \F_\ell, \\ 
0, & \mbox{otherwise. }   \end{cases}  \]
It is clear that  
\beq{eq:oktelos}{X_1 +\ldots + X_\ell=\ell-
Z_{P_1 \ldots P_n } (\ell ).}

\begin
{lemma}
\label
{lem:thisissufficiehdfpantelonint}
For all $m \in \F_\ell $ we have $\mathbb E_{\Omega}X_m=(1-\ell^{-1})^n $. 
\end
{lemma}
\begin{proof}  This is immediate from Lemma \ref{april}.
\end{proof}
\begin
{lemma}
\label
{lem:thisissufficient}
For all $k\neq m \in \F_\ell$ the random variables $X_k $ and $X_m $ are independent. 
\end
{lemma}
\begin{proof} 
Since $X_k$ and $X_m $ are 
 Bernoulli random variables, it suffices to show that they are uncorrelated. 
Using Lemma \ref{lem:thisissufficiehdfpantelonint} we write 
the covariance of $X_k$ and $X_m $ as
\[
\mathbb E_{\Omega}
\l[
\l(\prod_{i=1}^n 
\chi(P_i(m) )
- \l(1-\ell^{-1}\r)^n
\r)
\l(
\prod_{j=1}^n 
\chi(P_j(k) ) 
- \l(1-\ell^{-1}\r)^n
\r)
\r]
,
\]
which equals
\[
\mathbb E_{\Omega}
\l[
  \prod_{i=1}^n 
\chi(P_i(m ) ) \chi(P_i(k) )
\r]
- \l(1-\ell^{-1}\r)^{2n}
\! \! \! =
 \l(  \prod_{i=1}^n 
\mathbb E_{\Omega(d_i)}
\l[
\chi(P(m ) )\chi(P(k) )
\r]
\r)
- \l(1-\ell^{-1}\r)^{2n}
\]
by~\eqref{eq:mumumu}. Since $d_i\geq 1$ for all $i=1,\ldots,n$,
we conclude the proof by applying Lemma \ref{april}.
\end{proof}

For $d, s \in \Z_{\geq 0 }$ define  
\beq{eq:cofeeroastartisan}   {G_\ell(d,s):=\sum_{r=0}^s {s\choose r} \frac{(-1)^r}{\ell^{\min\{r,1+d\}}}    .}

\begin{lemma}\label{lem:bachcellosuite5} 
For a subset $\c J \subset \F_\ell$ of cardinality $s$ we have 
\[ \mathbb E_{\Omega}\prod_{m\in \c J}X_m=\prod_{k=1}^n G_\ell(d_k,s).\] \end{lemma} 
\begin{proof}
By multiplicativity of the principal Dirichlet character $\chi$ we have
\[ \prod_{m\in \c J}X_m=\prod_{m\in \c J}\chi\l(\prod_{k=1}^n   P_k(m)\r)=
\prod_{k=1}^n  \chi\l(   \prod_{m\in \c J} P_k(m)\r),\] 
hence
\[\mathbb E_{\Omega}\prod_{m\in \c J}X_m=
\prod_{k=1}^n    \mathbb E_{\Omega(d_k)}\prod_{m\in \c J}\chi(P(m))  .\] 
For a fixed $k$ we have
\[ \mathbb E_{\Omega(d_k)}\prod_{m\in \c J}\chi(P(m)) =
\mathbb E_{\Omega(d_k)}\prod_{m\in \c J}Y_m =
\sum_{r=0}^{s} (-1)^{\#\c A}  \sum_{\c A \subset \c J}
 \mathbb E_{\Omega(d_k)}\prod_{m\in \c A}(1-Y_m).\] 
From the definition of the random variables $Y_m$ we get
\[
 \mathbb E_{\Omega(d_k)}\prod_{m\in \c A}(1-Y_m)
=\ell^{-(d_k+1)}\#\l\{P\in \F_\ell[t]: \deg(P)\leq d_k, P(m)= 0\ \text{\rm if}\ m\in \c A\r\}.\]
If $\#\c A\leq d_k+1$, this equals $\ell^{-\#\c A}$ by (\ref{8apr}).
If $\#\c A\geq d_k+1$, then $P$ has more than $\deg(P)$ roots in $\F_\ell$, 
hence $P$ is identically zero and the quantity above is $\ell^{-(d_k+1)}$. 
Thus
\[ \mathbb E_{\Omega(d_k)}\prod_{m\in \c A}(1-Y_m)
=\ell^{-\min\{\#\c A, d_k+1\}}.\] 
This implies the lemma.
\end{proof} 

\begin{lemma}[Joint distribution of Bernoulli variables] \label{lem:nductivevandermcellobach} 
For $\gamma_1, \ldots, \gamma_\ell \in \{0,1\}$ we have 
\[
\P\l[X_m=\gamma_m \text{\rm \ for all \ } m =1,\ldots, \ell\r]
=(-1)^{\#\{i: \gamma_i = 0 \}}
\sum_{\substack{ \c J\subset \F_\ell \\ i\not\in \c J \Rightarrow \gamma_i=0 }}
 (-1)^{\ell-\#\c J} \prod_{k=1}^n  G_\ell(d_k, \#\c J) 
 .\]
\end{lemma}  
\begin{proof} 
The event $X_m=\gamma_m$ for $\gamma_m=0$ (respectively, $\gamma_m=1$)
is detected by the function $1-X_m$ (respectively, $X_m$).
Therefore, writing $\beta_i=1-\gamma_i$ we obtain 
\[\P\l[X_m=\gamma_m \text{\rm \ for all \ } m =1,\ldots, \ell\r]
=(-1)^{\#\{i:\, \gamma_i = 0 \}}
\mathbb E_{\Omega}\prod_{m=1}^\ell (X_m-\beta_m ).\] 
The mean in the right hand side equals 
\[
\sum_{\c J \subset \F_\ell}
\l(\prod_{i\notin \c J} (-\beta_i) \r)
 \mathbb E_\Omega\prod_{i\in \c J} X_i 
=
\sum_{\c J \subset \F_\ell}(-1)^{\ell-\#\c J} \prod_{k=1}^n  G_\ell(d_k,\# \c J)
\prod_{i\notin \c J} \beta_i 
\] due to 
Lemma~\ref{lem:bachcellosuite5}. In view of $\beta_i\in\{0,1\}$ this proves the lemma.
\end{proof}
 \subsection{Consequences of the Bernoulli model}
\label{subs:conseq}

For $n \in \N$ and any prime $\ell  $
define
\beq
{eq:bratsakipsomaki}
{
\gamma_n(\ell  )
:=
 1-\frac{1}{\ell  }  +\frac{\ell^{n-1}}{(\ell -1)^n}
.}

\begin{lemma}
\label
{lem:corolhaha}
We have
 \[\ell^{-(d+n)}
  \sum_{\substack{ P_1 \in \F_\ell[t],\, \deg(P_1)
\leq 
d_1 \\  \ldots 
\\
P_n \in \F_\ell[t],\, \deg(P_n)
\leq 
d_n 
}}
\l(
1-\frac{Z_{P_1 \ldots P_n  } (\ell )}{\ell }
\r)^2
= 
\gamma_n(\ell  )
\l(  1-\frac{1}{\ell} \r)^{2n} 
 .\]
\end{lemma}
\begin{proof}
We write 
 the left hand side as  $\ell^{-2}  \mathbb E_{\Omega}[(X_1+\ldots+X_\ell)^2]$, open up the square and use 
Lemmas~\ref{lem:thisissufficiehdfpantelonint} and \ref{lem:thisissufficient}.
\end{proof}
By considering $\ell^{-1}   \mathbb E_{\b P\in \Omega}[X_1+\ldots +X_\ell ]$ instead
we obtain
$$
\ell^{-(d+n)}
 \sum_{\substack{ P_1 \in \F_\ell[t],\, \deg(P_1)
\leq 
d_1 \\  \ldots 
\\
P_n \in \F_\ell[t],\, \deg(P_n)
\leq 
d_n 
}}
\l(
1-\frac{Z_{P_1 \ldots P_n  } (\ell )}{\ell }
\r)
= 
\l(  1-\frac{1}{\ell} \r)^{n}.$$
\begin{lemma}
\label
{lem:corolhahparapoligelioorimaprosopa}
Fix any $m \in \N$.
We have
\[\ell^{-(d+n)}
 \sum_{\substack{ P_1 \in \F_\ell[t],\, \deg(P_1)
\leq 
d_1, P_1(m)\neq 0 
 \\  \ldots 
\\
P_n \in \F_\ell[t],\, \deg(P_n)
\leq 
d_n , P_n(m)\neq 0 
}}
\l(
1-\frac{Z_{P_1 \ldots P_n  } (\ell )}{\ell }
\r)
= \gamma_n(\ell  )
\l(1-\frac{1}{\ell} \r)^{2n}  .\]
\end{lemma}
\begin{proof}
By~\eqref{eq:oktelos} and Lemma~\ref{lem:thisissufficient}
the left hand side in our lemma equals 
\[
\mathbb E_{\Omega}\l[\l(\frac{X_1+\ldots +X_\ell }{\ell}\r)
 X_m   
\r ]=
\frac{\mathbb E_{\Omega}\l[ X_m    \r ]
 }{\ell}+
\frac{\mathbb E_{\Omega}\l[ X_m \r ]
}{\ell}\sum_{i\neq m }
\mathbb E_{\Omega}\l[ X_i  \r ]
.\]
The proof now concludes by using Lemma~\ref{lem:thisissufficiehdfpantelonint}.\end{proof}

\subsection{Density of Schinzel $n$-tuples}
\label
{s:densschinzpolys}
 
For a prime $\ell$ define the set 
$$ \texttt{T}_{\ell}  :=
\{\b P\in (\F_\ell[t])^n : Z_{P_1\ldots P_n }(\ell) \neq \ell, \ 
\deg(P_i) \leq d_i\ \text{for all $i=1,\ldots, n$}\}.$$
By Lemma~\ref{lem:nductivevandermcellobach}
with all $\gamma_i=0$ we have $\#\texttt{T}_{\ell}=(1-c_\ell )\ell^{d+n }$, where 
\beq{eq:se5lepta}
{c_\ell:= 
 \sum_{\substack{ \c J\subset \F_\ell   }}
 (-1)^{\#\c J}\prod_{k=1}^n  G_\ell(d_k, \#\c J) 
 .}
When $\ell>d$ it is easy to see that $\#\texttt{T}_\ell=\prod_{i=1}^n(\ell^{d_i+1}-1)$,
hence $1-c_\ell=\prod_{i=1}^n(1-\ell^{-(d_i+1)})$.
\begin{proposition} \label
{prop:betagammadelta}
For any $M\in\N$ we have {\rm $$
\#\{\b P\in \texttt{Poly}(H)
:  Z_{P_1\ldots P_n }(\ell) \neq \ell \ \text{\rm for all} \ \ell\nmid M \}
=
  2^{d }
\l(\prod_{{\rm prime}\,\ell \nmid M}
(1-c_\ell )
 \r)
\l(\frac{ H}{  M  } \r)^{d+n}
\!+O \l(\frac{H^{d+n}}{\log H}\r).$$ }
The infinite product converges absolutely to a positive real number. 
In particular, the set of Schinzel $n$-tuples of given degrees has positive density in 
the set of all $n$-tuples of integer polynomials of the same degrees.
 \end{proposition} 
\begin
{proof}
Let $\c W$ be the product of all primes $\ell<\frac{1}{10} \log H$ such that
$\ell\nmid M$. Define
$$K(H)= \#\l\{\b P\in \texttt{Poly}(H)
:Z_{P_1\ldots P_n }(\ell) \neq \ell \ \text{\rm for all primes} \ \ell|\c W \r\}.$$
The counting function in the lemma is 
$K(H)+O(H^{d+n}(\log H )^{-1}).$
Indeed, the number of $\b P\in \texttt{Poly}(H)$ such that for some $j=1,\ldots, n $
there is a prime $\ell>\frac{1}{10} \log H$ for which $P_j$ is identically zero on $\F_\ell$
is  
\[\ll \sum_{{\rm prime} \,\ell >\frac{1}{10}\log H} 
\l(\prod_{\substack{i=1\\i\neq j }}^n H^{1+d_i} \r)(H/\ell)^{1+d_j }
\ll H^{d+n } \sum_{{\rm prime} \,\ell >\frac{1}{10}\log H }   \ell^{-2} \ll H^{d+n} (\log H )^{-1}. \]
We have
\[K(H)=\sum_{ \b P\in \texttt{Poly}(H)}
\prod_{{\rm prime} \,\ell \mid \c W}
\mathds 1_{\texttt{T}_{\ell}}(\b P)=
2^{-n }
\l(\frac{2 H}{\c W M}+O(1)\r)^{d+n} \prod_{{\rm prime} \,\ell \mid \c W}
 \#\texttt{T}_{\ell},\]
by the Chinese remainder theorem applied to the coefficients of the polynomials $P_i$. 
Taking into account that $\#\texttt{T}_{\ell}=(1-c_\ell )\ell^{d+n }$
we rewrite this as
$$K(H)=2^{d }\l(\frac{ H}{  M  } +O(1)\c W\r)^{d+n}
\prod_{{\rm prime}\,\ell \mid \c W}(1-c_\ell).$$
Note that $\log \c W \leq \sum_{\ell \leq (\log H ) /10 }\log \ell \leq (\log H)/2$ for all sufficiently large $H$ by the prime number theorem.
Hence $\c W \leq H^{1/2}$, which implies
\[
K(H)= 2^{d }
\l(\frac{ H}{  M  } \r)^{n+d}
\prod_{{\rm prime}\,\ell \mid \c W} (1-c_\ell)+O(H^{d+n -1/2})
.\] The estimate $\prod_{{\rm prime}\,\ell >\frac{1}{10}\log H} \l(1-\ell^{-(d_i+1) }\r)
=1+O((\log H)^{-d_i} )$ concludes the proof.

The product converges absolutely because for all $\ell>d$ we have
$$1-c_\ell=\prod_{i=1}^n(1-\ell^{-(d_i+1)})=1+O(\ell^{-2}).$$
Since $\texttt{T}_{\ell}\neq\emptyset$ we have $\#\texttt{T}_{\ell}=(1-c_\ell )\ell^{d+n }> 0$,
so the infinite product  is positive.
\end{proof}  

\begin{corollary}
 \label{densitySch01234}
Fix $d, M \in \N$. Let $Q(t)\in \Z[t]$ be a polynomial of degree at most $d$.
The number of degree $d$ polynomials $f(t)\in\Z[t]$
with positive leading coefficient and height at most $H$
such that $f\equiv Q \md M$ and $Z_f(\ell)\neq \ell$ for each prime $\ell \nmid M$
is 
\[
2^d
\l(
\prod_{  \text{\rm prime } \ell \nmid M }
(1-\ell^{-\min\{\ell, d+1\}})
\r)
\frac{H^{d+1}}{M^{d+1} }
+O \l(\frac{H^{d+1}}{\log H}\r)
.\]  
\end{corollary}
\begin{proof}
We apply Proposition~\ref{prop:betagammadelta} in the case $n=1$.
For $\ell>d+1$ we have $c_\ell=\ell^{-(d+1)}$. 
If $s \leq d+1$ then~\eqref{eq:cofeeroastartisan}
becomes  $G_\ell(d,s)=(1-1/\ell)^s$. Hence for $\ell \leq d+1$, ~\eqref{eq:se5lepta} gives $c_\ell=\ell^{-\ell}$.
\end{proof}
The case $M=1$ of Corollary~\ref{densitySch01234}
is particularly useful and is worth recording separately:
\begin{corollary}
 \label{densitySch}
The number of degree $d$ Bouniakowsky polynomials of height at most $H$~is 
\[
2^d
\l(
\prod_{  \text{\rm prime } \ell }
(1-\ell^{-\min\{\ell, d+1\}})
\r)
H^{d+1}
+O \l(\frac{H^{d+1}}{\log H}\r)
.\]  
\end{corollary}

\section{M\"obius randomness law}
\label{s:mob} 
For any 
$d, k,m \in \N$  
and $H\geq 1 $
we   let  \beq
{def:bachist}{ \c G_{k,m }(H; d) :=
\sum_{\substack{   P\in \Z[t]   ,\, \deg(P)=d \\ | P   | \leq H  ,\, P>0   } }
\Lambda( P(k ) )
\Lambda( P(m ) )
,}
where $\Lambda(n)$ is the von Mangoldt function.
The main result of this section is the following asymptotic
for $ \c G_{k,m}(H; d)$ as $H\to \infty$
that exhibits an effective dependence on $k$ and $m $. 
\begin{theorem}
\label
{thm:mainresulthalb}
Fix any 
$d\in \N$ and 
$\delta>0 $.
Then for all 
$H \geq 1 $,  $A>0$, and all natural numbers $k, m\leq(\log H)^{\delta}$, $k\neq m$,
we have  
 \[ 
 \c G_{k,m}(H; d)
=
2^d H^{d+1}\prod_{\substack{p\rm{\, prime} 
\\ p\mid k-m 
}} 
\frac{p}{p-1} 
+O_A\l(
H^{d+1}
(\log H)^{-A}
\r)
 ,\] 
where the implied constant is independent of $k,m $ and $H$.
\end{theorem}

\subsection{Using M\"obius randomness law}
\label
{s:usmobtruncate} 
As usual,  $\mu(r)$ is the M\"obius function. In broad terms, 
the M\"obius randomness law is a general principle which states that long 
 sums containing the M\"obius function  should exhibit cancellation. 
An early example is the following result of Davenport, 
whose proof is based on bilinear sums techniques.
\begin{lemma}
[Davenport]
\label
{lem:davenportmobius}
Fix $A>0$. Then for all $y\geq 1$ we have
\[
\sup_{\alpha \in \R} 
\l|
\sum_{r\in \mathbb N \cap[1,y]   }
\mu(r)
\mathrm e^{i r  \alpha }
\r|
\ll
 y (\log y)^{-A}
,\] where the implied constant depends only on $A$.
\end{lemma} 
\begin{proof}
See \cite{davp} or \cite[Thm.~13.10]{iwa}.
\end{proof}
Recall that for $r \in \N$ we have $\Lambda(r)=
 -\sum_{d|r}\mu(d) \log d$.
We define the truncated von Mangoldt function
\[ 
 \Lambda_z(r)
:= 
 -\sum_{d\leq z, \,  d\mid r} \mu(d) \log d, \quad \text{where}\quad z\geq 1,\]
which will give rise to the main term in Theorem~\ref{thm:mainresulthalb} for suitably large $z$. 
The remainder 
$$\c E_z(r):=
\Lambda(r)
-
\Lambda_z(r)
$$
will contribute to the error term. 
When taking the sum over $r$, the variable
$d$ in 
 $\c E_z(r)
=
  -\sum_{ z< d,   d\mid r    } \mu(d) \log d 
$ runs over a long segment, so 
 the presence of $\mu(d)$
will give rise to cancellations. In particular, $\Lambda_z(r)$ is a good approximation to $\Lambda(r)$
for suitably large  $z$ and when one sums over $r$.
The advantage of this is that one can easily 
take care of various error terms in averages involving $\Lambda_z(r)$, due to truncation. 

We shall use the following corollary of Lemma~\ref{lem:davenportmobius}.
\begin{corollary}
\label
{cor:davenpmulog}
Fix $A>0$. Then for all $y,z\geq 1$ we have\[\sup_{\alpha \in \R} \l|
\sum_{r\in \mathbb N \cap[1,y]   }
\c E_z(r) \mathrm e^{i r \alpha } \r | \ll_A y (\log  y )  (\log z)^{-A} ,\] where the implied constant depends only on $A$.
\end{corollary}
\begin{proof}
See \cite[Eq.~(19.17)]{iwa}.
\end{proof}   For a function $F : \Z \to \R$ we denote  \[ S_F(\alpha ):= \sum_{\substack{ c \in \Z   \\   |c| \leq (d+1) \c M^d H  }} F(c) \mathrm e^{i c \alpha }, \]
where    $\c M=\max\{k,m \}$. Recall that for $ t\in \R, H \in [1,\infty) $  the Dirichlet kernel is defined as  
$$ D_H( t):= \sum_{|c|\leq H } \mathrm e^{   i c t }.$$ 
We will also use  $ D^+_H( t):= 
 \sum _{0<c\leq H} \mathrm e^{   i c t }$.
\begin{lemma} \label {eq:La Villa Strangiato - Rush} For any integers $ k, m  $ and  any functions $f,g:\Z\to\R$ we   have    \[
\sum_{\substack{  P\in \Z[t]   ,\, P>0   \\| P   | \leq H  ,\, \deg(P)=d    } }   \hspace{-0,6cm}
f(P(k)) g(P(m)) \! =\!\frac{1}{4 \pi^2} 
\hspace{-0,12cm}
 \int\limits_{(-\pi,\pi]^2 }   
\hspace{-0,25cm}
\overline {S_f (\alpha _1 )}\overline{ S_g (\alpha _2 )}  
D^+_H (k^d\alpha_1 + m ^d \alpha_2)
  \prod_{j=0}^{d-1} D_H (k^j\alpha_1 + m ^j \alpha_2)  \,   \mathrm{d} \boldsymbol\alpha   .\] \end{lemma} 
\begin{proof} Firstly, we write 
\[\sum_{\substack{   | P   | \leq H     \\ P>0   } }  f(P(k))   g(P(m))   =  \sum_{|k_1| , |k_2| \leq (d+1) \c M^d H} f(k_1) g( k_2)
\sum_{\substack{   | P   | \leq H     \\ P>0   } }    \mathds 1 (k_1= P(k ) ) \mathds 1 (k_2= P(m ) ) .\] The following identity 
holds for all integers $r $ and $s$:
\[ \mathds 1( r = s ) =\frac{1}{2\pi } \int_{-\pi}^{\pi } \mathrm e^{ i (r-s) \alpha }   \mathrm d \alpha.  \] Using it twice turns the sum into 
\[      \frac{1}{4 \pi^2 } \int_{-\pi}^{\pi }\int_{-\pi}^{\pi }   
\sum_{   |k_1|  \leq (d+1) \c M^d H } f(k_1)  \mathrm  e^{- i  k_1 \alpha_1 } 
\sum_{    |k_2| \leq (d+1) \c M^d H   } g( k_2) \mathrm e^{- i k_2 \alpha_2}
\sum_{\substack{   | P   | \leq H     \\ P>0   } }  \mathrm  e^{ i  (P(k )  \alpha_1 +   P(m )  \alpha_2) }   \mathrm d  \alpha_1 \mathrm d  \alpha_2 . \]
The sums over $k_1$ and $k_2$ are equal to $\overline {S_f (\alpha _1 )}$ and  $\overline {S_g (\alpha _2 )}$, respectively.
To analyse the sum over $P $ we write $P(t) =\sum_{j=0}^d c_j t^j $ and recall that $c_d\in(0,H]$.
   We obtain
\[\sum_{\substack{   | P   | \leq H     \\ P>0   } }  \mathrm  e^{ i  (P(k )  \alpha_1+    P(m )  \alpha_2) } =D^+_H (k^d \alpha_1 + m ^ d \alpha_2)
\prod_{j=0}^{d-1} D_H (k^j \alpha_1 + m ^ j \alpha_2).\qedhere \]  
\end{proof} Before proceeding we recall a well-known result of Lebesgue \cite[Eq.~(12.1), p.~67]{MR1963498},
\beq{eq:lebesgue}{\int_{-\pi}^\pi |D_H( t)| \mathrm d t = O(\log H).} \begin{lemma} \label{lem:kunst der fuge}
For any integers $  k\neq m $ and any  functions $f,g:\Z\to\R$ we   have    \[   \sum_{\substack{   | P   | \leq H     \\ P>0   } } f(P(k)) g(P(m))  \ll 
\| S_f\|_\infty  S_{|g|}(0) H^{d-1} \frac{\c M (\log H )^2}{|k-m |},\] where $\| S_f\|_\infty  :=\max\{ |S_{f}( \alpha ) |  :  \alpha  \in   \R \} $, and the implied constant depends at most on $d$.  \end{lemma} \begin{proof} 
The bounds  $|S_g(\alpha)| \leq S_{|g|}(0) $,  $|D^+_H(\alpha) |\leq H, 
|D_H(\alpha) |
 \leq 1+ 2 H$ and Lemma \ref{eq:La Villa Strangiato - Rush} give
$$\sum_{\substack{   | P   | \leq H     \\ P>0   } } f(P(k)) g(P(m)) \ll  \| S_f\|_\infty   S_{|g|}(0) H^{d-1}  \int_{(-\pi,\pi]^2 }  | D_H (  \alpha_1 +  \alpha_2)  |  
| D_H (k \alpha_1 + m  \alpha_2)  |  \mathrm{d} \boldsymbol \alpha .$$ The  change of variables 
$t_1=\alpha_1+\alpha_2$, $t_2=  k \alpha_1 + m  \alpha_2$ shows that the integral is at most
\[ 
\frac{1}{|k-m|} \int_{-2 \pi }^{2 \pi } \int_{-2 \pi \c M } ^ {2 \pi \c M } |D_H(t_1) | |D_H(t_2) | \mathrm d \b t    . \]
The Dirichlet kernel $D_H(t)$ is an even and $2 \pi$-periodic   function of $t$, thus 
 \[ \int_{-2 \pi }^{2 \pi }
 \int_{-2 \pi \c M } ^ {2 \pi \c M } |D_H(t_1) | |D_H(t_2) | \mathrm d \b t =
4 \c M  \int_{-\pi  }^{  \pi } \int_{-\pi  } ^ {  \pi   } |D_H(t_1) | |D_H(t_2) | \mathrm d \b t  . \] The proof concludes by invoking Lebesgue's result \eqref{eq:lebesgue}.
\end{proof}
 \begin{remark} \label{rem:allbut2} The proof of Lemma \ref {lem:kunst der fuge}  makes clear that in order to prove
Theorem \ref{thm:mainresulthalb}
 one needs to range over only two random coefficients and we are allowed to have   the remaining $d-1$ coefficients  fixed. 
   \end{remark}
 \begin{remark} \label{rem:allbut45} It would be interesting to study the $N$-th moment 
$\sum_{  \b P    }  
\l( \theta_\b P(x) -  \mathfrak S_\b P(x)  x\r)^N $
in \eqref{eq:aha2} for $N \geq 3$. 
 The proof of Lemma \ref {lem:kunst der fuge}  can be adapted for this problem as long as $d$ is not too small compared to $N$. 
For example, when $n=1$ one would need to take $d \geq N-1$. \end{remark}

\begin{proposition}
\label
{prop:oupsstat} 
Fix any  $d\geq 1$, $A>0$,
and 
 $\delta_1, \delta_2>0 $ with $\delta_1<1$.
Then for all 
$  z, H  \geq 1 $ 
such that $ H^{\delta_1} \leq z \leq  H  $
and all natural numbers $ k\neq m$ 
satisfying
\[
k, m  \leq 
 (\log H)^{\delta_2}
\]
  we have 
\[    \c G_{k,m}(H;d)
=
\sum_{\substack{   P\in \Z[t]   ,\, \deg(P)=d \\ | P   | \leq H  ,\, P>0   } }
\Lambda_z(P(k))
\Lambda_z(P(m)) 
+O_A\l(
\frac{ H^{d+1}  }{(\log H)^A}
 \r)
,\] where the implied constant does not depend on $k, m , H $ and $z$.
\end{proposition}
\begin{proof} 
For both choices
$f=  \c E_z   $ and  $f = \Lambda_z$ we have 
$|f(t)|
\leq 
\sum_{m\mid t }    
 \log m
\leq 
(\log t)\tau(t)$, where $\tau$ is the divisor function.
In particular, we get  
$\sum_{t\leq y}
|f(t)|
\ll y 
(\log  y)^2
$, which  
shows that $$S_{|f|}(0)\ll H (\log H)^2 \c M^d \ll H (\log H)^{2+d\delta_2 }.$$ 
Furthermore, by Corollary \ref {cor:davenpmulog} we have 
\beq{eq:2rgh8}  {  \|      S_{\c E_z} \|_\infty  \ll  _C \c M^d  H (\log H  ) (\log z)^{-C}\ll_{\delta_1}    H   (\log H)^{1+d \delta _2 -C}}
for every $C>0$.
Therefore, 
by Lemmas~\ref{eq:La Villa Strangiato - Rush}
and \ref{lem:kunst der fuge}
we obtain
 \[ 
\l| 
\sum_{ | P   | \leq H, P>0 }
\c E_z(P(k))
\c E_z(P(m))\r|
,\l|
\sum_{ | P   | \leq H, P>0   } 
\c E_z(P(k))
\Lambda_z(P(m)) \r|  
 \ll    \frac{ \c M H^{d+1 } }{   (\log H )^{C-  2 d \delta_2       -5  }  } 
 .\]    Using $\c M \leq (\log H)^{\delta_2 }$ and letting  $A=C-  (2 d +1) \delta_2       -5$ gives the required error term. The proof now concludes by
recalling that $\Lambda=\Lambda_z+\c E_z$.
\end{proof}
For later use we need a version of this result for one polynomial value instead of two but with the additional condition 
that the polynomial is in an arithmetic progression.   
\begin{lemma} \label{lem:unifimprog}
Fix any  $d\geq 1   $
and 
 $\delta_1, \delta_2>0 $ with $\delta_1<1$.
Then for all $  z, H  \geq 1, A>0 $, all natural numbers $ k ,  \Omega $, and all $ R \in (\Z/\Omega )[t]$ of degree at most  $d$
such that
\[
k   \leq 
 (\log H)^{\delta_2}, \
 H^{\delta_1} \leq z \leq  H  ,  \ \Omega \leq H 
\]
  we have 
\[    
\sum_{\substack{   | P   | \leq H, \,  P>0  \\ \deg(P)=d \\   P\equiv R \md{\Omega  }  } }
\Lambda (P(k)) 
- 
\sum_{\substack{   | P   | \leq H, \,  P>0  \\ \deg(P)=d \\   P\equiv R \md{\Omega  }     } }
\Lambda_z(P(k))  = O_A\l(
\frac{ H^{d+1}  }{(\log H)^A}
 \r)
,\] where the implied constant does not depend on $k, m , H , R, \Omega$ and $z$.
\end{lemma}
The crucial  point is that the estimate is uniform in the progression.
\begin{proof} 
Using that $\Lambda-\Lambda_z=\c E_z$ turns the left  hand side into 
 \[
\sum_{\substack{   | P   | \leq H, \,  P>0  \\ \deg(P)=d \\   P\equiv R \md{\Omega  }  } }
\c E_z(P(k))
=    \frac{1}{2 \pi  } \int_{-\pi}^{\pi } 
S_{\c E_z}(-\alpha_1) 
 \bigg(\sum_{\substack{   | P   | \leq H, \,  P>0  \\ \deg(P)=d \\   P\equiv R \md{\Omega  }  } }
  \mathrm  e^{ i  P(k )  \alpha_1    }   \bigg)
\mathrm d  \alpha_1  . \] Writing  $P(t)=\sum_{j=0}^d c_j t^j $ and 
choosing integers $0\leq r_j<\Omega $ such that  $R(t)\equiv \sum_{j=0}^d r_j t^j \md \Omega$,
converts    the right hand  sum over $ P$ into  
$$ \bigg(\sum_{\substack{ 0<c_d \leq H \\ c_d\equiv r_d\md \Omega}} \mathrm e^{i c_dk^d\alpha_1 }\bigg)
\prod_{j=0}^{d-1} \bigg(\sum_{\substack{ |c_j|\leq H \\ c_j\equiv r_j\md \Omega}} \mathrm e^{i c_jk^j\alpha_1 }\bigg).$$
For each $j \neq 0$ we bound the sum over $c_j$ trivially by $O(H)$.
Using~\eqref{eq:2rgh8} to bound $S_{\c E_z}$ gives
 \[
\sum_{\substack{   | P   | \leq H, \,  P>0  \\ \deg(P)=d \\   P\equiv R \md{\Omega  }  } }
\c E_z(P(k))
  \ll_{\delta_1}    H   (\log H)^{1+d \delta _2 -C} 
H^{d}   \int_{-\pi}^{\pi } 
 \bigg|\sum_{\substack{ |c_0|\leq H \\ c_0\equiv r_0\md \Omega}} \mathrm e^{i c_0\alpha_1 }\bigg|
\mathrm d  \alpha_1 
.\] It suffices to prove that the integral is $O(\log H)$, since taking $C$ large enough compared to $d\delta _2$ will complete the proof.

Letting $c_0=b\Omega+r_0$ makes the sum over $c_0$ equal to  
\[  \mathrm e^{i   r_0 \alpha_1 }  \sum_{\substack{ |b+r_0/\Omega  | \leq H/\Omega      }} \mathrm e^{i b     \Omega \alpha_1  } .\]  
Since $|r_0|\leq \Omega$, the terms in the sum over $b $ that do not satisfy $|b|\leq H/\Omega$ are at most $O(1)$ with an absolute implied constant.
Hence, \[ \int_{-\pi}^{\pi } 
 \bigg|\sum_{\substack{ |c_0|\leq H \\ c_0\equiv r_0\md \Omega}} \mathrm e^{i c_0\alpha_1 }\bigg|
\mathrm d  \alpha_1 
\ll 1+ \int_{-\pi}^{\pi } 
 \bigg| \sum_{\substack{ |b  | \leq H/\Omega      }} \mathrm e^{i b     \Omega \alpha_1  } \bigg|
\mathrm d  \alpha_1 =
 1+ \frac{1}{\Omega}
\int_{-\pi\Omega}^{\pi \Omega} 
|D_{H/\Omega }(t)|
 \mathrm d  t
.\] Since $|D_{H/\Omega }(t)|$ is even and has period $2\pi $ we can bound the integral  by 
 $ \ll   
\int_{-\pi}^{\pi } 
|D_{H/\Omega }(t)|
 \mathrm d  t$. Alluding to  Lebesgue's result \eqref{eq:lebesgue} is now sufficient to finish the proof.   \end{proof} 
\subsection{The main term}
\label
{s:mainterm234}
It now remains to estimate the sum 
involving $\Lambda_z$ 
in 
Proposition~\ref{prop:oupsstat}.
This will be straightforward but somewhat involved
because 
we need to  keep track of the dependence of the error term on the  
parameters $k$ and $m $. 
\begin{lemma}
\label
{lem:initialtransform} 
For all $z,   H\geq 1 $ 
with $z^2 \leq H$ 
and all distinct $k,m\in\N$ we have
\begin{align*}
\sum_{\substack{   P\in \Z[t]     \\ | P   | \leq H ,\, P>0\\   \deg(P)=d  } }  \hspace{-0,2cm}
\Lambda_z(P(k))
\Lambda_z(P(m)) 
&=
2^d 
H^{d+1}
\hspace{-0,5cm}
 \sum_{\substack{ c, l_0 \in \N  \\ cl_0 \leq z \\ \gcd(c,l_0)=1}}
  \hspace{-0,3cm}
\frac{ \mu(c)  \mu(l_0)^2 \gcd(l_0, k-m) }{  (cl_0)^2 }
\l( \sum_{\substack{ t \in \N   \\ c l_0  t  \leq z \\  \gcd(t, c l_0  )=1    }} 
\hspace{-0,5cm}
   \frac{   \mu(  t) \log (cl_0  t )    }{    t  }\r)^2
\\
&+
O (
H^d 
z^3 
)
,\end{align*} 
where  the  implied constant depends only on $d$.
\end{lemma}
\begin{proof}
Write $\b c=(c_0,\ldots,c_d)$ and $P(t) =P_{\b c}(t)= \sum_{i=0}^d c_i t^i $.
The left hand side becomes
\beq{eq:ref12}{
\sum_{k_1,k_2 \leq z }
\mu(k_1)
\mu(k_2)
\log(k_1)
\log(k_2)
\sum_{\substack{
\b c \in (\Z \cap [-H, H] )^{d+1}, \, c_d > 0 \\
k_1 \mid P_\b c(k), \,
k_2 \mid P_\b c(m)
}}
1
.}
We only need to consider the terms corresponding to square-free $k_1$ and $k_2$.
Then $
l_0=\gcd
(k_1,k_2),
l_1=k_1/l_0,
l_2=k_2/l_0$
are square-free and pairwise 
coprime.
The simultaneous conditions $k_1 \mid P_\b c(k)$, $k_2 \mid P_\b c(m)$ can be written equivalently as 
$$ P_\b c (k )  \equiv P_\b c (m) \equiv 0 \md{l_0}, \ l_1  \mid    P_\b c (k), \  l_2  \mid    P_\b c (m)  .$$
Then splitting the summation over each $c_i $ in arithmetic progressions modulo $l_0l_1l_2$ turns the sum over $\b c $ into 
\[
\sum_{\substack{
\b b \in (\Z \cap [0, l_0 l_1 l_2) )^{d+1} 
\\
P_\b b (k )  \equiv P_\b b (m) \equiv 0 \md{l_0}
\\
l_1  \mid    P_\b b (k), \,
l_2  \mid    P_\b b (m) 
 }}
\#
\l
\{
\b c \in (\Z \cap [-H, H] )^{d+1}: c_d > 0, 
\b c 
\equiv \b b \md{l_0 l_1 l_2} 
\r\}
.\]
Since 
$z^2 \leq H$
we have 
 $ l_0 l_1 l_2 \leq k_1 k_2 \leq z^2 \leq H$.
Therefore, the summand $\#\{\b c \}$
 is \[
\frac{1}{2} 
 \l(  \frac{2H}{l_0 l_1 l_2} \r)^{d+1} 
+O\l(\l(  \frac{H}{l_0 l_1 l_2} \r)^{d}     \r).\] 
By the Chinese Remainder Theorem, 
the number of terms in the  sum over $\b b $ is   
 \begin{align*}
\prod_{\substack{p\, {\rm prime}\\p\mid l_0}} \#\{\b b\in \F_p^{d+1}:P_{\b b }(k)= P_{\b b }(m)=0\}
& \prod_{\substack{p\, {\rm prime}\\p\mid l_1}} \#\{\b b\in \F_p^{d+1}:P_{\b b }(k)= 0\}
\\ \times  & \prod_{\substack{p\, {\rm prime}\\p\mid l_2}} \#\{\b b\in \F_p^{d+1}:P_{\b b }(m)=0\},\end{align*}
where we used that each $l_i $ is square-free and that $\gcd(l_i,l_j)=1$ for all $i\neq j $.
Fixing all $b_i$ except 
$b_0$ shows that 
 $$ \#\{\b b\in \F_p^{d+1}:P_{\b b }(k) =0\}=
 \#\{\b b\in \F_p^{d+1}:P_{\b b }(m)= 0\}=p^d
.$$   Fixing all $b_i$ except 
$b_0$ and $b_1$ shows that  
$ \#\{\b b\in \F_p^{d+1}:P_{\b b }(k)= P_{\b b }(m)=0\}$
 equals $p^{d-1}$ if $p \nmid  k-m $ and $p^d $ if $p \mid k-m$.
Hence, the number of terms in the  sum over $\b b $ is   
\[
(l_1l_2)^d 
\prod_{\substack{{\rm prime}\, p\mid l_0\\ p\mid   k-m }}p^d
\prod_{\substack{{\rm prime}\, p\mid l_0\\ p\nmid   k-m }} p^{d-1}
=(l_1l_2)^d l_0^{d-1} \gcd(l_0, k-m).\]Hence,~\eqref{eq:ref12} becomes   \[
2^d H^{d+1}
\hspace{-0,6cm}
\sum_{\substack{l_0, l_1,l_2 \in \N
 \\ 
\gcd(l_i,l_j)=1\, \text{\rm for}\, i\neq j
\\
 l_0 l_1,\, l_0 l_2 \leq z
    }}
\hspace{-0,6cm}
\mu(l_0)^2
\mu(l_1)
\mu(l_2)
\log (l_0 l_1)
\log (l_0 l_2) 
 \frac{ \gcd(l_0, k-m)
}{
  l_0^2 l_1 l_2
}
\]
up to a quantity whose 
modulus is 
 \beq{eq:odfj78}
{
\ll
H^d 
\hspace{-0,6cm}
\sum_{\substack{l_0, l_1,l_2 \in \N
 \\ 
\gcd(l_i,l_j)=1\, \text{\rm for}\, i\neq j
\\
 l_0 l_1,\, l_0 l_2 \leq z
    }}
\hspace{-0,6cm}
\mu(l_0)^2
\mu(l_1)^2
\mu(l_2)^2
\log( l_0 l_1)
\log (l_0 l_2)  
  \frac{ \gcd(l_0, k-m)}{l_0 }
.}
The condition $\gcd(l_1,l_2)=1$ has  indicator function given by 
 $$
\sum_{\substack{ c\in \N \\ c\mid \gcd(l_1,l_2) }} \mu(c) 
=\sum_{\substack{ c, t_1, t_2 \in \N \\ l_1=c t_1 ,\, l_2 =c t_2  }} \mu(c) 
,$$ hence
the sum over $l_0,l_1,l_2$ in the main term can be written as 
 \begin{align*}
&\sum_{\substack{c, l_0, t_1,t_2 \in \N
 \\  \gcd(l_0,c t_1 t_2 )=1  \\  l_0 c t_1,\, l_0 c t_2 \leq z    }} 
 \mu(l_0)^2 \mu(c) \mu(c t_1) \mu(ct_2) \log (l_0 c t_1) \log (l_0 c t_2)   \frac{  \gcd(l_0, k-m) }{  l_0^2 c^2 t_1 t_2 }
\\=&
\sum_{\substack{c  \in \N\cap [1,z]   }}\frac{\mu(c  ) }{c^2}
\sum_{\substack{ l_0, t_1,t_2 \in \N
 \\  \gcd(l_0,c t_1 t_2 )=1 \\  \gcd(c,  t_1  t_2)=1  \\  l_0 c t_1,\, l_0 c t_2 \leq z    }} 
 \mu(l_0)^2 \mu(  t_1) \mu(t_2) \log (l_0 c t_1) \log (l_0 c t_2)   \frac{  \gcd(l_0, k-m) }{  l_0^2  t_1 t_2 }
,\end{align*} where we used     that the presence of $\mu(c t_1)\mu(c t_2)$  
 forces $\gcd(c,t_1t_2)=1$ and $\mu(c t_1)\mu(c t_2)=\mu(c)^2\mu( t_1)\mu(  t_2)$.
The variables $t_1,t_2$ in the last sum are now independent hence we get the sum in the lemma.
Turning to~\eqref{eq:odfj78}, we use  
  $ \gcd(l_0, k-m)\leq l_0$ to bound it by 
   \[
\ll 
H^d 
 \sum_{\substack{l_0, l_1,l_2 \in \N
 \\ 
  l_0 l_1, \, l_0 l_2 \leq z
    }}
 \mu(l_0)^2
\mu(l_1)^2
\mu(l_2)^2
\log( l_0 l_1)
\log (l_0 l_2)   
\ll H^d (\log z)^2\l( \sum_{\substack{l_0, l_1 \in \N
 \\ 
  l_0 l_1\leq z 
    }}1\r)^2
\ll H^d z^2(\log z)^4
, \] which  completes the proof. \end{proof}

Our aim is now to prove asymptotics for the sum over $t$
in the right hand side of the equation in 
Lemma~\ref{lem:initialtransform}. We need the 
 following lemma.
\begin{lemma}
\label
{lem:pntconseq}
Fix any $A>0$.
Then 
for all $T\geq 1 $ and $q\in \N \cap [1,T^{1/2}]$
we have 
\[\sum_{\substack{ t\leq T/q \\ \gcd(t,q)=1}}\frac{\mu(t)\log (qt)}{t} 
=-\frac{q}{\phi(q)}+
O_A((\log  T )^{-A})
,\]
where the implied constants
depend only on $A$.
\end{lemma}
\begin{proof}
This can be deduced directly from
\beq{eq:easytoprove}{
\sum_{\substack{ t\leq T\\ \gcd(t,q)=1 }}
\frac{\mu(t) \log t }{ t }
=-\frac{q}{\phi(q)}
+O_A((\log  T )^{-A})
\
\text{ and }
\ 
\sum_{\substack{ t\leq T\\ \gcd(t,q)=1 }}
\frac{\mu(t)   }{ t }
=
O_A((\log  T )^{-A})
,}
which are consequences of the prime number  
theorem, see~\cite[Ex.~17, p.~185]{MR2378655}. 
\end{proof}

Recall the following standard bounds from~\cite[Thm.~2.9, Thm.~2.11]{MR2378655}:
 \beq
{eq:stndrd}
{
\frac{1}{\phi(n) } \ll \frac{\log \log n}{n},
\
\
\
\tau(n)\leq n^{O(\frac{1}{\log \log n })}
.}
 \begin{lemma}
\label
{lem:mobsumtn}
Keep the setting of  Lemma~\ref{lem:initialtransform} and fix an arbitrary positive constant $A$.
Then  the sum over the $c, l_0 $ 
in Lemma~\ref{lem:initialtransform} equals 
 \[  
\prod_{{\rm prime}\,p\mid k-m } \frac{p}{p-1}
 +O_A \l(\frac{|k-m|}{(
\log z
)^{A} }\r) ,\]    where the implied constant does not depend on $k, m , z $ and $H$.
\end{lemma}
\begin{proof} To apply Lemma~\ref{lem:pntconseq} we must have $c l_0\leq z^{1/2} $.
Using  the bound $\sum_{n\leq z } 1/n \ll \log z $ we see that the contribution of the terms failing this condition is 
in modulus at most 
\[
 \sum_{\substack{ c, l_0 \in \N  \\ cl_0 >  z^{1/2} }}
  \frac{    |k-m| }{  (cl_0)^2 }
\l( \sum_{\substack{    t  \leq z    }} 
    \frac{    \log  z   }{    t  }\r)^2
\ll |k-m| (\log z)^4
\sum_{s>  z^{1/2}} \frac{\tau(s) }{s^2}
,\] where we  write $s=cl_0$. By~\eqref{eq:stndrd}
the sum over $ s $ is $\ll \sum_{s>\sqrt z } s^{-3/2} \ll z^{-1/4}$, which is satisfactory.
By Lemma~\ref{lem:pntconseq} the remaining terms make the following contribution:
\[
 \sum_{\substack{ c, l_0 \in \N  \\ cl_0 \leq z^{1/2} \\ \gcd(c,l_0)=1 }}
  \frac{ \mu(c)  \mu(l_0)^2 \gcd(l_0, k-m) }{  (cl_0)^2 }
\l( \frac{(cl_0)^2}{\phi(cl_0)^2}
+O_A\l(\frac{1}{(\log z )^A}\r)
\r)^2
.\] 
The error term is  \[ \ll \frac{1}{(\log z )^A}  \sum_{\substack{ c, l_0 \in \N   }}   \frac{   | k-m| }{  (cl_0)^2 }   \ll \frac{ | k-m|}{(\log z )^A} . \]
The main term equals 
 \[ \sum_{\substack{ c, l_0 \in \N  \\ cl_0 \leq z^{1/2} \\ \gcd(c,l_0)=1 }}
  \frac{ \mu(c)  \mu(l_0)^2 \gcd(l_0, k-m) } {\phi(cl_0)^2}=
 \sum_{\substack{ c, l_0 \in \N   \\ \gcd(c,l_0)=1}}
  \frac{ \mu(c)  \mu(l_0)^2 \gcd(l_0, k-m) } {\phi(cl_0)^2}
+O\l(
 \sum_{\substack{ c, l_0 \in \N  \\ cl_0 >  z^{1/2} }}
  \frac{  |k-m| } {\phi(cl_0)^2}
\r)
.\] By~\eqref{eq:stndrd} we have 
\[ \sum_{\substack{ c, l_0 \in \N  \\ cl_0 >  z^{1/2} }}
  \frac{ 1  } {\phi(cl_0)^2}=\sum_{s>z^{1/2} } \frac{\tau(s)}{\phi(s)^2} \ll \sum_{s>z^{1/2} }  s^{-3/2} \ll z^{-1/4}
.\]
The main term has Euler product
\[
 \sum_{\substack{ c, l_0 \in \N   \\ \gcd(c,l_0)=1}}
  \frac{ \mu(c)  \mu(l_0)^2 \gcd(l_0, k-m) } {\phi(cl_0)^2}
=
\prod_{p \textrm{ prime}}
\l(1-\frac{1}{(p-1)^2}+\frac{\gcd(p,k-m)}{(p-1)^2}\r) 
. \] Only the primes dividing $k-m $ contribute. In particular, we get the product 
\[ \prod_{{\rm prime} \ p \mid k-m } \l(1 + \frac{1}{ p-1 }\r) =\prod_{{\rm prime} \ p\mid k-m} \frac{p}{p-1}  ,\] which concludes the proof.
 \end{proof}
Using Lemmas~\ref{lem:initialtransform} and \ref{lem:mobsumtn}  with $z =H^{1/8}$ we obtain 
 \begin{lemma}
\label
{lem:mobsumhdfhd8tn}
Fix any 
$\delta>0 $.  
Then for all $H\geq 1 , A>0 $, 
and all pairs of distinct natural numbers
$k, m \leq  (\log H)^{\delta}$ we have
 \[    
\sum_{\substack{   P\in \Z[t]   ,\, \deg(P)=d \\ | P   | \leq H  ,\, P>0   } }
\Lambda_z(P(k))
\Lambda_z(P(m)) 
=
2^d H^{d+1}
\prod_{{\rm prime}\,p\mid k-m } \frac{p}{p-1} +O_A\l(
 H^{d+1}
(\log H)^{-A}
\r),
 \]  where  $z =H^{1/8}$ and 
the  implied constant does not depend on $k$, $m $, 
$z $ and $H$.
\end{lemma}

Combining 
Proposition~\ref{prop:oupsstat}
with
Lemma~\ref{lem:mobsumhdfhd8tn}   
proves Theorem~\ref{thm:mainresulthalb}.
 \subsection{A variant}
We shall also need the following variant 
 of   Theorem~\ref{thm:mainresulthalb}.
\begin{lemma} \label{lem:inidim}  Fix any  $d\geq 1   $ and   $\delta >0 $. Then for all  $    H  \geq 1, A>0 $, all natural numbers 
$ k ,  \Omega $, and all $ R \in (\Z/\Omega )[t]$ such that $   k   \leq   (\log H)^{\delta }$ and $\Omega \leq H  $   we have 
\[     \sum_{\substack{   | P   | \leq H  ,\,  P>0  ,\,  \deg(P)=d \\  P(k)\,    \text {{\rm prime, }}    P\equiv R \md{\Omega  }  } }
\hspace{-0,3cm }  \log  P(k)  = \frac{2^d H^{d+1}  }{  \Omega^{d  }  \phi(\Omega )   }   \mathds 1 ( \gcd( R( k ) , \Omega) = 1 ) 
+O_A\l( \frac{ H^{d+1}  }{(\log H)^A}  \r) ,\] where the implied constant does not depend on $k,   H , R $ and $ \Omega$. 
\end{lemma} \begin{proof}  
If $\gcd( R( k ) , \Omega) \neq 1 $, then $P( k ) $ is a prime divisor of $\Omega$. 
Since there are $O(H^d )$ polynomials $P(t)$ of degree $d$ 
with $|P| \leq H$ such that $P(k)$ is  equal to a given integer, we deduce that the sum in the lemma 
is $ \ll  \#\{ \ell \textrm{ prime}: \ell \mid \Omega \} H^d \log H $. The number of prime divisors is
$\ll \log \Omega \leq \log H$, thus the proof is complete when $\gcd( R( k ) , \Omega) \neq 1 $. 

Let us now assume that 
$\gcd( R( k ) , \Omega) = 1 $.
We first transition to the von Mangoldt function by noting that 
\[ 
\sum_{\substack{   | P   | \leq H, \,  P>0  \\ \deg(P)=d \\   P\equiv R \md{\Omega  }  } }
\Lambda (P(k)) -
\sum_{\substack{   | P   | \leq H  ,\,  P>0  , \, \deg(P)=d \\  P(k)    \text { prime} \\ P\equiv R \md{\Omega  }  } }
\log  P(k) 
\ll
\sum_{2\leq \alpha \ll \log H } \sum_{\substack{ \ell \text{ prime } \\  \ell^\alpha  \leq (d+1) H k ^d  }}
(\log \ell ) 
\sum_{\substack{   | P   | \leq H  \\ \deg(P)=d \\ P(k ) = \ell^\alpha } } 1  
.\] 
The last sum over $P$ is $O(H^d)$, thus the error term is $\ll (\log H)^2 H^{d } (Hk^d)^{1/2} $, which is acceptable.
To conclude the proof it therefore suffices to consider $\sum_P \Lambda(P( k ) )$. 
 Define    $z=H^{1/4} $.   
By  Lemma \ref{lem:unifimprog} it is enough to estimate $$ 
\sum_{\substack{   | P   | \leq H, \,  P>0  \\ \deg(P)=d, \,   P\equiv R \md{\Omega  }     } }
\Lambda_z(P(k))  = -\sum_{\substack{   k_1 \leq z  \\ \gcd( k_1, \Omega)=1 }} 
\mu(k_1) (\log k_1) \sum_{\substack{   | P   | \leq H  , \, P>0  \\ k_1 \mid P( k ) ,  \, P\equiv R \md{\Omega  }     } } 1,$$ 
where $\gcd( k_1, \Omega)=1$ follows from $\gcd(R(k ) , \Omega)=1$. 
Hence  the sum over $P$ is 
\[2^d \l( \frac{H^{d+1}  }{k_1^{d+1}  \Omega^{d+1 }  }+O\l(1+\frac{H^d}{k_1^d \Omega^d }\r) \r) 
\#\{P \in (\Z/k_1)[t]: \deg(P)\leq d , P(k)\equiv 0 \md{k_1 }  \} 
.\] Since $\#\{P\}=k_1^ d $ and $k_1\leq z\leq H 
$, the above becomes  \[   \frac{2^d H^{d+1}  }{  \Omega^{d+1 }  }    \frac{1}{k_1 }   +O(H^d)   .\]  The error term contribution is \[\ll H^d \sum_{ \substack{  k_1  \leq z    } } 
\log k_1   \ll H^d z\log z \ll  H^{d+1/2} . \]  
The main term contribution is \[ - \frac{2^d H^{d+1}  }{  \Omega^{d+1 }  }  \sum_{ \substack{ k_1  \leq z   \\    \gcd( k_1, \Omega) =1   } }    
 \frac{\mu( k_1) \log k_1  }{k_1 } =\frac{2^d H^{d+1}  }{  \Omega^{d  }  \phi(\Omega )   }   + O_A \l(           \frac{ H^{d+1}  }{\log^A z}\r)
, \]  where we used~\eqref{eq:easytoprove}.     \end{proof}

\section{Dispersion} \label{e:estimatingthevar}  

Recall that $  \c V(x, H)$ was defined in~\eqref{eq:aha2}.
In this section we prove  $\c V(x, H)\ll x^2/(\log x)^{-1}$  
via Linnik's dispersion method \cite{MR0168543}. Theorem~\ref{thm:almostal}
then follows by the Cauchy--Schwarz inequality $\c R(x, H)^2\leq \c V(x, H)$.  
Removing the condition $P_i\equiv Q_i\md{M}$ can only increase $\#\texttt{Poly}(H)\c V(x, H)$, thus
 \beq
{eq:kakakafe}
{\#\texttt{Poly}(H) \c V(x, H)
  \leq  \sum_{\substack{ 
\b P \in \Z[t]^n,\, |\b P| \leq H
\\
 \deg(P_i)=d_i,\,  P_i>0  
} }
   \hspace{-0,5cm}
\theta_{\b P}(x)  ^2
 -2 x
\hspace{-0,5cm}
 \sum_{\substack{ 
\b P \in \Z[t]^n, \,|\b P| \leq H
\\
 \deg(P_i)=d_i, \, P_i>0  
} }
\hspace{-0,5cm}
\mathfrak S_{\b P}(x)
  \theta_{\b P}(x)  
+
x^2 
\hspace{-0,5cm}
\sum_{\substack{ 
\b P \in \Z[t]^n,\, |\b P| \leq H
\\
 \deg(P_i)=d_i,\,  P_i>0  
} }
\hspace{-0,5cm}
\mathfrak S_{\b P}(x) 
^2 
.}  
The term 
$\sum_{\b P} \theta_{\b P}(x)^2$
is studied in
\S \ref{s:introlambd} using 
Theorem~\ref{thm:mainresulthalb}.
 The terms
$\sum_{\b P} \mathfrak S_{\b P}(x)^2$ 
and $\sum_{\b P}
\mathfrak S_{\b P}(x)
 \theta_{\b P}(x)$ 
  are estimated 
  in 
\S\ref{ss:momo}
 and
\S\ref{sub:corecor}, respectively. 

Throughout this section $d=d_1+\ldots+d_n$. We write $P_i(t)=\sum_{j=0}^{d_i} c_{ij}t^j$
for each $i=1,\ldots,n$.

\subsection{The term $\sum_{\b P} \theta_{\b P}(x)^2$}
\label
{s:introlambd} 
  Recall that 
$\c G_{k,m}(H; d_i)$ is defined in (\ref{def:bachist}).
\begin{lemma}
\label
{lem:brunti}
 Fix any $\delta >0$. For all $x,H$ with  $1\leq x \leq (\log H)^{\delta }$  
we have 
\[
 \sum_{\substack{
 \b P \in \Z[t]^n,\, |\b P| \leq H \\  \deg(P_i)=d_i,\,  P_i>0  
} }
   \theta_{\b P}(x)^2
=
2 
\sum_{
\substack{
1\leq  m<  k  \leq x 
\\
k\equiv m \equiv n_0 \md{M} 
}}
\prod_{i=1}^n
 \c G_{k,m}(H 
; d_i)
+O\l(
x  {H}^{d+n} (\log H)^n
 \r)
,\]
where the implied constant depends only on $\delta$ and $d_i$.
\end{lemma}
\begin{proof}
First, note that 
for all $j\in \N$
we have 
$\mathds 1_{\text{primes}}(j)\log j 
\leq 
\Lambda(j)$, 
where $\Lambda$ is the von Mangoldt function. 
Therefore, the sum over the $P_i$ in our lemma is at most
\[ 
 \sum_{\substack{ P_1, \ldots, P_n
\\
|P_i| \leq H,\,  P_i>0  
} }
\l(
\sum_{\substack{  m \leq x  \\ m\equiv n_0 \md{M}   } } 
 \Lambda( P_1(m) )\ldots \Lambda(P_n(m))
  \r)^2
=
\sum_{\substack{
1\leq k , \, m \leq x 
\\
k\equiv m \equiv n_0 \md{M} 
 }} 
\prod_{i=1}^n
 \c G_{k,m}(H; d_i)
.\]
The contribution of the diagonal terms 
$k=m$ is at most
   \[ 
\sum_{1\leq  m \leq x }
\prod_{i=1}^n 
\sum_{\substack{   | P_i   | \leq H     ,\, P_i>0 \\ \deg(P_i)=d_i    } }
\Lambda(P_i(m ) )^2
.\]    Using 
$0\leq \Lambda(h) \leq \log h$ 
gives the bound 
 \[
 \ll 
(\log H  )^n
\sum_{1\leq  m
\leq x }
\prod_{i=1}^n 
\sum_{\substack{   | P_i   | \leq H,\, P_i>0   \\ \deg(P_i)=d_i    } }
\Lambda (P_i(m ) )
.\] We can now apply Lemma \ref {lem:inidim} with $\Omega=1$ and $d=d_i$. It shows that the sum over the $P_i$ is $O(H^{1+d_i } )$, hence
\[ (\log H  )^n \sum_{1\leq  m \leq x } \prod_{i=1}^n  \sum_{\substack{   | P_i   | \leq H ,\, P_i>0   \\ \deg(P_i)=d_i    } }
 \Lambda (P_i(m ) )\ll  
(\log H  )^n     x H^{d+n }  
,\] which is sufficient for the proof.
\end{proof}
\begin{remark} \label{rem:lowerbound} Lemma \ref {lem:brunti} shows  why we need to have  $x/(\log H)^n \to +\infty $: 
if $x$ is not this large compared to the typical size of the coefficients of the polynomials, then
 the diagonal terms  in the second moment dominate; using Lemmas \ref{lem:bachist23457}, \ref{lem:secomom},
\ref{lem:secomomter5226anna} it is then easy to see that
the three principal terms do not cancel. In particular, one has $$\c V(x,H) \asymp x (\log H)^n \gg x^2,$$ which is not sufficient for proving Theorem \ref{cool}.
\end{remark} 
Our next step is to use Theorem~\ref{thm:mainresulthalb} to estimate the sum over $m, k $ in Lemma~\ref{lem:brunti}.
This will give rise to an average of the multiplicative function $$  \prod_{ \text{\rm prime}\, p\mid t  }\left(1+\frac{1}{p-1}\right)^n  .$$
For this we need the following lemma.
\begin{lemma} \label{lem:wintner}
Fix any $n\in \N$ and $c>0$. Let $f$ be a  function  defined on the primes such that
$|f(p)|\leq c/p $ for all $p$. Then for all 
$x,T\geq 1 $ we have
\[ \sum_{\substack{ t\in \N \\ t\leq x  }}  \prod_{ \text{\rm prime}\, p\mid t  } (1+f(p))^n =
O(x)
\]and \[
\int_0^T \sum_{\substack{ t\in \N \\ t\leq x  }}  \prod_{ \text{\rm prime}\, p\mid t  } (1+f(p))^n   \mathrm d x 
=   \frac{T^2}{2}\prod_{  \text{\rm prime} \, p  }  \l( 1 +\frac{(1+f(p))^n-1}{p} \r)  +O(T^{3/2}),\]
where the implied constants depend  only on $n $ and $c$.
\end{lemma}
\begin{proof}
Wintner's theorem (as generalised by Iwaniec--Kowalski~\cite[Eq.~(1.72)]{iwa})
states that for any arithmetic function
  $g$ and any monotonic and bounded $h:[0,\infty)\to \R$, one has 
\beq{eq:wint}{
\sum_{t\leq x } (g\ast h)(t) =\int_0^ x\l(\sum_{t\leq y }\frac{g(t)}{t}h\l(\frac{y}{t}\r)\r)\mathrm d y +O\l(\sum_{t\leq x } |g(t)| \r)}
for all $x\geq 1 $. Here  $g\ast h $ is the Dirichlet convolution. Letting  $h=1$ and 
$$ g(t)=|\mu(t)| \prod_{ \text{\rm prime}\, p\mid t  } \l((1+f(p))^n-1\r) $$
gives $(g\ast h)(t) =\prod_{p\mid t }(1+f(p))^n $, hence, by~\eqref{eq:wint}, we obtain
\beq{eq:bkc}{
\sum_{\substack{ t\in \N \\ t\leq x  }}\prod_{ \text{\rm prime}\, p\mid t  } (1+f(p))^n =
\int_0^ x \sum_{t\leq y }\frac{g(t)}{t}  \mathrm d y +O\l(\sum_{t\leq x } | g(t)|  \r)
.}
For a prime $p$ we have 
\[ | g (p)  |=\l|\sum_{j=1}^n {n\choose j} f(p)^j\r| \leq \sum_{j=1}^n {n\choose j}\frac{c^j}{p^j} \leq \frac{2^\alpha}{p}\]
for some positive constant $\alpha$ that depends only on $n$ and $c$.
Therefore, by~\eqref{eq:stndrd}
we obtain 
\[
t|g(t)| \leq   |\mu(t)| \tau(t)^\alpha =O(t^{1/2})
.\]
This implies that for all $x,y \geq 1 $ one has 
\[
\sum_{t\leq x } | g(t)|
\ll \sum_{t\leq x } t^{-1/2} \ll x^{1/2}
\ \ \textrm{  and  } \ \  \sum_{t>y} \frac{|g(t)|}{t} \ll \sum_{t>y} t^{-3/2} \ll y^{-1/2}.\]
Therefore,  
\[
\sum_{t\leq y }\frac{g(t)}{t}
=
\sum_{t\in \N  }\frac{g(t)}{t}
+O(y^{-1/2})
=\prod_p \l(1+\frac{g(p) }{p}\r)
+O(y^{-1/2}).\] 
Using    $1+g(p) =(1+f(p))^n $ in the product and alluding to~\eqref{eq:bkc}, we obtain 
\[
\sum_{\substack{ t\in \N \\ t\leq x  }}\prod_{ \text{\rm prime}\, p\mid t  } (1+f(p))^n
 =x\prod_{{\rm prime}\, p} \l( 1 +\frac{(1+f(p))^n-1}{p} \r)+O(x^{1/2}).\]
Clearly this is $O(x)$, which proves the first claim in the lemma.
The second claim follows by integrating over the range $0\leq x \leq T$.
\end{proof}
Recall that $\gamma_n(\ell  )$ was defined in~\eqref{eq:bratsakipsomaki}.
\begin{lemma}
\label
{lem:bachist23457}
Fix any $\delta >0$. For all $x,H$ with  $1\leq x \leq (\log H)^{\delta }$
we have 
\[
 \sum_{\substack{ \b P\in \Z[t]^n,\, |\b P|\leq H
\\
\deg(P_i) =d_i,\,  P_i>0  
} }
 \theta_{\b P}(x)^2
=
\frac{x^2 M^{n-2 }  }{  \phi(M)^n }
 2^d
H^{d+n}
 \prod_{\text{\rm prime} \, \ell \nmid M }
\gamma_n (\ell  ) 
+O\l(
x
  H^{d+n} 
(\log H)^n
+x^{3/2}
H^{d+n}
 \r)
,\] 
where the implied constant 
 depends only on $\delta, n, M$ and $d_i$.
 \end{lemma}
\begin{proof} 
 Taking sufficiently large $A$ in 
 Theorem
\ref{thm:mainresulthalb} and 
 using 
Lemma~\ref
{lem:brunti} 
yields
 \[ 
 \sum_{\substack{ P_1, \ldots, P_n
\\
|P_i| \leq H,\,  P_i>0  
} }
   \theta_{\b P}(x)^2
=
 2^{d+1}
H^{d+n}
 T_0(x)
+O_A\l(
x  {H}^{d+n} (\log H)^n
+
 H^{d+n} (\log H)^{-A}
\r) 
,\] 
where  
\[
T_0(x) 
:= 
\sum_{\substack{ 1\leq  m <  k  \leq x \\
k
\equiv m \equiv n_0
 \md{M}
}} 
\prod_{{\rm prime}\, p\mid k-m } \frac{p^n}{(p-1)^n} .
\] 
We have $k-m= tM$ for some integer $t$. Hence,   $T_0(x)$ equals 
\beq{eq:refvgn}{
\sum_{\substack{ t\in \N \\ 1<t M \leq x } } \l(\prod_{p\mid t M } \frac{p^n}{(p-1)^n}\r)
\sum_{\substack{ m \in \N \\  m <  x-tM \\ m\equiv n_0 \md M  } }  1 
=
\sum_{\substack{ t\in \N \\ 1<t M \leq x } } 
\l(\prod_{p\mid t M } \frac{p^n}{(p-1)^n}\r)
\l(\frac{x}{M}-t +O(1)\r)
.} Define a function $f$ on the primes such that 
$ f(p)= 1/(p-1) $ if $p\nmid M$, and 
$f(p)=0$ if $p\mid M$. Then 
\[ \prod_{{\rm prime}\, p\mid t M } \frac{p^n}{(p-1)^n} = \frac{M^n}{\phi(M)^n} \prod_{{\rm prime}\, p\mid t   } (1+f(p))^n ,\] hence the right hand side of~\eqref{eq:refvgn} is 
\[
 \frac{M^n}{\phi(M)^n}
\sum_{ t \leq x/M   } 
\l( \prod_{{\rm prime}\, p\mid t   } (1+f(p))^n \r)
\l(\frac{x}{M}-t \r)
+O (x )
,\] where we used the first part of Lemma~\ref{lem:wintner}
to bound the contribution of the $O(1)$ term.
Using $\int_{t}^{x/M} 1\mathrm d y=x/M-t$ 
 we can write the sum over $ t $ as 
\[
\int_{0}^{x/M} 
\sum_{ t \leq y   }  \prod_{p\mid t   } (1+f(p))^n 
 \mathrm d y.
\] Invoking the second part of  Lemma~\ref{lem:wintner}  shows that 
 this is
\[
  \frac{x^2}{2M^2}\prod_{ {\rm prime}\, p\nmid M      }  \gamma_n(p)  +O(x^{3/2})
,\] which concludes the proof.
  \end{proof}
It is   convenient to truncate the product over $\ell$ in Lemma~\ref{lem:bachist23457} now, 
as it will make it easier to compare   
  $\sum_\b P \theta_\b P(x)^2$ to  $\sum_\b P \theta_\b P(x) \mathfrak S_\b P(x) $ and $\sum_\b P \mathfrak S_\b P(x)^2$.
\begin{lemma}
\label
{lem:10februarylemma} Fix $n \in \N$. Then for all $x\geq 1 $ we have 
\[
\prod_{\text{\rm prime} \,
\ell > \log x  }
\gamma_n(\ell)
=1+O\l(\frac{1}{\log x}\r)
.\]
\end{lemma}
\begin{proof}
The bound 
$(1+\psi)^n \leq 1 + n\psi  +n 2^n \psi^2 $, valid for all $0<\psi< 1$,
can be used for $\psi=1/(\ell -1)$
to show that  
\[
\gamma_n(\ell ) 
 =
 1-\frac{1}{\ell  }  +
\frac{1}{\ell  } 
\l(1+\frac{1}{\ell-1 }\r)^n 
\leq 
  1-\frac{1}{\ell  }  +
\frac{1}{\ell  } 
\l(
1 +\frac{n}{\ell-1}  +\frac{ n 2^n}{(\ell-1)^2 } 
\r)
\leq 1 +  \frac{n 2^{n+1}}{\ell(\ell-1)}  .\]
In particular, $\log \gamma_n(\ell) \leq  
  \frac{n 2^{n+1}}{\ell(\ell-1)} 
$.
We obtain  
\[
\log \l(
\prod_{\substack{\text{\rm prime} \ \ell \\ 
\ell > \log x  } } 
\gamma_n(\ell)
\r)
\leq  \sum_{\substack{\text{\rm prime} \ \ell \\ 
\ell > \log x  } }     \frac{n 2^{n+1}}{\ell(\ell-1)}  
\leq  n 2^{n+1} 
\sum_{\substack{k \in \N \\
k > \log x}}
\frac{1 }{k(k-1)}
\leq  \frac{n 2^{n+1} }{-1+\log x}
 .\] 
Exponentiating gives 
\[
\prod_{\substack{\text{\rm prime} \ \ell \\ 
\ell > \log x  }  } 
\gamma_n(\ell)
\leq
\exp\l(\frac{n 2^{n+2}}{-1+
\log x}
 \r)=1+O \l(\frac{1}{\log x }\r). \qedhere \]  
\end{proof} Combining 
Lemma~\ref{lem:10februarylemma}
with 
 Lemma~\ref{lem:bachist23457}
gives
\beq{eq:finsecmom}{
 \sum_{\substack{ \b P\in (\Z[t])^n, |\b P|\leq H
\\
\deg(P_i) =d_i,  P_i>0  
} }
  \hspace{-0,2cm}
\theta_{\b P}(x)^2
=
\frac{x^2 M^{n-2 }  }{  \phi(M)^n }
 2^d
H^{d+n}
 \prod_{\substack{
\ell   \nmid M 
\\
\ell \leq \log x
}} 
\gamma_n (\ell  ) 
+O\l(
\frac{x^2 H^{d+n}}{\log x }+
x
  H^{d+n} 
(\log H)^n
 \r)
.}

\subsection{The term $\sum_{\b P} \mathfrak S_{\b P}(x)^2$}
\label
{ss:momo} 
Let  $$W=\prod_{\substack{\text{\rm prime} \ \ell \\ 
\ell\nmid M, \ \ell \leq  \log x  } } \ell .$$ The prime number theorem implies  that  
\[\log W \leq\sum_{\substack{\text{\rm prime} \,
  \ell \leq  \log x  } }  \log \ell \leq 2  \log x, \]   whence we obtain
\beq {eq:tousebound4} {W \leq x^2 .}
  
\begin{lemma}
\label
{lem:secomom3}
For every square-free
$m \in \N$ 
we have
$$\sum_{
\substack{
R_1,\ldots, R_n \in (\Z/m)[t]
\\
 \deg(R_i) \leq d_i
}}
\hspace{0,3cm}
  \prod_{\text{\rm prime } \ell \mid m }
\l(
\frac
{1-\ell^{-1}Z_{R_1 \ldots R_n}(\ell)}
{(1-\ell^{-1})^n }
\r)^2 = m^{n+d }
\prod_{\text{\rm \ prime} \,    \ell \mid m }
\gamma_n(\ell).$$
\end{lemma}
\begin{proof}
A standard argument based on 
the Chinese remainder theorem shows that 
 the left hand side 
is a multiplicative function of $m$. Invoking 
  Lemma~\ref
{lem:corolhaha}
concludes the 
proof. 
\end{proof}
 \begin{lemma}
\label
{lem:secomom}
For  $ 1\leq  x \leq H^{1/4}$ we have  
\[
 \sum_{\substack{ 
\b P \in \Z[t]^n,\, |\b P| \leq H
\\
\deg(P_i)=d_i,\,
P_i>0  
} }
 \mathfrak{S}_{\b P}(x)
^2
= 
\frac{2^d H^{d+n} M^{ n-2}}{\phi(M)^{ n}  }
 \prod_{\substack{  {\rm prime}\,
\ell \nmid M
\\
\ell \leq   \log x 
}} 
\gamma_n(\ell) 
+O(H^{d+n-1/2})  
,\] 
where the implied constant depends only on $n , M $ and $d_1,\ldots,d_n$.
\end{lemma}
\begin{proof}  
By~\eqref{eq:bachbibaldi}
our sum can be rewritten as
\beq{eq:oursum}{\frac{M^{2n-2}} {\phi(M)^{2n}  }
 \sum_{\substack{ \b P \in (\Z[t])^n,\, |\b P| \leq H\\\deg(P_i)=d_i,\,P_i>0  \\ \gcd(M, \prod_{i=1}^n P_i(n_0) )  =1} }
B_{\b P}(x)^2,  \  \text{ where } \  B_{\b P}(x):=
\prod_{\substack{ \text{ prime} \, 
\ell \nmid M\\\ell \leq   \log x }} 
\frac{1-\ell^{-1}Z_{P_1 \ldots P_n } (\ell )}
{\l(1-\ell^{-1}\r)^n }.}If the coefficients of $P$ and $R$ in
$\Z[t]$ are congruent modulo $\ell$, then   $Z_P(\ell )=
Z_{R}(\ell )$.
Hence, denoting the reduction of $P_i(t)$ in $(\Z/W)[t]$ by $R_i (t)$,
the sum over the $P_i$ in~\eqref{eq:oursum} becomes
\[   \sum_{ \substack{ R_1,\ldots, R_n \in (\Z/W)[t] \\ \deg(R_i)\leq d_i  }} B_{\b R}(x)^2\,
 \# \left\{ P_1, \ldots, P_n \in \Z[t] : 
	\begin{array}{l}	 
	|P_i| \leq H,\,  P_i>0   \\	  
	\deg(P_i)=d_i,\,  P_i>0  \\
	P_i \equiv R_i \md{W} \\ 
	\gcd(M,  P_i(n_0) )  =1
	\end{array}	\right \}
	.\] By M\"obius inversion we have \[ \sum_{\substack{k_i\in \N  \\ 
	k_i \mid M,\,  k_i \mid P_i(n_0)  }}\mu(k_i) =\begin{cases}  1, &\mbox{if } \gcd(M,  P_i(n_0) )  =1, \\ 
0, & \mbox{otherwise. }   \end{cases}\]
 Hence, denoting   the reduction of $P_i(t)$ in $(\Z/k_i)[t]$ by $F_i (t)$,
we obtain   \[
   \sum_{
\substack{
R_1,\ldots, R_n \in (\Z/W)[t]
}}
B_{\b R}(x)
  ^2
\sum_{\b k\in \N^n ,\, k_i\mid M} 
\l(  \prod_{i=1}^n \mu(k_i)  \r)
\sum_{\substack{ 
F_1 \in  (\Z/k_1)[t], 
\ldots, F_n \in (\Z/k_n)[t]
\\ 
 F_i (n_0 ) \equiv 0 \md{k_i}
} }
 \sum_{\substack{ P_1, \ldots, P_n \in \Z[t]
\\ 
|P_i| \leq H,\,  P_i>0  
  \\
P_i \equiv R_i \md{W}
\\
P_i\equiv F_i\md{k_i }
} }
   1
,\] where $
\deg(P_i)=d_i$, 
$\max\{
\deg(R_i),
\deg(F_i)\}
\leq 
d_i$.
Viewing 
the sum over the $P_i$ as a sum over $1+d_i$ integers in arithmetic progressions modulo $k_i W$
we obtain  \begin{align*}
 &
\sum_{
\substack{\b R\in (\Z/W)[t]^n
\\ \deg(R_i) 
\leq 
d_i  
}}
B_{\b R}(x)^2
\sum_{\b k\in \N^n,\, k_i\mid M} 
\l(  \prod_{i=1}^n \mu(k_i)  \r)
 \sum_{\substack{ 
F_i \in (\Z/k_i)[t]
\\
 F_i (n_0 ) \equiv 0 \md{k_i}
\\ \deg(F_i) 
\leq 
d_i
} }
 \prod_{i=1}^n 
\l(\frac{2^{d_i} H^{1+d_i}}{(k_iW)^{1+d_i}}+O\l(1+\frac{H^{d_i} }{W^{d_i}  }\r)  \r)
.\end{align*} Now
note that $W\leq H^{1/2} $ due to    $  x \leq    H^{1/4} $ and \eqref {eq:tousebound4}. 
The sum over $F_1,\ldots , F_n  $ has $\prod_{i=1}^n  k_i^{d_i}$ terms because
the condition $F_i(n_0)\equiv 0\md{k_i}$ determines uniquely
the constant term of every $F_i$  by $n_0$ and the other coefficients of $F_i$.
This gives \[
  \sum_{
\substack{\b R\in (\Z/W)[t]^n\\ \deg(R_i) 
\leq 
d_i  
}}
B_{\b R}(x)^2
\sum_{\b k\in \N^n, \, k_i\mid M} 
\l(  \prod_{i=1}^n \frac{ \mu(k_i)}{k_i}  \r) 
 \l(
1 +  O(H^{-1/2})
\r) 
\frac
{2^d
H^{d+n}}
{W^{d+n}}
 \]
and   the   identity 
$\sum_{k\mid M } \mu(k) k^{-1}= \phi(M) M^{-1}$ shows that the sum over $\b P$  in~\eqref{eq:oursum}  is 
\[    \frac{\phi(M)^n}{M^n}    
\frac {2^d H^{d+n}}  {W^{d+n}}
 \l(  1 +  O(H^{-1/2}) \r)
\sum_{ \substack{\b R\in (\Z/W)[t]^n  \\ \deg(R_i) 
\leq 
d_i  }}
\prod_{\substack{ 
\text{ prime} \, 
\ell \nmid M
\\
\ell \leq   \log x 
}} 
\l(
\frac
{1-\ell^{-1}Z_{P_1 \ldots P_n } (\ell )}
{\l(1-\ell^{-1}\r)^n }
\r)^2
 . \] 
 By Lemma~\ref{lem:secomom3} applied to $ W$,
the quantity in \eqref{eq:oursum} becomes
 \[\frac{2^d H^{d+n} M^{ n-2}}{\phi(M)^{ n}  }
 \l(
1 +  O (H^{-1/2})
\r)\prod_{\substack{  
\ell \nmid M
\\
\ell \leq   \log x 
}} 
\gamma_n(\ell) = 
\frac{2^d H^{d+n} M^{ n-2}}{\phi(M)^{ n}  }
 \prod_{\substack{  
\ell \nmid M
\\
\ell \leq   \log x 
}} 
\gamma_n(\ell) 
+O(H^{d+n-1/2})   \]  
because $\prod_\ell \gamma_n(\ell)$ converges.
\end{proof}

 \begin{remark} It would be interesting to study moments higher than the second moment in the setting of Lemma \ref{lem:secomom}. 
This has been studied previously by Kowalski \cite{MR2786162}.
\end{remark}
 
\subsection{The term $
\sum_{\b P} \mathfrak S_{\b P}(x)
 \theta_{\b P} (x)
$}
\label
{sub:corecor}

 \begin{lemma}
\label
{lem:secomomter5226anna}
Fix any $A_2>0$. Then for all $x,H\geq 1 $ such that $ 1\leq x \leq (\log H)^{A_2 }$ 
  we have 
\[
  \sum_{\substack{ \b P \in (\Z[t])^n,\, |\b P | \leq H
\\
\deg(P_i)=d_i, \,  P_i>0  
} }
 \mathfrak{S}_{\b P}(x)
  \theta_{\b P}(x) 
=x
 2^{d} H^{d+n} 
\frac{M^{n-2} }{\phi(M)^n}
\prod_{\substack{{\rm prime}\,
\ell\nmid M
\\
 \ell \leq  \log x    
}} 
\gamma_n(\ell ) 
+
O\l(   H^{d+n}\r)
.\]  
\end{lemma} 
\begin{proof}   Using the definition of $\theta_\b P$ in~\eqref{def:thetP}
and changing the order of summation turns 
the sum over $\b P$ in our lemma into
\[  
\sum_{\substack{ m \in \N \cap [1,x]  \\ 
m\equiv n_0 \md{M}}}
\sum_{\substack{ \b P \in (\Z[t])^n,\, |\b P | \leq H
\\
\deg(P_i)=d_i, \,  P_i>0  
\\
P_i(m ) \text{ prime for}\, i=1,\ldots,n 
} }
 \mathfrak{S}_{\b P}(x)
\prod_{i=1}^n \log  P_i(m )
.\]
By~\eqref{eq:bachbibaldi} and~\eqref{eq:oursum} 
we can write this as 
\[  
\frac{M^{n-1} }{\phi(M)^n}
\sum_{\substack{ m \in \N \cap [1,x]  \\ 
m\equiv n_0 \md{M}}}
\sum_{\substack{ \b P \in (\Z[t])^n,\, |\b P | \leq H
\\ \deg(P_i)=d_i, \,  P_i>0  
\\ \gcd(M,  P_i(n_0) )  =1
\\ P_i(m ) \text{ prime for}\, i=1,\ldots,n 
} }
\Big(\prod_{i=1}^n \log  P_i(m ) \Big)
 B_{\b P}(x) .\] 
 Letting $R_i$ denote the reduction of $P_i $ in $(\Z/W)[t]$ we note that  
$ B_{\b P}(x)= B_{\b R}(x)$, hence we 
obtain 
  \beq {eq:fagakiexoume?} { \frac{M^{n-1} }{\phi(M)^n}
\sum_{\substack{  1\leq m  \leq x   \\ m  \equiv n_0 \md{M}  }}   
  \sum_{ \substack{\b R\in (\Z/W)[t]^n\\
\deg(R_i)\leq d_i  }}
B_{\b R}(x)   \prod_{i=1}^n  \l( \Osum_{\substack{ |P| \leq H  \\  P\equiv R_i \md{W}  }  } \log P( m )    \r) ,} where $\Osum$ has the extra conditions 
$     \deg(P)=d_i$,  $\gcd(P(n_0), M )=1  $, and $ P(m) $ is prime. The polynomials $P$ with $ \gcd(P(n_0), M )\neq 1 $ contribute $ O(H^{d_i} \log H)$ towards $\Osum$ because $P( m ) $ must be a prime divisor of $M$. Hence, ignoring the condition   $ \gcd(P(n_0), M )=1$,
brings $\Osum $ to a shape     suitable for the application of Lemma \ref {lem:inidim}. Thus for all $A>0 $ we have 
\[\Osum_{\substack{|P|\leq H \\ P\equiv R \md{W} }} \log P(m )=      \frac{2^{d_i }   H^{d_i +1}  }{  W^{d_i   }  \phi(W )   }   \mathds 1 ( \gcd( R_i( m ) , W) = 1 ) +O_A\l( \frac{ H^{d_i +1}  }{(\log H)^A}  \r) .\] 
     To study the contribution of the error term towards \eqref {eq:fagakiexoume?}
we bound every other  $\Osum  $ trivially by $ O(H^{1+d_i } \log H )$, hence we obtain
\[\ll  \frac{H^{d+n }} {(\log H)^{A-n}}   x    \sum_{ \substack{\b R\in (\Z/W)[t]^n\\ \deg(R_i)\leq d_i  }}  B_{\b R}(x) 
  \ll  \frac{H^{d+n }} {(\log H)^{A-n}}   x W^{d+n }  ( \log \log x)^n   ,\]  where we used  
$$ B_{\b R}(x)=
\prod_{\substack{ 
\text{ prime} \, 
\ell \nmid M
\\
\ell \leq   \log x 
}} 
\frac
{1-\ell^{-1}Z_{R_1 \ldots R_n } (\ell )}
{\l(1-\ell^{-1}\r)^n }
\leq 
\prod_{\substack{  
\ell \leq   \log x 
}}  
 \l(1-\ell^{-1}\r)^{-n } \ll ( \log \log x)^n  $$ which follows   from Mertens' theorem. Using \eqref {eq:tousebound4}, $ x \leq (\log H)^{A_2}$ and enlarging $A$ 
we see that the contribution towards \eqref {eq:fagakiexoume?} is $O(H^{d+n }  (\log H)^{-A})$. The main term is 
\[
\frac{ 2^d H^{d+n}   }{W^{d+n } \phi(W)^n }
\frac{M^{n-1} }{\phi(M)^n}
\sum_{\substack{ 
1\leq m  \leq x  
\\
m  \equiv n_0 \md{M} 
}}   
 \sum_{
\substack{\b R\in (\Z/W)[t]^n, \,
\deg(R_i)\leq d_i\\\gcd( R_i ( m ) , W) = 1 
}}
B_{\b R}(x)  
 .\]   By 
Lemma~\ref
{lem:corolhahparapoligelioorimaprosopa} and a factorisation argument  this becomes   \[
 2^d H^{d+n}  
\frac{M^{n-1} }{\phi(M)^n}
\l( x/M +O(1 ) \r)
 \prod_{ \ell \mid W } 
 \gamma_n(\ell)  = 
 2^d H^{d+n}  
\frac{xM^{n-2} }{\phi(M)^n}
 \prod_{ \ell \mid W } 
 \gamma_n(\ell)  + O\l (  H^{d+n}   \r) 
. \qedhere \]
\end{proof}

\subsection{The proof of  Theorem~\ref{thm:almostal}  }
\label
{s:asintheothersection}
Recall that $A_1,A_2$ are fixed constants with $n< A_1< A_2$
and that 
$(\log H)^{A_1}<x\leq(\log H)^{A_2}$.
Then~\eqref{eq:finsecmom},
together with 
  Lemmas  
\ref {lem:secomom}  and~\ref{lem:secomomter5226anna},
shows 
 that the right hand 
side of 
~\eqref
{eq:kakakafe}
is $
\ll
 x^2 H^{d+n} 
( \log x  )^{-1}$. 
The reason behind this is that 
the main terms
compensate each other.
Since
$H^{d+n} 
\ll 
\#\texttt{Poly}(H) 
$,
this concludes the proof of  
Theorem~\ref{thm:almostal}.

\subsection{The proof of 
Theorem \ref{cool}
}
\label
{s:kimamaior8ios}
 
To study the numerator
in the left hand side of~\eqref{eq:mahlersymphny5}
we use Theorem~\ref{thm:almostal}
to see that for almost all
 Schinzel $n$-tuples $\b P$
the prime
counting 
 function $\theta_\b P(x)$
is closely approximated by 
$\mathfrak S_\b P(x)x$.
\begin{lemma}
\label
{lem:mahlersymph3}
Let $\epsilon:\R \to (0,\infty)$ be a function.  
Fix any $A_1, A_2 $
with $n<A_1<A_2$.
Then for any $x, H\geq 2$ such that
$(\log H)^{A_1}<x< (\log H)^{A_2}$
we have {\rm 
$$\frac{
\#\{
\b P\in \texttt{Poly}(H)
: \b P \text{\ is Schinzel}, 
|\theta_\b P(x) - \mathfrak S_\b P(x) x | \leq  \epsilon(x) x 
 \}}
{
\#\{
\b P\in \texttt{Poly}(H)
: \b P \text{\ is Schinzel}
\}}
=1+O\l(
\frac{1}{\epsilon(x) (\log x)^{1/2}}
\r)
.$$ } 
\end{lemma}
\begin
{proof}
It is enough to show that
 \beq
{eq:mler}
{
\frac{
\#\{
\b P\in \texttt{Poly}(H)
: \b P \text{\ is Schinzel}, 
|\theta_\b P(x) - \mathfrak S_\b P(x) x | 
>
 \epsilon(x) x 
 \}}
{
\#\{
\b P\in \texttt{Poly}(H)
: \b P \text{\ is Schinzel}
\}}
\ll
 \frac{1}{\epsilon(x) (\log x)^{1/2} }
.}
The values of the function $|\theta_\b P(x) - \mathfrak S_\b P(x) x | 
 \epsilon(x)^{-1}  x^{-1}   $ are non-negative, and greater than 1 when
$|\theta_\b P(x) - \mathfrak S_\b P(x) x | > \epsilon(x) x $.
Thus the left hand side of~\eqref{eq:mler} is at most 
\[
\frac{1}
{
\#\{
\b P\in \texttt{Poly}(H)
: \b P \text{\ is Schinzel}
\}}
\sum_{\substack{\b P\in \texttt{Poly}(H)
 \\ \b P \text{\ is Schinzel}
 }}
\frac{|\theta_\b P(x) - \mathfrak S_\b P(x) x | }
 {\epsilon(x)  x}  .\]
Using Theorem~\ref{thm:almostal} we see that this is $$
\ll
\frac{   \#\texttt{Poly}(H)
}
{
\#\{
\b P\in \texttt{Poly}(H)
: \b P \text{\ is Schinzel}
\}}  \epsilon(x)^{-1} (\log x )^{-1/2}
.$$
An application of Proposition~\ref{prop:betagammadelta}
concludes the proof.   
\end
{proof}
We next show that if 
 $\b P$ is Schinzel,
then 
$\mathfrak S_\b P(x)$ stays at a safe 
distance from zero. Thus, 
$\mathfrak S_\b P(x)$ 
may be thought of as a `detector'
of Schinzel $n$-tuples.
\begin{lemma}
\label
{lem:beta0}
Let $\b P$ be a Schinzel $n$-tuple such that
$\prod_{i=1}^n P_i(n_0)$ and $M$ are coprime.
Then there exists a positive constant
$\beta_0=\beta_0(n,n_0,M, d_1, \ldots, d_n)$
such that  for all sufficiently large $x$ we have
$
\mathfrak S_{\b P}(x)
>
 \beta_0 (\log \log x)^{n-d}
$. 
\end{lemma}
\begin{proof}
Our assumption implies that
\[
\mathfrak{S}_\b P(x)
\gg 
 \prod_{\substack{ 
\text{ prime}\, 
\ell \nmid M
\\
\ell \leq d
}} 
\frac
{1-\ell^{-1}Z_{P_1 \ldots P_n } (\ell )}
{\l(1-\ell^{-1}\r)^n }
 \prod_{\substack{ 
\text{ prime} \, 
\ell \nmid M
\\
d< \ell \leq  \log x 
}} 
\frac
{1-\ell^{-1}Z_{P_1 \ldots P_n } (\ell )}
{\l(1-\ell^{-1}\r)^n }
.\]
To deal with the product over $\ell \leq d $,
we note that 
$Z_{P_1 \ldots P_n}(\ell) \neq \ell$ gives
$Z_{P_1 \ldots P_n}(\ell) \leq  \ell -1 $.
In particular, 
\[
  \prod_{\substack{ 
\text{ prime} \, 
\ell \nmid M
\\
\ell \leq d
}} 
\frac
{1-\ell^{-1}Z_{P_1 \ldots P_n } (\ell )}
{\l(1-\ell^{-1}\r)^n }
 \geq 
 \prod_{\substack{ 
\text{ prime} \, 
\ell \nmid M
\\
\ell \leq d
}} 
\frac
{\ell^{-1}}{\l(1-\ell^{-1}\r)^n }
\gg 
1. \]
To deal with the product over $\ell >d $
we observe 
 that $Z_{P_1 \ldots P_n}(\ell) \neq \ell$
implies that  
$P_1\ldots P_n$ is not identically
zero  in  $\F_\ell$, thus
$Z_{P_1 \ldots P_n}(\ell) \leq d$. 
This shows that  
\[
  \prod_{\substack{ 
\text{ prime} \, 
\ell \nmid M
\\
d< \ell \leq   \log x 
}} 
\frac
{1-\ell^{-1}Z_{P_1 \ldots P_n } (\ell )}
{\l(1-\ell^{-1}\r)^n }
\geq 
  \prod_{\substack{ 
\text{ prime} \, 
\ell \nmid M
\\
d< \ell \leq  \log x 
}} 
\frac
{1-d\ell^{-1}}
{\l(1-\ell^{-1}\r)^n }
\gg 
\prod_{d<\ell \leq
 \log x }
 \frac
{1-d\ell^{-1}}
{\l(1-\ell^{-1}\r)^n }
.\] For each fixed  $d \in \N$    we have 
$$\lim_{\psi\to 0 }
\psi^{-2} \left(
\frac{1-d\psi}{(1-\psi)^d} -1\right)  =-\frac{d(d-1) }{2} 
.$$ 
In particular, for each $d,n\in \mathbb N$ there exist constants $\psi_{d,n}>0, K_{d,n}>0$,
such that 
$$
\frac{1-d\psi}{(1-\psi)^n} \geq (1-\psi)^{d-n} \l(1-K_{d,n} \psi^2\r)
$$ for all $\psi\in (0,\psi_{d,n})$. We obtain 
\begin{align*}
 \prod_{d<\ell \leq  \log x } \frac{1-d\ell^{-1}}{\l(1-\ell^{-1}\r)^n }
&\gg_{d,n}
 \prod_{\max\{d,\psi_{d,n}^{-1}, K_{d,n} \}< \ell \leq  \log x } \frac{1-d\ell^{-1}}{\l(1-\ell^{-1}\r)^n }
\\&\geq 
 \prod_{\max\{d,\psi_{d,n}^{-1} , K_{d,n} \}< \ell \leq  \log x }
\l(1-\ell^{-1 }\r )^{d-n }\l(1- K_{d,n} \ell^{-2}\r) 
.\end{align*} 
By Mertens' estimate
this is $\gg_{d,n} (\log \log x )^{-n+d} $.
\end{proof}

\noindent{\bf End of proof of  Theorem \ref{cool}.}
Take $A_1=n+A/2$, $A_2=n+3A/4$ and let  $ x ,H,\epsilon(x)$ be   
as in Lemma~\ref{lem:mahlersymph3}. 
By Lemma \ref{lem:beta0}, $|\theta_\b P(x) - \mathfrak S_\b P(x) x | \leq  \epsilon(x) x $
implies 
$$
\theta_\b P(x) \geq \mathfrak S_\b P(x) x  -\epsilon(x)  x 
 \geq  
   \beta_0 (\log \log x)^{n-d} x  - \epsilon(x) x 
.$$
Hence Lemma~\ref{lem:mahlersymph3} gives
\[\frac{
\#\{
\b P\in \texttt{Poly}(H)
: \b P \text{\ is Schinzel}, 
\theta_\b P(x)
\geq 
(   \beta_0 (\log \log x)^{n-d}   - \epsilon(x) ) x 
 \}}
{
\#\{
\b P\in \texttt{Poly}(H)
: \b P \text{\ is Schinzel}
\}}
=1+O\l(
\frac{1}{\epsilon(x) (\log x)^{1/2} }
\r)
.\]
The choice $\epsilon(x)= \frac{1}{2} \beta_0 (\log \log x )^{n-d}$ gives
 \[\frac{
\#\{
\b P\in \texttt{Poly}(H)
: \b P \text{\ is Schinzel}, 
\theta_\b P(x) \geq   \frac{ \beta_0 }{2} (\log \log x)^{n-d}    x 
 \}}
{
\#\{
\b P\in \texttt{Poly}(H)
: \b P \text{\ is Schinzel}
\}}
=1+O\l(
\frac{(\log \log x )^{d-n}}{ \sqrt{\log x }}
\r)
. \] Since
$(\log H)^{A_1} < x \leq (\log H)^{A_2} $,  
the
error term is 
$\ll (\log \log \log H )^{d-n}   (\log \log H )^{-1/2} $, thus, 
 \beq
{eq:Monteverdi - Vespro della Beata Vergine}
{ 
\hspace{-0,2cm}
\frac{
\#\l\{
\b P\in \texttt{Poly}(H)
: \b P \text{\ is Schinzel}, 
\theta_\b P(x)
\geq  \frac{\beta_0 x}{2( \log  \log x)^{d-n} }   
 \r\}}
{
\#\{
\b P\in \texttt{Poly}(H)
: \b P \text{\ is Schinzel}
\}}
=1+O\l(
\frac{(\log \log \log H )^{d-n}}{ \sqrt{\log \log  H } }
\r)
. }  
It remains to find  a lower bound for $ \#S_{n+A}(\b P)$.
Observing that 
for all, except $O(H^{n+d-1/2})$, $n$-tuples  $\b P$ with $|\b P|\leq H$ 
one has $|\b P| > H^{1/2}$, we see that 
$x\leq (\log H)^{A_2}\ll (\log |\b P|)^{A_2} \leq (\log |\b P| )^{n+A}$,
hence $$ \theta_\b P (x )=
\sum_{\substack{ m \in \N \cap [1,x]  , \,
m\equiv n_0\md{M}
\\
P_i(m ) \text{ prime for}\, i=1,\ldots,n }}
\prod_{i=1}^n  \log  P_i(m )
 \leq \#S_{n+A}(\b P)  \prod_{i=1}^n  \log ((d_i+1) H x^{d_i}  ) 
$$ due to $m\leq x $ and $|\b P| \leq H$.
From $x \leq (\log H)^{A_2} $
we obtain $ \theta_\b P (x ) \ll \#S_{n+A}(\b P)   (\log H)^n  $.
By~\eqref{eq:Monteverdi - Vespro della Beata Vergine}
all, except $O(H^{n+d } (\log \log \log H )^{d-n}( \log \log  H )^{-1/2}   )$,
Schinzel $n$-tuples $   \b P\in \texttt{Poly}(H)$ fulfil $\theta_\b P(x) \geq   \frac{ \beta_0 }{2} (\log \log x)^{n-d}    x $. For these $\b P$ we use the upper and the lower bound for $\theta_\b P(x)$
 in conjunction with 
$x\geq (\log H)^{A_1}$
 to get  the following when $H\gg_{d,n, A} 1 $:
\[(\log H)^{n+A/3}\leq 
\frac{(\log H)^{A_1}}
{(\log \log  \log H)^{n-d} }
\ll
\frac{\beta_0 x}{2  (\log \log  x)^{n-d}  } 
 \leq
 \theta_\b P (x ) \ll \#S_{n+A}(\b P)   (\log H)^n   
.\] Together with
$|\b P| > H^{1/2}$, this gives  $\#S_{n+A}(\b P)  \geq (\log |\b P|)^{A/3}$. 
\qed

\section{Random Ch\^atelet varieties} 
\label{rat}

\subsection{Irreducible polynomials}
 
Let $K$ be a finite field extension of $\Q$ of degree $r=[K:\Q]$. Let 
${\rm N}_{K/\Q}: K\to\Q$ be the norm.
Choose a $\Z$-basis $\omega_1,\ldots,\omega_r$ of the ring of integers $\mathcal O_K\subset K$.
For $\z=(z_1,\ldots,z_r)$
we define a norm form 
$${\rm N}_{K/\Q}(\z)={\rm N}_{K/\Q}(z_1\omega_1+\ldots+z_r\omega_r).$$
For a positive integer $d$
consider the affine $\Z$-space $\AA_\Z^{d+2}=\AA^1_\Z\times \AA^{d+1}_\Z$, where 
$\AA^{d+1}_\Z={\rm Spec}(\Z[x_0,\ldots,x_d])$
and 
$\AA^1_\Z={\rm Spec}(\Z[t])$.
Let $V$ be the open subscheme of $\AA_\Z^{d+2}$ given by
$$P(t,\x):=x_dt^d+x_{d-1}t^{d-1}+\ldots+x_1t+x_0\neq0,$$
where $\x=(x_0,\ldots,x_d)$.
Let $U$ be the affine scheme given by $$P(t,\x)={\rm N}_{K/\Q}(\z)\neq0,$$
and let $f:U\to V$ be the natural morphism.
Note that $U_\Q$ is smooth over $V_\Q$ with geometrically integral fibres.
Let $g:U\to \AA^1_\Z$ be the projection to the variable $t$,
and let $h:U\to \AA^{d+1}_\Z$ be the projection to the variable $\x$.

For a ring $R$ and a point $\m=(m_0,\ldots,m_d)\in R^{d+1}$ 
of $\AA_\Z^{d+1}$ define $U_\m=h^{-1}(\m)$.
Then $g:U_\m\to \AA^1_R\setminus \{P(t,\m)=0\}$ is a morphism given by coordinate $t$.
For $\nu\in R$ we define $U_{\nu,\m}=f^{-1}(\nu,\m)$.

For a prime $p$, a point $(\nu,\m)\in \Z_p^{d+2}$ belongs to $V(\Z_p)$
if and only if $P(\nu,\m)\in\Z_p^*$. Similarly, $U(\Z_p)$ in 
$\Z_p^{d+2}\times(\mathcal O_K\otimes\Z_p)$
is given by $P(\nu,\m)={\rm N}_{K/\Q}(\z)\in\Z_p^*$.

\begin{lemma} \label{erti}
Let $S$ be the set of primes where $K/\Q$ is ramified. Then
for any $p\notin S$ and any $(\nu,\m)\in V(\Z_p)$ the fibre $U_{\nu,\m}$ has a $\Z_p$-point.
\end{lemma}  
 \begin{proof} This follows from the fact that
for any finite unramified extension $\Q_p\subset K_v$ any element of $\Z_p^*$ is the norm of 
an integer in $K_v$, see \cite[Ch. 1, \S 7]{CF}.
\end{proof}
\begin{lemma} \label{ori}
Let $p$ be a prime and let $N\in U(\Q_p)$.
There is a positive integer $M$ such that if $\nu\in\Q_p$ and 
$\m\in(\Q_p)^{d+1}$ satisfy
$$\max\big( |\nu-g(N)|_p, |\m-h(N)|_p\big)\leq p^{-M},$$ then
$U_{\nu,\m}(\Q_p)\neq\emptyset$.
\end{lemma}
\begin{proof} 
We note that $U_\Q$ is smooth, so every $\Q_p$-point of $U_\Q$
has an open neighbourhood $\mathcal U$ homeomorphic to an open $p$-adic ball.
Since $f\colon U_\Q\to V_\Q$, $V_\Q\to \AA^1_\Q$ and $V_\Q\to \AA^{d+1}_\Q$ are smooth morphisms, 
$g$ and $h$ are also smooth. This implies that
the maps of topological spaces $g\colon U(\Q_p)\to \Q_p$ 
and $h\colon U(\Q_p)\to (\Q_p)^{d+1}$ are open, cf. \cite[p.~80]{Conrad}.
Thus there exist open $p$-adic balls $\mathcal U_1\subset \Q_p$ with centre $g(N)$ and
$\mathcal U_2\subset (\Q_p)^{d+1}$ with centre $h(N)$
such that $\mathcal U_1\times\mathcal U_2\subset f(\mathcal U)$. 
 \end{proof}

\begin{theorem} \label{thm1}
Let $K$ be a cyclic extension of $\Q$ and let $S$ be the set of primes where $K/\Q$ 
is ramified.
Let $\mathcal P$ be the set of $\m\in\Z^{d+1}$ such that $P(t,\m)$
is a Bouniakowsky polynomial. 
Let $\mathcal M$ be the set of $\m\in\mathcal P$ such that $U_\m(\Z_p)\neq\emptyset$
for each $p\in S$. When $\mathcal P$ is ordered by height,  there is a subset $\mathcal M'\subset \mathcal M$ of density $1$ 
such that $U_\m(\Q)\neq\emptyset$ for every $\m\in\mathcal M'$.
The set $\mathcal M'$ has positive density in $\Z^{d+1}$ ordered by height.
\end{theorem}

\begin{remark} (1)
The Bouniakowsky condition at $p\notin S$ implies that $U_\m(\Z_p)\neq\emptyset$. 
Indeed, for $\m\in\mathcal P$
the reduction of $P(t,\m)$ modulo $p$ is a non-zero function $\F_p\to \F_p$.
Hence we can find a $t_p\in\Z_p$ such that $P(t_p,\m)\in\Z_p^*$ and apply Lemma \ref{erti}.
Likewise, the positivity of the leading term of $P(t,\m)$, which is
the `Bouniakowsky condition at infinity', implies that $U_\m$ has real points over large real values of $t$.
Thus in our setting the condition that $U_\m(\Z_p)\neq\emptyset$
for each $p\in S$ implies that $U_\m$ is everywhere locally soluble.

(2) The existence of a subset $\mathcal M'\subset \mathcal M$ of density $1$ 
can be linked to the triviality of the unramified Brauer group of $U_\m$ when $K/\Q$ is cyclic and
$P(t,\m)$ is an irreducible polynomial, as follows from \cite[Cor.~2.6 (c)]{CTHS03},
see also \cite[Prop.~2.2 (b), (d)]{Wei12}.
\end{remark}

\begin{proof} 
Since $\Z_p^*$ is closed in $\Z_p$
and $P(t,\x)$ is a continuous function, $V(\Z_p)$ is closed in $\Z_p^{d+2}$,
hence compact. For the same reason $U(\Z_p)$ is compact,
thus $h(U(\Z_p))$ is compact as a continuous image of a compact set.
Therefore, $\prod_{p\in S}h(U(\Z_p))$ is compact.

Take any $(N_p)\in \prod_{p\in S}U(\Z_p)$. For each $p\in S$
there is a positive integer $M_p$ such that
the $p$-adic ball $\mathcal B_{N_p}\subset \Z_p^{d+1}$ of radius $p^{-M_p}$ around
$h(N_p)$ satisfies the conclusion of Lemma \ref{ori}. Thus
the open sets $\prod_{p\in S}\mathcal B_{N_p}$, where $(N_p)\in \prod_{p\in S}U(\Z_p)$, 
cover $\prod_{p\in S}h(U(\Z_p))$. By compactness, there exist
finitely many points $(N_p^{(i)})\in \prod_{p\in S}U(\Z_p)$, $i=1,\ldots, n$,
such that the corresponding open sets $\prod_{p\in S}\mathcal B_{N_p^{(i)}}$
cover $\prod_{p\in S}h(U(\Z_p))$.

It follows that $\mathcal M=\cup_{i=1}^n \mathcal M_i$, where 
$\mathcal M_i=\mathcal M\cap \prod_{p\in S}\mathcal B_{N_p^{(i)}}$ for all $i$. 
Thus it is enough to prove that for 100\%
of $\m\in\mathcal M_i$ we have $U_\m(\Q)\neq\emptyset$.

In the rest of proof we write $\mathcal M=\mathcal M_i$ and $N_p=N_p^{(i)}$, where $p\in S$.
Write $n_p=g(N_p)$ and $\m_p=h(N_p)$, where $p\in S$. 
Note that $P(n_p,\m_p)\in\Z_p^*$ for each $p\in S$.
Write $M=\prod_{p\in S} p^{M_p}$. By the Chinese remainder theorem
we can find $n_0\in\Z$ and $\m_0\in\Z^{d+1}$ such that $n_0\equiv n_p\bmod {p^{M_p}}$ and 
$\m_0\equiv \m_p\bmod {p^{M_p}}$ for each $p\in S$. Our new set 
$\mathcal M$ consists of all $\m\in\mathcal P$
such that $\m\equiv \m_0\bmod M$.
Since $P(n_p,\m_p)\in\Z_p^*$ for each $p\in S$, we obtain that $P(n_0,\m_0)$ is coprime to $M$. 

Thus we can apply Theorem \ref{cor:almostal} to our $n_0$, $M$, with $Q(t)=P(t,\m_0)$.
It gives that for 100\% of $\m\in\mathcal M$, ordered by height,
one can choose $\nu\equiv n_0\bmod M$
such that $P(\nu,\m)$ is a prime. Call this prime $q$. 

We claim that $q={\rm N}_{K/\Q}(\xi)$ for some $\xi\in K^*$, so that $U_{\nu,\m}(\Q)\neq\emptyset$. 
Since $K$ is a cyclic extension of $\Q$, it is enough to show
that for all places $v$ of $\Q$, except possibly the place corresponding to the prime $q$, 
we have $U_{\nu,\m}(\Q_v)\neq\emptyset$,
see, e.g., \cite[Cor.~13.1.10]{CTS21} and references there. 
Indeed, the prime $q$ is a local norm at $\Q_v=\R$, since  
any positive real number is a norm for any finite extension. Next,
$q$ is a local norm at $\Q_p$ for $p\in S$, by the definition of $\mathcal M$ and Lemma \ref{ori}. 
Finally, $q$ is a local norm at $\Q_p$ for $p\notin S$, $p\neq q$, 
since $q\in\Z_p^*$ implies $(\nu,\m)\in V(\Z_p)$, so we can apply Lemma \ref{erti}.

Proving that $\mathcal M'$ has positive density in $\Z^{d+1}$
is equivalent to proving the same for $\mathcal M$. 
We have $\mathcal M=\cup_{i=1}^n \mathcal M_i$, where each $\mathcal M_i$ consists
of all Bouniakowsky polynomials $P(t)$ of degree $d$ satisfying 
$P(t)\equiv Q(t)\bmod M$ with $(Q(n_0),M)=1$.
Corollary~\ref{densitySch01234} implies that any such set has positive density.
Similarly, any non-empty intersection of some of the sets $\mathcal M_i$ also 
has positive density. By inclusion-exclusion $\mathcal M$ has 
positive density in $\Z^{d+1}$.
\end{proof}

\begin{remark} It is not clear to us if $U_{\nu,\m}(\Z)\neq \emptyset$.
\end{remark}

\begin{example} Let $K=\Q(\sqrt{-1})$. Then $S=\{2\}$.
Fix a positive integer $m\geq 2$. Let $s=|(\Z/2^m)^*|=2^{m-1}$. Consider
$$P(t)=3+(2^m-3)t^s + 2^{m+2}Q(t), \quad\text{where}\quad Q(t)\in\Z[t].$$
If $n\in\Z$ is even, then $P(n)\equiv 3 \bmod 4$ 
so $P(n)$ is not a sum of two squares in $\Q_2$.
If $n$ is odd, then $n^s \equiv 1 \bmod{2^m}$, hence
$P(n)$  is divisible by $2^m$.
Since $P(1)=2^m(1+4k)$ is a sum of two squares in $\Z_2$,
our equation $x^2+y^2=P(t)$ is solvable in $\Z_2$,
but for {\em any} $2$-adic solution the $2$-adic valuation of 
the right hand side is divisible by $2^m$.
This example shows that the set of $\m \in \Z^{d+1}$ such that
$U_\m(\Z_2)=\emptyset$ while $U_\m(\Q_2)\neq\emptyset$ has positive density. 
\end{example}

\medskip 

Let us now give a simpler version of Theorem \ref{thm1} applicable to some non-cyclic abelian
extensions $K/\Q$. 
Let $K^{(1)}$ be the Hilbert class field of $K$ and let
$K^{(+)}$ be the {\em extended Hilbert class field} of $K$, see \cite[p.~241]{J}
(it is also called the strict Hilbert class field \cite[Def. 15.32]{C}). 
By definition, $K^{(+)}$ is the
ray class field whose modulus is the union of all real places of $K$. 
Thus $K^{(+)}$ is a maximal abelian extension of $K$ unramified at all the {\em finite}
places of $K$, so that $K^{(1)}\subset K^{(+)}$.
By class field theory a prime $\mathfrak p$ of $K$ splits in $K^{(+)}$ if and only if
$\mathfrak p=(x)$ is a principal prime ideal with a totally positive generator $x\in K$.

\begin{theorem} \label{easy}
Let $d$ be a positive integer. Let $K$ be a finite abelian extension of $\Q$
such that $K^{(+)}$ is abelian over $\Q$.
Then for a positive proportion of polynomials $P(t)\in\Z[t]$ of degree $d$ 
ordered by height
the equation (\ref{norm-eq}) is soluble in $\Z$.
\end{theorem}
\begin{proof}
Since $K^{(+)}$ is abelian over $\Q$, by the Kronecker--Weber theorem
there is a positive integer $M$ such that $K^{(+)}\subset\Q(\zeta_M)$. 
Thus if a prime number $p$ is $1\bmod M$ then $p$ splits in $K^{(+)}$.
This implies that $p$ splits in $K$ so that every prime $\mathfrak p$ of $K$
over $p$ has norm $p$; moreover, $\mathfrak p$
splits in $K^{(+)}$ and so $\mathfrak p=(x)$ where $x\in \mathcal O_K$ is totally positive.
Then the ideal $(p)\subset\Z$ is the norm of the ideal $(x)\subset\mathcal O_K$,
hence $(p)=({\rm N}_{K/\Q}(x))$. Since $x$ is totally positive, we have 
${\rm N}_{K/\Q}(x)>0$, so $p={\rm N}_{K/\Q}(x)$.

A positive proportion of polynomials of degree $d$ are Bouniakowsky polynomials, and
a positive proportion of these are congruent to the constant polynomial $Q(t)=1$ modulo $M$,
by Proposition \ref{prop:betagammadelta}.
Taking $n_0=0$ in
Theorem \ref{cor:almostal} we see that for 100 \% of such polynomials $P(t)$
there is an integer $m$ such that $P(m)$ is a prime number
$p\equiv 1\bmod M$. Then $p={\rm N}_{K/\Q}(x)$ for some $x\in\mathcal O_K$.
\end{proof}

If $K$ is a totally imaginary abelian extension of $\Q$ of class number 1, then
$K=K^{(1)}=K^{(+)}$ so that Theorem \ref{easy} can be applied. For example,
this holds for $K=\Q(\sqrt{-1},\sqrt{2})$, which is one of 47 biquadratic extensions of $\Q$
with class number 1, see \cite{BP}. 
If $K$ is an imaginary quadratic field, then $K^{(1)}$ is abelian over $\Q$ 
if and only if the class group of $K$ is an elementary 2-group \cite[Cor. VI.3.4]{J}.

\subsection{Reducible polynomials}

Let $d_1,\ldots,d_n$ be positive integers. In this section we
let $U$ be the affine $\Z$-scheme given by 
\begin{equation}
\prod_{i=1}^n P_i(t,\b x_i)={\rm N}_{K/\Q}(\z)\neq 0, \label{red}
\end{equation}
where $\b x_i=(x_{i,0},\ldots,x_{i,d_i})$ and
$$P_i(t,\b x_i)=x_{i,d_i}t^{d_i}+x_{i,d_i-1}t^{d_i-1}+\ldots+x_{i,1}t+x_{i,0}, \quad 
\quad i=1, \ldots,n.$$
Write $d=d_1+\ldots+d_n$ and $\b x=(\b x_1,\ldots,\b x_n)$. 
Consider the affine space $\AA^{d+n+1}_\Z$ with coordinates
$t$ and $x_{ij}$ for all pairs $(i,j)$, where $1\leq i\leq n$ and $0\leq j\leq d_i$.
Define $V$ as the open subscheme of $\AA^{d+n+1}_\Z$
given by $\prod_{i=1}^n P_i(t,\b x_i)\neq 0$.
The morphism $f:U\to V$ is the product of the morphism $g$ (the projection to $t$)
and the morphisms $h_i$ (the projection to $\b x_i$), for $i=1,\ldots,n$.

\begin{theorem} \label{thm2}
Let $K$ be a cyclic extension of $\Q$ of degree $r=[K:\Q]$ with character
$$\chi:{\rm Gal}(\overline\Q/\Q)\longrightarrow\Z/r.$$
Let $S$ be the set of primes where $K/\Q$ ramifies.
Let $\mathcal P$ be the set of $\m=(\m_1,\ldots,\m_n)\in\Z^{d+n}$ such that 
$P_1(t,\m_1),\ldots,P_n(t,\m_n)$ is a Schinzel $n$-tuple.
Let $\mathcal M\subset\mathcal P$ be the subset whose elements $\m$
satisfy the following condition:

\smallskip

\noindent {\rm for each $p\in S$ there is a point $(t_p,\z_p)\in U_\m(\Z_p)$ 
such that for each $i=1,\ldots, n$ we have 
\begin{equation}\sum_{p\in S}{\rm inv}_p(\chi,P_i(t_p,\m_i))=0.\label{BM}\end{equation}}

\noindent Then there is a subset $\mathcal M'\subset \mathcal M$ of density $1$ 
such that $U_\m(\Q)\neq\emptyset$ for every $\m\in\mathcal M'$.
The set $\mathcal M'$ has positive density in $\Z^{d+n}$ ordered by height.
\end{theorem}

Let us explain the notation used in this statement. For a place $v$ of $\Q$ and $a\in\Q_v^*$
we denote by $(\chi,a_v)$ the element of the Brauer group
${\rm Br}(\Q_v)$ which is the class of the  
cyclic algebra over $\Q_v$ of degree $r$ defined by $\chi$ and $a_v$, see \cite[\S 1.3.4]{CTS21}.
We have $(\chi,a_v)=0$ if and only if $a_v$ is a local norm for the extension $K/\Q$.
The local invariant ${\rm inv}_v$ is an injective homomorphism
$${\rm inv}_v\colon {\rm Br}(\Q_v)\to\Q/\Z,$$
which is surjective if $v$ is a finite place, and has image $\frac{1}{2}\Z/\Z$ if $\Q_v=\R$. The sum of maps ${\rm inv}_v$ for all places $v$ of $\Q$ 
fits into the exact sequence
\begin{equation}
0\lra{\rm Br}(\Q)\lra \oplus_v{\rm Br}(\Q_v)\lra\Q/\Z\lra 0, \label{glob}
\end{equation}
where each map ${\rm Br}(\Q)\to {\rm Br}(\Q_v)$ is the natural restriction, 
see \cite[\S 13.1.2]{CTS21}.

\begin{remark} (1) For $n=1$ condition (\ref{BM})
is automatically satisfied, so we recover Theorem \ref{thm1} as a particular case
of Theorem \ref{thm2}.

(2) Since each $P_i(t,\m_i)$ is a Bouniakowsky polynomial, for each $p\notin S$ we can find
a $t_p\in\Z_p$ such that $P_i(t_p,\m_i)\in\Z_p^*$ and hence ${\rm inv}_p(\chi,P_i(t_p,\m_i))=0$.
Taking the product over $i=1,\ldots,n$ we see that $U_\m$ has a $\Z_p$-point over $t_p$.
Similarly, each $P_i(t,\m_i)$ takes positive values when $t_0\in \R$ is large,
so ${\rm inv}_\R(\chi,P_i(t_0,\m_i))=0$. Thus $U_\m$ has a real point over $t_0$.
Thus (\ref{BM}) implies that $U_\m$ has $\Z_p$-points $(t_p,\z_p)$ for all $p$ and a real point $(t_0,\z_0)$
such that $$\sum{\rm inv}_p(\chi,P_i(t_p,\m_i))=0$$ for $i=1,\ldots,n$, where the sum is over
all places of $\Q$.
Since $K/\Q$ is cyclic, from \cite[Cor. 2.6 (c)]{CTHS03} we know that the unramified Brauer
group of $U_\m$ is contained in the subgroup of ${\rm Br}(\Q(U_\m))$
generated by ${\rm Br}(\Q)$ and the classes $(\chi,P_i(t,\m_i))$, for $i=1,\ldots, n$. 
We conclude that for any smooth and proper model $X$ of $U_\m$, the Brauer group ${\rm Br}(X)$
does not obstruct the Hasse principle on $X$.
\end{remark}

\begin{proof} We follow the proof of Theorem \ref{thm1} with necessary
adjustments. The analogue of Lemma \ref{ori} says that
for $p\in S$ and $N_p\in U(\Q_p)$ there is a positive integer $M_p$ such that if $\nu\in\Q_p$ and 
$\m\in(\Q_p)^{d+n}$ satisfy
\begin{equation}
\max\big( |\nu-g(N_p)|_p, |\m_{i}-h_i(N_p)|_p\big)\leq p^{-M_p}, \ \text{for} \ i=1,\ldots,n,
\label{openset}
\end{equation}
then ${\rm inv}_p(\chi,P_i(\nu,\m_{i}))$ is constant and equal to 
${\rm inv}_p(\chi,P_i(g(N_p),h_i(N_p)))$.
This implies
\begin{equation}
{\rm inv}_p(\chi,\prod_{i=1}^nP_i(\nu,\m_{i}))=\sum_{i=1}^n{\rm inv}_p(\chi,P_i(\nu,\m_{i}))
={\rm inv}_p(\chi, \prod_{i=1}^nP_i(g(N_p),h_i(N_p)))=0, \label{midnight}
\end{equation}
in particular, $U_{\nu,\m}(\Q_p)\neq\emptyset$.

Let $Z\subset \prod_{p\in S}U(\Z_p)$ 
be the subset consisting of the points $(N_p)$ subject to the condition
\begin{equation}
\sum_{p\in S}{\rm inv}_p(\chi,P_i(g(N_p),h_i(N_p)))=0, \ \text{for} \ i=1,\ldots,n. \label{cond}
\end{equation}
The left hand side of (\ref{cond}), for a fixed $i$, takes values in $\Z/r$ and each level set is open,
hence also closed. We know that $\prod_{p\in S}U(\Z_p)$ is compact, hence $Z$ is compact.
Thus $f(Z)$ is compact, so $f(Z)$ can be covered by finitely many open subsets
given by congruence conditions on $\nu$ and $\m$ as in (\ref{openset}) such that (\ref{cond}) holds.

The condition (\ref{BM}) in the theorem implies that $\mathcal M\subset h(Z)$.
 As a consequence, using the Chinese remainder theorem, we represent 
$\mathcal M$ as a finite union of subsets $\mathcal M_j$, 
each of which consists of all Schinzel $n$-tuples satisfying a congruence condition of the form
$\m\equiv \m_0\bmod M$, where $\m_0\in\Z^{d+n}$ and $M=\prod_{p\in S}p^{M_p}$.
Moreover, there exists an
$n_0\in\Z$ with $(\prod_{i=1}^nP_i(n_0,\m_{0,i}),M)=1$ such that the following holds:
if $\nu\equiv n_0\bmod M$, then for all $\m\in\mathcal M_j$ we have
\begin{equation}
\sum_{p\in S}{\rm inv}_p(\chi,P_i(\nu,\m_i))=0, \ \text{for} \ i=1,\ldots,n,
\label{night}
\end{equation}
and
\begin{equation}
\sum_{i=1}^n{\rm inv}_p(\chi,P_i(\nu,\m_i))=0, \ \text{for} \ p\in S, \label{mrak}
\end{equation}
which follow from (\ref{cond}) and (\ref{midnight}), respectively.
It is enough to prove that for 100\%
of $\m\in\mathcal M_j$ we have $U_\m(\Q)\neq\emptyset$.

We apply Theorem \ref{cor:almostal} to our $n_0$ and $M$, with $Q_i(t)=P_i(t,\m_{0,i})$.
It gives that for 100\% of $\m$ there is an integer
$\nu\equiv n_0\bmod M$ such that each $q_i=P_i(\nu,\m_i)$ is a prime.
We have 
\beq{eq:BWV 853 Stokowski}{{\rm inv}_p(\chi, q_i)={\rm inv}_p(\chi,P_i(\nu,\m_i))=0}
for every prime $p\notin S\cup\{q_i\}$ and also for the real place. The real condition trivially
holds since $q_i>0$.  A prime $p\notin S\cup\{q_i\}$
does not divide $q_i$ and is unramified in $K$, so the condition holds for such $p$.
Therefore, by global reciprocity we have
\begin{equation}
{\rm inv}_{q_i}(\chi, q_i)=-\sum_{p\neq q_i}{\rm inv}_p(\chi, q_i)=
-\sum_{p\in S}{\rm inv}_p(\chi, q_i)=0, \ \text{for} \ i=1,\ldots,n, \label{long}
\end{equation}
where the last equality follows from (\ref{night}).
We claim that $${\rm inv}_p(\chi,q_1\ldots q_n)=0$$
for every prime $p$ (and also for the real place). 
This is clear for $p\notin S\cup\{q_1,\ldots, q_n\}$ and for the real place, but this is also clear
for $p=q_i$ by~\eqref{long} and \eqref{eq:BWV 853 Stokowski}.
Using (\ref{mrak}) we obtain the vanishing for $p\in S$, thus proving the claim.

The class $(\chi,q_1\ldots q_n)\in{\rm Br}(\Q)[r]$
has all local invariants equal to 0, so it is zero due to the exactness of (\ref{glob}).
Thus $\prod_{i=1}^nP(\nu,\m_i)=q_1\ldots q_n$ is a global norm for the extension $K/\Q$, so $U_{\nu,\m}(\Q)\neq\emptyset$. 

The last statement of the theorem is proved in the same way as the last statement
of Theorem \ref{thm1}, using Proposition \ref{prop:betagammadelta}.
 \end{proof}

\section{Random conic bundles} 
\label{rat2}
The classification of Enriques--Manin--Iskovskikh~\cite[Thm.~1]{isk} states that smooth 
projective geometrically rational surfaces over a field, up to birational equivalence, 
fall into finitely many exceptional families (del Pezzo surfaces of degree $1\leq d \leq 9 $) and 
infinitely many families of conic bundles $X\to \P^1$.
The generic fibre of a conic bundle over $\Q$ is a projective conic over the field $\Q(t)$ which can be 
described as the zero set of a diagonal quadratic form of rank 3.
We consider the equation
\beq
{eq:defisko}
{\hspace{-0,1cm}
a_1 \prod_{j=1}^{n_1}P_{1,j}(t)\,
x^2+a_2 \prod_{k=1}^{n_2}P_{2,k}(t)\,y^2
+a_3 \prod_{l=1}^{n_3}P_{3,l}(t)\,z^2=0,} 
where $a_1, a_2, a_3$ are fixed non-zero integers and $ P_{ij}\in \Z[t]$
is a polynomial of fixed degree $d_{ij}$, for $i=1,2,3$ and $j=1,\ldots,n_i$, where $n_1>0$, $n_2>0$ and 
$n_3\geq 0$. 
Let $d=\sum_{i,j}d_{ij}$. 
We write $P_{ij}(t,\m_{ij})$ for the polynomial of degree $d_{ij}$ with coefficients $\m_{ij}\in\Z^{d_{ij}+1}$, and write $\m=(\m_{ij})\in\Z^{d+n}$.
Let $U_\m\subset \P^2_\Z\times\A^1_\Z$ be the scheme given by equation (\ref{eq:defisko}) 
together with the condition $\prod_{i,j} P_{ij}(t,\m_{ij})\neq 0$. 
The proof of the following theorem is given in~\S\ref{proof_thm3}.

\begin{theorem} \label{thm3}
Let $n_1, n_2, n_3 $ be integers
such that $n_1>0$, $n_2>0$, and $n_3\geq 0$, and let $n=n_1+n_2+n_3$. Let 
$a_1, a_2, a_3$ be non-zero integers not all of the same sign and 
such that $a_1a_2a_3$ is square-free. Let $S$ be the set of prime factors of $2a_1a_2a_3$.
Let $d_{ij}$ be natural numbers, for $i=1,2,3$ and $j=1,\ldots,n_i$, and let $d=\sum_{i,j}d_{ij}$. 
Let $\mathcal P$ be the set of $\m=(\m_{ij})\in\Z^{d+n}$ such that the $n$-tuple
$(P_{ij}(t,\m_{ij}))$ is Schinzel.
Let $\mathcal M$ be the set of $\m\in\mathcal P$ such that $U_\m(\Z_p)\neq\emptyset$
for each $p\in S$. Then there is a subset $\mathcal M'\subset \mathcal M$ of density $1$ 
such that $U_\m(\Q)\neq\emptyset$ for every $\m\in\mathcal M'$.
The set $\mathcal M'$ has positive density in $\Z^{d+n}$ ordered by height.
\end{theorem}

\begin{remark} 
Let $\x=(x_{ij})$, for $i=1,2,3$ and $j=1,\ldots, n_i$, be independent variables.
We expect that for the generic polynomials $(P_{ij}(t,\x_{ij}))$ the unramified Brauer group 
of the conic bundle (\ref{eq:defisko}) over $\Q(\x)$ is reduced to ${\rm Br}(\Q(\x))$. 
This explains the absence of extra conditions like (\ref{BM}) in Theorem \ref{thm3}.
\end{remark}

\subsection{Correlations between prime values of polynomials and quadratic characters}
\label{s:verfsubs}

When $a$ and $b$ are integers such that $b>0$ we write $\l(\frac{a}{b}\r)$
for the Legendre--Jacobi quadratic symbol. We allow $b$ to be even, so that
$\l(\frac{a}{2}\r)$ is 0 or 1 when $a$ is even and odd, respectively.

A new analytic input in this section is the following result of Heath-Brown.

\begin{lemma}[Heath-Brown]
\label
{lem:largequa}
Let $(a_k)_{k\in \N}$ and $(b_l)_{l \in \N}$ be sequences of complex numbers
such that $a_k=0$ for $k> K$ and $b_l =0$ for $l>L$.
Then for any $\epsilon>0$ 
we have 
\[
 \sum_{\substack{\text{\rm primes } k, l }}
a_k
b_l
  \l(\frac{k}{l}\r)
\ll_\epsilon
\max\{|a_k|\}
\max\{|b_l|\}
\l((KL)^{1+\epsilon}
\l(\min \{K, L\} \r)^{-1/2}
+ K\r)
,\]
where the implied constant depends only on $\epsilon$.
\end{lemma}
\begin{proof}
We write the sum as 
\[
\sum_{\substack{k, l \in \N \\ l \text{ odd }}}
\l( a_k \mathds 1_{\text{primes}}(k) \r)
 \l( b_l \mathds 1_{\text{primes}}(l) \r)
  \l(\frac{k}{l}\r)
+ 
\sum_{k \text{ prime }} a_k
b_2
  \l(\frac{k}{2}\r)
.\]
By~\cite[Cor. 4]{MR1347489}
 the first sum 
is 
$\ll 
\max\{|a_k|\}
\max\{|b_l|\}
(KL)^{1+\epsilon}
\l(\min \{K, L\} \r)^{-1/2}$. 
The    second sum is 
 trivially bounded by  
$\max\{|a_k|\}|b_2|K$, which is enough.
\end{proof}

The following definition introduces a class of character sums to which Heath-Brown's estimate
will be applied.

 \begin{definition}
\label{def:sumsrpf}
Let $n\geq 2$. Let $\c F_1, \c F_2, \c G$ be functions
$$\c F_1, \c F_2 : \Z^{n-1} \to \{z \in \mathbb C:|z|\leq 1 \}, \quad 
\c G: \Z^{n-2} \to \{z \in \mathbb C:|z|\leq 1 \} ,$$where 
 $\c G $ is the constant function $1$ when $n=2$.
Let $\b P=(P_i) \in (\Z[t])^n$ be an $n$-tuple such that each $P_i$
has positive leading coefficient.
For any integers $h\neq k$ such that $1\leq h,k\leq n$ and 
any  $n_0 \in \N$, $M \in \N$, we define   
 \[\eta_{\b P  }(x; h, k )
\hspace{-0,1cm}
:=\hspace{-0,5cm}
\sum_{\substack{ m \in \N \cap [1,x]  \\ 
m\equiv n_0 \md{M}
\\
P_i(m ) \text{ prime,}\, i=1,\ldots,n }}
\hspace{-0,7cm}
\l(\prod_{i=1}^n  \log  P_i(m ) \r)
\hspace{-0,1cm}
 \l(\frac{P_h(m)}{P_k(m)}\r)
\hspace{-0,1cm}
\c F_1(P_a(m)_{a\neq k}) 
\c F_2(P_b(m)_{b\neq h})
  \c G (P_c(m)_{\substack{c\neq h\\c \neq k}})
.\]
Here the functions $\c F_1$, $\c F_2$, 
$\c G$ are applied to $P_1(m),\ldots, P_n(m)$, where 
$P_k(m)$ is omitted in $\c F_1$, $P_h(m)$ is omitted in $\c F_2$, and $P_h(m)$ and $P_k(m)$
are omitted in $\c G$.
\end{definition} 

Our work in previous sections shows that 
$\theta_\b P(x)$ is typically of size $x$.
We now prove
 that for $100\%$ of 
 $\b P\in (\Z[t])^n$ 
one has 
$\eta_{\b P}(x;h, k )=O(x^\delta)$ for some constant $\delta<1$.

\begin
{proposition}
\label
{prop:rhbsums}
Let $n, d_1, \ldots, d_n, M$ be positive integers and let
 $\c F_1, \c F_2,  \c G, h, k $ be as in Definition~\ref{def:sumsrpf}.
Let $n_0\in\N$ and $\b Q \in (\Z[t])^n$
be such that $(Q_i(n_0),M)=1$ for all $i=1,\ldots,n$.
Fix $A_1, A_2\in \R $ with $n<A_1<A_2$.
Then for all $H\geq 3  $
 and all $ x$ with  
$
(\log H)^{A_1}
<
x
\leq 
(\log H )^{A_2}
$
we have {\rm
$$
\frac{
1}{
\#{\texttt{Poly}}(H) 
} 
  \sum_{
\substack{ 
\b P \in 
\texttt{Poly}(H) 
  } } 
| \eta_{\b P  }(x; h, k )|
\ll  
x^{\frac{1}{2} +\frac{n}{ 2 A_1}} ,$$} where 
the implied constant depends only on $ d_1, \ldots, d_n, M, n_0, \b Q, A_1, A_2$. 
\end{proposition}
\begin{proof}
By the Cauchy--Schwarz inequality it is enough to prove
\beq
{eq:tobeusedlater2}
{
\frac{1}{
\#\texttt{Poly}(H) 
}
\sum_{
\substack{ 
\b P \in 
\texttt{Poly}(H) 
  } } 
|\eta_{\b P  }(x; h, k )|^2
\ll x^{1+\frac{n}{A_1} }
.}
Without loss of generality we assume that $h=1, k=2$
and write $\eta_{\b P}(x)$ for $\eta_{\b P}(x; 1,2)$.
Using $|\eta_{\b P  }(x)|^2=\eta_{\b P  }(x)\overline{\eta_{\b P  }(x)}$  
and changing the order of summation  we 
write $\sum_{\b P \in  \texttt{Poly}(H)   } |\eta_{\b P  }(x)|^2$   
as
\begin{align*}  
\sum_{\substack{ m_1, m_2  \in \N \cap [1,x]  \\ 
m_1, m_2 \equiv n_0 \md{M}
}}
\
 &\sum_{ \substack{  
\b P \in  \texttt{Poly}(H)   
\\ P_i(m_j  ) \text{ prime for}\, i=1,\ldots,n, \,  j=1,2
 } }  
  \l(\frac{P_1(m_1)}{P_2(m_1)}\r)
\l(\frac{P_1(m_2)}{P_2(m_2)}\r)
\l(\prod_{\substack{  1\leq i \leq n \\ j=1,2   }} \log  P_i(m_j ) \r)
\times 
\\
\times 
\c F_1(P_i(m_1)_{i\neq 2  })
&\c F_2(P_i(m_1)_{i\neq 1  })
  \c G ( P_i(m_1)_{i\notin \{1, 2 \} }) 
\times 
\\
\times 
&
\overline{\c F_1(P_i(m_2)_{i\neq 2  }) } \  \overline{\c F_2(P_i(m_2)_{i\neq 1  })  } \ \overline{ \c G ( P_i(m_2)_{i\notin \{1, 2 \} }) }.\end{align*} 
Ignoring the congruence conditions modulo $M$
and using
 $|\c F_i |, |\c G| \leq 1 $
 we see that the modulus of the 
contribution of the diagonal terms   $m_1=m_2$
is at most 
 \[  
\sum_{1\leq  m_1 \leq x }
\prod_{i=1}^n 
\sum_{\substack{   | P_i   | \leq H,\, P_i>0   } }
\Lambda(P_i(m_1 ) )^2    
,\] which is
 $
\ll
x H^{d+n} (\log H)^n $ as in the proof of Lemma~\ref{lem:brunti}. 
This is   
sufficient 
because 
\[
x H^{d+n} (\log H)^n
=
x H^{d+n} 
 ((\log H)^{A_1})^{n/A_1}
\leq 
x H^{d+n} 
x^{n/A_1}
\ll  \#\texttt{Poly}(H) x^{1+n/A_1}
.\]  To study the 
remaining terms we introduce the variables 
\[
k_1 :=    P_1(m_1), k_2 :=P_2(m_1)  
\ \text{  and } \   
l_ 1 := P_1(m_2), l_2:= P_2(m_2) 
\] and sum over all   values of 
$l_i, k _i $.
Take any $\epsilon>0$.
For any integer polynomial $P$ of degree at most $d_{i}$ satisfying $|P|\leq H$  and for 
any  $m\leq x  $  with $P_i(m)$ prime one has $ \log P_i(m)= O_{\epsilon,d_{i}}(H^\epsilon) $. Using this we bound the 
modulus of the 
remaining sum 
 by  $O(\Xi) $, where 
 \begin{align*}
\Xi:=
&\sum_{\substack{
l_1, l_2  
 \in \N  \\ 1\leq m_1 \neq  m_2  \leq x      }}
 (\log l_1)(\log l_2) 
 \sum_{\substack{   P_3, \ldots, P_n \in \Z[t]  \\  P_i >0,\, \deg(P_i)=d_i,\, |P_i|\leq H  } } 
\hspace{-0,4cm}H^\epsilon
 \l| \sum_{\substack{ k_1 , k_2  \text{ primes } }}  \l(\frac{k_1 }{ k_2   }\r)
F_1(k_1,l_1) F_2(k_2,l_2) 
\r|  , \end{align*}where for  $i=1,2$  
and  $k , l \in \mathbb N $
we let  \[
F_i(k,l ) := (\log k ) N_i(k,l) 
 \c F_i (k  , (P_j(m_1 ))_{j\notin \{1, 2 \} })
 \overline{ \c F_i (l , (P_j(m_2))_{j\notin \{1, 2 \} })} 
 , \] and denote by 
$N_i(k,l)$   the number
$$\#\{
P \in \Z[t] : P>0,  \deg (P)=d_i,  
 |P|\leq H,  P\equiv Q_i\md{M}, P(m_1) = k ,    P(m_2) = l\}.$$ 
To complete the proof of~\eqref{eq:tobeusedlater2}
it is now sufficient to prove \begin{equation}\label{eq:xibound}\Xi \ll  \texttt{Poly}(H) \   x^{1+\frac{n}{A_1} }.\end{equation}
The conditions $P(m_1) = k $,    $P(m_2) = l$
define an affine subspace of codimension 2 in the vector space of polynomials
of degree $d_i$, hence $N_i(k,l) 
\ll H^{d_i-1} $. 
(This uses   $m_1\neq m_2$, which explains the precursory manoeuvre of    separating the diagonal terms  $m_1=m_2$.)
We obtain 
the estimate 
$F_i(k,l) 
 \ll (\log H ) H^{d_i-1} $  
with an implied constant depending only on $n$ and $ d_i $.
Since 
we have
$|P_i(m_1) |\leq (1+d_i) H x^{d_i}$, we can see that  
 $N_i(k,l)= 0 $  
unless
 $k , l  \leq (1+d_i) H x^{d_i} $,    
so we can apply Lemma~\ref{lem:largequa}
with $
K=  (1+d_1) H x^{d_1} 
$ and $
L= (1+d_2) H x^{d_2} 
$.
Hence the sum over $k_1, k_2 $ in the definition of $\Xi$
is $
\ll 
H^{d_1+d_2-1/2 +\epsilon}
$, where we used that $x\leq (\log H)^{A_2}\ll H^\epsilon$.
Therefore, 
\[
\Xi\ll 
H^{d_1+d_2-1/2 +\epsilon}
\hspace{-0,5cm}
\sum_{\substack{ 
l_1\leq K , l_2 
\leq L  \\ 1\leq m_1 \neq  m_2  \leq x      }}
(\log l_1)(\log l_2)\hspace{-0,5cm} 
\sum_{\substack{   P_3, \ldots, P_n \in \Z[t]  \\
P_i >0, \, \deg(P_i)=d_i, \, |P_i|\leq H 
} }
H^\epsilon
.\]
The number of terms in the 
sum over the $P_i$ is   
 $\ll H^{d+n-d_1-d_2-2}$ and
the sum over 
  $l_1, l_2,  
m_1, m_2$ is $\ll K L x^2 (\log K) (\log L) \ll 
 H^{2+  \epsilon}$. 
This proves that 
 \[
\Xi\ll 
 H^{d+n-1/2+3\epsilon}
\ll
 \#\texttt{Poly}(H)  H^{-1/2+3\epsilon}
,\]which immediately implies \eqref{eq:xibound} 
by choosing $\epsilon=1/6$. 
\end{proof}

\subsection{Indicator function of solvable conics}

Recall that for $a,b,c\in\Q_p^*$ the projective conic
$$ax^2+by^2+cz^2=0$$
has a $\Q_p$-point if and only if the Hilbert symbol $(-ac,-bc)_p$ is $1$.
We refer to \cite[Ch. III, \S 1]{Serre} for the standard formulae for the calculation of the Hilbert symbol.

Let $a_1$, $a_2$, $a_3$ be non-zero integers.
Let $p_{ij}$, where $i=1,2,3$ and $j=1,\ldots,n_i$, be distinct primes not dividing $2a_1a_2a_3$.
(If $n_3=0$, then $i=1,2$.)
For $k\in\N$ write $[k]=\{1,\ldots, k\}$.
Let $S_i$ be a subset of $[n_i]$. Define
$\pi(S_i)=\prod_{j\in S_i} p_{ij}$ and abbreviate $\pi([n_i])$ to $\pi_i$.
We denote by $S_i^c=[n_i]\setminus S_i$ the complement to $S_i$ in $[n_i]$. 
Let
$$Q=2^{-n}\l(2+ \Osum_{S_1, S_2, S_3}
\l(\frac{-a_{2} a_{3}\pi_{2} \pi_{3}}{\pi(S_1)}\r)
\l(\frac{-a_{1} a_{3}\pi_{1} \pi_{3}}{\pi(S_2)}\r)
\l(\frac{-a_{1} a_{2}\pi_{1} \pi_{2}}{\pi(S_3)}\r)\r),$$
where the sum is over all subsets $S_i\subset[n_i]$, $i=1,2,3$, such that
$(S_1, S_2, S_3)\neq (\emptyset, \emptyset, \emptyset )$
and $(S_1 , S_2 , S_3)\neq ([n_1], [n_2], [n_3])$.

\begin{lemma}
\label{lem:chatel34}
Let $n_1, n_2, n_3 $ be integers such that $n_1>0$, $n_2>0$, $n_3\geq 0$.
Let $a_1, a_2, a_3$ be non-zero integers 
not all of the same sign such that $a_1 a_2 a_3$ is square-free. 
Suppose that $p_{ij}$, for $i=1,2,3$ and $j=1,\ldots, n_i$, are distinct primes 
not dividing $2a_1a_2a_3$ such that
the conic $C$ given by
\begin{equation}
a_1 \pi_1 x^2+a_2\pi_2 y^2+  a_3 \pi_3 z^2=0, \label{eqC}
\end{equation}
has a $\Q_p$-point for all $p|2a_1a_2a_3$. Then $C(\Q)\neq\emptyset$
if and only if $Q=1$, otherwise $Q=0$. 
\end{lemma}
\begin{proof} 
The condition concerning the signs of the $a_i$ guarantees that $C(\R)\neq\emptyset$.
Therefore, $C(\Q)\neq\emptyset$ if and only if for every $i,j$ we have   
\[
\l(\frac{-a_{i'}   a_{i''} \pi_{i'} \pi_{i''}}
{p_{ij}}\r)=1 
,\] where  $\{i,i',i''\}=\{1,2,3\}$.  
Thus the following is $2^n$ when $C(\Q)\neq\emptyset$, and
0 when $C(\Q)=\emptyset$:
$$
\prod_{i=1}^3
\prod_{j=1}^{n_i}
\l(1+\l(\frac{-a_{i'}   a_{i''}\pi_{i'} \pi_{i''}}
{p_{ij}}
\r)
\r)
=
\sum_{S_1,S_2,S_3} \l(\frac{-a_{2} a_{3}\pi_{2} \pi_{3}}{\pi(S_1)}\r)
\l(\frac{-a_{1} a_{3}\pi_{1} \pi_{3}}{\pi(S_2)}\r)
\l(\frac{-a_{1} a_{2}\pi_{1} \pi_{2}}{\pi(S_3)}\r),$$
where the sum is over all subsets $S_i\subset\{1,\ldots,n_i\}$, $i=1,2,3$.
We separate the term 1 corresponding to the case when $S_i=\emptyset$ for $i=1,2,3$.
The term corresponding to the case when $S_i=[n_i]$ for $i=1,2,3$ is 
$$\l(\frac{-a_{2} a_{3}\pi_{2} \pi_{3}}{\pi_1}\r)
\l(\frac{-a_{1} a_{3}\pi_{1} \pi_{3}}{\pi_2}\r)
\l(\frac{-a_{1} a_{2}\pi_{1} \pi_{2}}{\pi_3}\r).$$
This equals $(-1)^r$, where $r$ is the number of pairs $(i,j)$ such that $C(\Q_{p_{ij}})=\emptyset$.
Since $C$ is locally soluble everywhere except, perhaps, at the primes $p_{ij}$, the product
formula for the Hilbert symbol implies that $r$ is even. Hence the above term is 1. 
\end{proof}

\begin{proposition} \label{sunset}
Let $n_1, n_2, n_3 $ be integers such that $n_1>0$, $n_2>0$, $n_3\geq 0$, and let
$n=n_1+n_2+n_3$.
Let $a_1, a_2, a_3$ be non-zero integers not all of the same sign 
such that $a_1 a_2 a_3$ is square-free. 
Let $M$ be a multiple of $8a_1a_2a_3$.
Let $n_0$ be an integer. Let
$Q_{ij}(t)\in \Z[t]$ be a polynomial of degree at most $d_{ij}$ such that $(Q_{ij}(n_0),M)=1$,
for $i=1,2,3$ and $j=1,\ldots,n_i$,
satisfying the following condition: for 
any integer $m\equiv n_0\bmod M$ and
any $n$-tuple of polynomials $\b P=(P_{ij}(t))\in(\Z[t])^n$
with $\deg P_{ij}=d_{ij}$
such that  $\b P\equiv\b Q\bmod M$ 
the conic $\eqref{eq:defisko}$ with $t=m$ has a $\Q_p$-point, for any $p|M$.
Then for 100\% of Schinzel $n$-tuples $\b P\equiv\b Q\bmod M$ with $\deg P_{ij}=d_{ij}$, ordered by height,
the conic bundle surface \eqref{eq:defisko} has a $\Q$-point.
\end{proposition}
\begin{proof}
For $\b P \in (\Z[t])^n $ such that $\b P \equiv\b Q \md{M}$ define the following
counting function
\[
C_{\b P}(x):=
 \sum_{\substack{ m \in \N \cap [1,x]  \\ 
m\equiv n_0 \md{M}
\\
P_{ij}(m ) \text{ prime}\, \text{for all} \, i,j
\\ P_{ij}(m)\neq P_{rs}(m)\, \text{if}\, (i,j)\neq (r,s)}} 
\l(\prod_{i=1}^3 \prod_{j=1}^{n_i}
\log  P_{ij}(m )\r)\mathds 1(m), 
\]
where $\mathds 1$ is the indicator function of those $m $ for which 
the conic~\eqref{eq:defisko} with $t=m$ has a $\Q$-point. Define
$$
\widetilde{ \theta}_\b P (x )=
\sum_{\substack{ m \in \N \cap [1,x]  \\ 
m\equiv n_0 \md{M}
\\
P_i(m ) \text{ prime for}\, i=1,\ldots,n 
\\
P_{ij}(m)\neq P_{rs}(m)\, \text{if}\, (i,j)\neq (r,s)}} 
\prod_{i=1}^3 
 \prod_{j=1}^{n_i}
\log  P_{ij}(m)
.$$
By the condition in the proposition and Lemma~\ref{lem:chatel34} we have 
\beq{eq:lastgoal1}
{
C_{\b P}(x)=
\frac{1}{ 2^{n-1} }
\widetilde{ \theta}_\b P (x ) 
+
\frac{1}{2^n}
 \Osum_{\b S\ } T_{\b S, \b P}(x). 
}
Here $\Osum$ is the sum over $\b S=(S_1,S_2,S_3)$, where $S_i\subset [n_i]$ for $i=1,2,3$
are such that at least one $S_i$
is non-empty and at least one complement $S_j^c=[n_j]\setminus S_j$ is non-empty, and
\begin{equation} \label{bigT}
\hspace{-0,2cm}
T_{\b S, \b P}(x):=
\hspace{-0,2cm}
 \sum_{\substack{ m \in \N \cap [1,x]  \\ 
m\equiv n_0 \md{M}
\\
P_{ij}(m ) \text{ prime for all} \, i,j\\
P_{ij}(m)\neq P_{rs}(m)\, \text{if}\, (i,j)\neq (r,s)}} 
\hspace{-0,2cm}
\prod_{i=1}^3 
 \l(\frac{-a_{i'}a_{i''}\prod_k P_{i' k}(m) \prod_l P_{i'' l}(m)  }
{\prod_{j \in S_i } P_{i j } (m)   }\r) 
\prod_{j=1}^{n_i}
\log  P_{ij}(m)
,\end{equation}
where $\{i,i',i''\}=\{1,2,3\}$. The bound $P_{ij}(m)=O_{d_{ij } } ( H x^{d_{ij} } ) $ yields     
$\log P_{ij}(m) =O_{d_{ij } } (  \log (Hx) )$, hence
\beq{eq:cantata}{ 0\leq  \theta_\b P (x )-\widetilde{ \theta}_\b P (x )
\ll_{n,d_{ij} } (\log (H x) )^n
.}

 We claim that for   all $x$ and $H\geq 3  $ with  
$
(\log H)^{2n }
<
x
\leq 
(\log H )^{3n  }
$ and all $\b S$ as above  we have 
\beq
{eq:lastgoal2}
{
   \frac{
1}{
\#\texttt{Poly}(H) 
} 
 \sum_{
\substack{ 
\b P \in 
\texttt{Poly}(H) 
  } } 
| T_{\b S, \b P}(x) |\ll  x^{3/4}
.} Assuming this, we see from~\eqref{eq:lastgoal1} 
 and~\eqref{eq:cantata} 
  that \[   \frac{1}{\#\texttt{Poly}(H) }  \sum_{ \substack{  \b P \in \texttt{Poly}(H)    } }  | C_{\b P}(x)-  2^{-n+1}  \theta_{\b P}(x) |     \ll  x^{3/4}
 +(\log H)^n \ll x^{3/4} \] due to $(\log H)^n \leq x^{1/2} $.
  Therefore, 
\[
\frac{
\#\{
\b P\in \texttt{Poly}(H)
:   
| C_{\b P}(x)-
 2^{-n+1} 
\theta_{\b P}(x) |   > x^{4/5}  \}}
{
\# \texttt{Poly}(H)}
\leq 
   \frac{
1}{
\#\texttt{Poly}(H) 
} 
 \sum_{
\substack{ 
\b P \in 
\texttt{Poly}(H) 
  } } 
\frac{| C_{\b P}(x)-
 2^{-n+1} 
\theta_{\b P}(x) |   }{x^{4/5}} 
,\] is $\ll  x^{-1/20}\ll (\log H)^{-2n/20}$. 
Schinzel $n$-tuples
$\b P\equiv \b Q \md{M}$
have positive density within $\texttt{Poly}(H) $
by Proposition~\ref{prop:betagammadelta}, hence, for $100\%$ of them
  one has 
\[
  C_{\b P}(x) \geq 
 2^{-n+1} 
\theta_{\b P}(x)    - x^{4/5}\geq 
 2^{-n+1} 
 \frac{\beta_0 x}{2( \log  \log x)^{d-n} }   - x^{4/5}
,\] where we used~\eqref{eq:Monteverdi - Vespro della Beata Vergine} in the second inequality.  
(The constant $\beta_0$ was introduced in Lemma \ref{lem:beta0}.)
Since $x\geq (\log H)^n $, we see that for all sufficiently large $H$ one has 
$  C_{\b P}(x) >0$.

To verify~\eqref{eq:lastgoal2} we
check that $T_{\b S, \b P}(x)$ is a particular case of the sum introduced in
Definition~\ref{def:sumsrpf}. 
(This crucially uses the assumptions $n_1>0$ and $n_2>0$.)
Using quadratic reciprocity 
and the identities $\pi_i=\pi(S_i)\pi(S_i^c)$, $i=1,2,3$, we rewrite each summand in 
(\ref{bigT}) as 
the product of $\prod_{i,j}\log  P_{ij}(m)$ and
$$ \l(\frac{-a_2a_3\pi(S_2^c)\pi(S_3^c) }{\pi(S_1)  }\r)   
\l(\frac{-a_1a_3\pi(S_1^c)\pi(S_3^c) }{\pi(S_2)  }\r)  
\l(\frac{-a_1a_2\pi(S_1^c)\pi(S_2^c) }{\pi(S_3)  }\r)
$$
multiplied by the product of  $(-1)^{(p-1)(q-1)/4}$ for all primes $p\in S_i$ and $q\in S_{i'}$, 
where $i\neq i'$.
Without loss of generality
we can assume that $S_1\neq\emptyset$. Take any $k\in S_1$.
If  $S_2^c $ or $S_3^c $ is non-empty, say $S_2^c\neq\emptyset$, choose any $h \in S_2^c$ and
separate the term $(\frac{P_h(m)}{P_k(m)})$
in the first quadratic symbol above. 
If $S_2^c $ or $S_3^c $ are both empty, then $S_1^c \neq \emptyset$ and
 $S_2\neq\emptyset$. Hence
there exist $h \in S_1^c $ and $k \in S_2$ so that 
we can separate the term $(\frac{P_h(m)}{P_k(m)})$ in the second quadratic symbol above.
Let $\c F_1$ be the product of
all the terms involving $h$ but not $k$, let $\c F_2$ be the product of all the terms involving
$k$ but not $h$, and let $\c G$ be the product of all the terms that depend neither on $k$ nor on $h$. We conclude by applying Proposition \ref{prop:rhbsums} with $A_1=2n$ so that $\frac{n}{2A_1}=\frac{1}{4}$.
\end{proof}

\subsection{Proof of Theorem \ref{thm3}} \label{proof_thm3}

Recall that $\m_{ij}\in\Z^{d_{ij}+1}$ are the coefficients of the polynomial 
$P_{ij}(t)\in\Z[t]$ of degree $d_{ij}$, where $i=1,2,3$ and $j=1,\ldots,n_i$.
Let $\x_{ij}=(x_{i,j,0},\ldots,x_{i,j,d_{ij}})$ be variables and let
$P_{ij}(t,\x_{ij})=\sum_{k=0}^{d_{ij}} x_{ijk}t^k$ be the generic polynomial
of degree $d_{ij}$. Let $V$ be the open subscheme of 
$\A^{d+n+1}_\Z$ given by the condition $\prod_{i,j} P_{ij}(t,\x_{ij})\neq 0$. 
Let $U$ be the subscheme of
$\P^2_\Z\times\A^{d+n+1}_\Z$ given by (\ref{eq:defisko}) 
and $\prod_{i,j} P_{ij}(t,\x_{ij})\neq 0$. 
Assigning the value $\m_{ij}\in\Z^{d_{ij}+1}$ to the variable $\x_{ij}$ 
we obtain a conic bundle $U_\m\subset \P^2_\Z\times\A^1_\Z$ given by (\ref{eq:defisko}) 
together with the condition $\prod_{i,j} P_{ij}(t,\m_{ij})\neq 0$. 

Let $f:U\to V$ be the projection to the coordinates $t$ and $\x$. As in Section \ref{rat}
we denote by $g$ (respectively, by $h$) the projection to the coordinate $t$ (respectively, $\x$).

We follow the scheme of proof of Theorem \ref{thm1}. 
Let $S$ be the set of prime factors of $2a_1a_2a_3$.
The analogue
of Lemma \ref{erti} says that the fibre of the projective morphism $f:U\to V$
at any $\Z_p$-point of $V$ has a $\Q_p$-point when $p\notin S$. 
Indeed, this fibre is a conic with good reduction.

Since $f:U\to V$ is proper, the induced map $f:U(\Q_p)\to V(\Q_p)$ is topologically proper
\cite[p.~79]{Conrad}.
As $V(\Q_p)$ is locally compact and Hausdorff, $f:U(\Q_p)\to V(\Q_p)$ is a closed map.
We have $f(U(\Z_p))=f(U(\Q_p))\cap V(\Z_p)$, hence
$f(U(\Z_p))$ is closed in $V(\Z_p)$.
Since $V(\Z_p)$ is compact, $f(U(\Z_p))$ and $h(U(\Z_p))$ are compact too.
Thus $\prod_{p\in S}h(U(\Z_p))$ is compact.

Lemma \ref{ori} only uses the smoothness of 
$g\colon U_\Q\to\A^1_\Q$ and $h\colon U_\Q\to\A^{d+n}_\Q$, so it also holds in
our case. It implies that
for $p\in S$ and $N_p\in U(\Z_p)$ there is a positive integer $M_p$ such that if $\nu\in\Z_p$ and 
$\m\in(\Z_p)^{d+n}$ satisfy
\begin{equation}
\max\big( |\nu-g(N_p)|_p, |\m-h(N_p)|_p\big)\leq p^{-M_p}, 
\label{openset-conic}
\end{equation}
then $U_{\nu,\m}(\Z_p)\neq\emptyset$.
Let $\mathcal B_{N_p}\subset \Z_p^{d+n}$ be the $p$-adic ball of radius $p^{-M_p}$ around
$h(N_p)$. The open sets $\prod_{p\in S}\mathcal B_{N_p}$, where $(N_p)\in \prod_{p\in S}U(\Z_p)$, 
cover $\prod_{p\in S}h(U(\Z_p))$. By compactness, finitely many such open sets 
cover $\prod_{p\in S}h(U(\Z_p))$. Hence
$\mathcal M=\cup_{i=1}^n \mathcal M_i$, where 
$\mathcal M_i=\mathcal M\cap \prod_{p\in S}\mathcal B_{N_p}$ for one of these finitely many choices
of $(N_p)\in \prod_{p\in S}U(\Z_p)$.
Thus it is enough to prove that for 100\%
of $\m\in\mathcal M_i$ we have $U_\m(\Q)\neq\emptyset$.

In the rest of proof we write $\mathcal M=\mathcal M_i$.
Write $n_p=g(N_p)$ and $\m_p=h(N_p)$, where $p\in S$. 
Note that $N_p\in U(\Z_p)$ implies $P_{ij}(n_p,\m_p)\in\Z_p^*$ for each $p\in S$.
Write $M=\prod_{p\in S} p^{M_p}$. By the Chinese remainder theorem
we can find $n_0\in\Z$ and $\m_0\in\Z^{d+1}$ such that $n_0\equiv n_p\bmod {p^{M_p}}$ and 
$\m_0\equiv \m_p\bmod {p^{M_p}}$ for each $p\in S$. Our new set 
$\mathcal M$ consists of all $\m\in\mathcal P$
such that $\m\equiv \m_0\bmod M$.
Since $P_{ij}(n_p,\m_p)\in\Z_p^*$ for each $p\in S$, 
we see that $P_{ij}(n_0,\m_0)$ is coprime to $M$. 

We now apply Proposition \ref{sunset} to our $n_0$ and $M$, with $Q_{ij}(t)=P_{ij}(t,\m_0)$
for all $i$ and $j$. 
This is legitimate because $P_{ij}(n_0,\m_0)$ is coprime to $M$ and for any integer 
$\nu\equiv n_0\bmod M$ and any $\m\equiv \m_0\bmod M$ we have 
$U_{\nu,\m}(\Z_p)\neq\emptyset$ whenever $p\in S$.
Thus for 100\% of $\m\in\mathcal M$ we have $U_\m(\Q)\neq\emptyset$.

The last statement of Theorem \ref{thm3} is proved in the same way as in Theorems \ref{thm1}
and \ref{thm2}. 

 \subsection{The proof of Theorem \ref{TTT}}
\label{ppp}
We can ensure that $a_1$, $a_2$, $a_3$ are not all of the same sign
by replacing $P_{1,1}(x)$ by $-P_{1,1}(x)$, if necessary. 
We can also ensure that $a_1a_2a_3$ is square-free.
(If $p$ is a prime such that $p^2|a_1$, we absorb $p$ into $x$; if $p|a_1$ and $p|a_2$, then 
we multiply (\ref{eq:defisko}) by $p$ and absorb $p$ into $x$ and $y$.)
It remains to apply Theorem \ref{thm3}.

\section{Explicit probabilities}
 \label{s:newapp} 
In this section we obtain an explicit estimate for the probability that 
random affine Ch\^atelet surfaces have integer points, following the method of Theorem \ref{easy}.
We prove that this probability exceeds $56\%$ 
for  a family that has attracted much attention in the literature, namely,  
\beq
{eq:bachlute456}
{
x^2+y^2 =f(t)
,}
where $f$ is a polynomial of fixed degree  $d$ with positive
leading coefficient. V.A.~Iskovskikh \cite{isk} gave a first counter-example to the Hasse principle
with $d=4$; the density of such counterexamples was studied in \cite{bretim}  and \cite{MR3976470}. 
Little is known about the arithmetic of~\eqref{eq:bachlute456} when $d>6$ and $f(t)$ is irreducible.
Let
\[
P_d(H):=\{f\in \Z[t]: \deg(d)=d, |f|\leq H, \text{the leading coefficient of $f$ is positive} \}
.\]
\begin{theorem}
\label{thm:bachnight2}
For all $d\geq 2$, $ \epsilon>0$ and all sufficiently large $H   $  we have 
\[
\frac{\#\{f\in P_d(H): x^2+y^2=f(t) \text{\ is soluble in } \Z\}}{\#P_d(H)}
\geq(1 -\epsilon)
\frac{\l(38+\mathds 1(d\geq 3 )  \r) }{64}
\prod_{p\geq 3 } 
\l(1-\frac{1}{p^{\min\{p,d+1\} }}\r) 
. \] 
\end{theorem}  
The infinite product  
is a strictly increasing function of $d$. 
For $d=2$ it equals  $
0.95\ldots$ 
and as $d\to \infty $ the limit of the  product is $ \prod_{  p\geq 3    } ( 1- p^{-p } ) = 0.962 \ldots \ . $
\begin{corollary} For every  $d\geq 2 $ and all 
sufficiently large $H$   we have 
\[
\frac{\#\{f\in P_d(H): x^2+y^2=f(t)\text{\ is soluble in } \Z \}}{\#P_d(H)}
> \frac{56 }{100 } 
.\] 
\end{corollary} 
 To prove Theorem~\ref{thm:bachnight2}
we apply Theorem~\ref{cor:almostal}
with $n=1$, $M=4$, $n_0 \in \{0,1,2,3\}$
and arbitrary  $Q_1(t)$ of degree at most $d$ such that $Q_1(n_0)$ is 1 modulo 4.
It shows that for $100\%$ of Bouniakowsky polynomials 
$f(t)$  of degree $d$
such that $f(n_0)$ is 1 modulo 4,
there exists an integer $m$ such that $f(m)$ is a prime congruent to 1 modulo 4.
In this case~\eqref{eq:bachlute456}
has an integer solution. Thus, for all 
 $\epsilon>0$ and  all sufficiently large $H$ we have 
\[
\frac{\#\{f\in P_d(H): x^2+y^2=f(t) \text{\ is soluble in } \Z \}}{\#P_d(H)}
\geq
R_d(H)
 -\epsilon
,\]
where 
\[
R_d(H)
:=
\frac{\#\{f\in P_d(H): f \text{\,is Bouniakowsky}, \ \exists \   n_0 \in \{0,1,2,3\}  \text{\ such that } f(n_0)\equiv 1 \md{4} \}}{\#P_d(H)}
.\]
It is therefore sufficient to show  that $\lim_{H\to \infty} R_d(H)$ exists and find its value.
For this we partition the coefficients of   $f$  according to their values modulo $4 $ as follows:
$$R_d(H)\#P_d(H)=
\sum_{\substack{Q\in (\Z/4\Z)[t],  \deg(Q)\leq d \\ \exists  n_0 \in \Z/4\Z:\  Q(n_0)\equiv 1 \md 4 }}
\#\{f\in P_d(H): f \equiv Q \md 4, \, Z_f(p)\neq p, \ \forall p\ge 3 \}
.$$ By Corollary~\ref{densitySch01234} with $M=4 $ and the fact that $\#P_d(H)$ is asymptotic to $2^d H^{d+1}$ 
we obtain  \[
\lim_{H\to \infty} R_d(H)=r_d
\prod_{p\geq 3 } \l(1-\frac{1}{p^{\min\{p,d+1\} }}\r)
,
\] where 
\[
r_d
:=
\frac{1}{4^{d+1} }\#\{Q\in (\Z/4\Z)[t]  : \deg(Q)\leq d,   \exists \   n_0 \in \{0,1,2,3\} \text{\ such that } Q(n_0)\equiv 1 \md{4}  \}
.\]
A straightforward listing shows that $r_2=19/32$.
For the remaining case $d\geq 3 $ we 
write $f(t)=\sum_{i=0}^d c_i t^i$, thus  
\[1-r_d=
\frac{1}{4^{d+1} }
\sum_{(v_0, v_1,v_2,v_3 ) \in \{0,2,3\}^4 }
\#\l\{\b c \in (\Z/4\Z)^{d+1}: \sum_{i=0}^d c_i j^i \equiv v_j \md{4}, \ \forall j=0,1,2,3   \r\}
.\]
The system of 
four equations corresponding to $ j=0,1,2,3$ is equivalent to 
\[
c_0\equiv v_0 \md{4},
2 c_1  \equiv v_2-v_0 \md{4},
\sum_{0\leq i \leq d  }    c_i    \equiv v_1 \md{4},
2 \sum_{0\leq   i \leq d/2  }   c_{2i}   \equiv   v_1+v_3  \md{4}
.\]
This system has at least four unknowns $c_i$ due to $d\geq 3 $.
It is soluble if and only if both 
$v_0\equiv v_2 \md{2}$ and $v_1\equiv v_3 \md{2}$
hold; this happens for exactly $25$ vectors $(v_i)\in \{0,2,3\}^4$.
For each of these vectors, the first equation
determines $c_0$ uniquely and the second equation gives two values of $c_1$.
For any such $c_0,c_1$ and any $c_4, c_5, \ldots, c_d$ 
the last equation gives two values of $c_2$.
The   third equation determines  $c_3$ uniquely. Thus we obtain
\[1-r_d=
\frac{1}{4^{d+1} }
\times   25 \times 
 (1
\times  2 
\times  1 
\times 2 
\times 4^{d+1-4}
)=\frac{25}{64}
. \]

\end{document}